\documentclass[12pt]{amsart}

\usepackage{mystyle}

\begin{document}

\title[Visual Regularity Theory for Elliptic Equations]{A Visual Approach to the Regularity Theory for Fully Nonlinear Elliptic Equations}

\author[H. A. Chang-Lara]{H\'ector A. Chang-Lara}
\address{Department of Mathematics, CIMAT, Guanajuato, Mexico}
\email{hector.chang@cimat.mx}

\begin{abstract}
In this set of notes, we revisit the regularity theory for uniformly elliptic equations without divergence structure.
\end{abstract}

\maketitle

\tableofcontents

\section{Introduction}

The analysis of partial differential equations (PDEs) is a well-established discipline that treasures beautiful geometric insights. Although these concepts often come to life during seminar talks or discussions with experts in the field, students approaching these topics individually may encounter difficulties in extracting geometric meaning from the typical computations found in textbooks and research articles.

In these notes, our aim is to provide a pedagogical perspective that places a strong emphasis on the visual approach. Our journey begins with the fundamental case, that of the \textit{Laplacian}, and progressively leads to the \textit{Krylov-Safonov regularity theory}. Along the way, we engage the reader with a series of challenging problems that provide a valuable opportunity for hands-on learning. These problems are far from easy; many of them were actually taken from results in contemporary research articles; however, we expect that the techniques presented give some clues on their solutions. Toward the end, we offer hints for each one of these problems and provide references for those seeking a more comprehensive treatment. Finally, we also offer some discussions on some further developments and open problems in the field.

The main idea we hope to illustrate is how \textit{uniform ellipticity} enforces regularity on the solutions of elliptic equations through \textit{measure estimates}. The main example of this phenomenon is the \textit{mean value property} for \textit{harmonic functions}. In the next section, we will see how it implies an \textit{interior Hölder estimate} on any harmonic function, depending exclusively on its \textit{oscillation}.

To appreciate the significance of this result, consider an arbitrary uniformly bounded sequence of functions defined over a compact set. In contrast to numerical sequences, a bounded sequence of functions does not necessarily possess a convergent sub-sequence, even when each function within the sequence is smooth. However, as soon as we assume that the functions are harmonic, such limits are always guaranteed in the sense of uniform convergence by the Arzelà–Ascoli theorem and the Hölder continuity estimate for harmonic functions. This compactness plays an important role in numerous fundamental results in analysis of elliptic PDEs, such as existence theorems or the convergence of numerical schemes.

The analysis of PDEs is heavily based on establishing bounds for the modulus of continuity of a solution and its derivatives. These estimates serve as the cornerstones of what is known as the \textit{regularity theory}. 

The development of the regularity theory for elliptic equations has a rich history with numerous contributing authors. These notes delves into the regularity theory for uniformly elliptic equations in non-divergence form, a field that originated during 1980s and 1990s together with the theory of \textit{viscosity solutions}. This period witnessed significant contributions from notable figures such as Krylov, Safonov, Evans, Trudinger, and Caffarelli, among many others. Collectively, their work has given rise to what is now widely recognized as the Krylov-Safonov theory. For a detailed discussion on the history of the subject, we recommend \cite{MR1617413,MR1655532}, Chapter 1 in \cite{MR3243534}, and the notes at the end of Chapter 9 in \cite{MR737190}.

In this work, we do not provide a discussion of the rich variety of problems that motivate the study of elliptic equations. For this, we recommend the recent book \cite{MR4784613}, where a wide array of intriguing real-world problems, ranging from heat conduction, population dynamics, and random walks, are explored in depth. The book \cite{MR2309862} provides a visual approach that discusses applications for a wider range of PDEs.

The subjects we present in this article have been covered in various books \cite{MR737190,MR1351007,MR1406091,MR2777537,MR4560756}, notes \cite{mooney}, and articles \cite{MR1135923,MR1447056,MR2244602,MR2334822,MR3500837,MR3295593}.

These notes originated from a lecture offered by the author at the Learning Seminar on Analysis of PDEs for the Center for Mathematics at the University of Coimbra in July 2021. They are dedicated to Luis Caffarelli with deep gratitude and admiration.

\textbf{Acknowledgment:} The author was supported by CONACyT-MEXICO grant A1-S-48577.

\section{The Laplacian}

In this section, we show how the most basic form of uniform ellipticity leads to interior Hölder estimates. The main result is Theorem \ref{thm:int_hold_est}.

\subsection{Preliminaries}

In the following results, we use the oscillation of a function to measure how much it varies in a given set. For $u\colon \W\ss\R^n \to \R$ and  $E\ss\W$,
\[
\osc_E u:= \sup_{x,y\in E} |u(y)-u(x)| = \sup_E u - \inf_E u.
\]
For a fixed $x_0\in \W$, the oscillation of $u$ over\footnote{We will constantly use the following notation for open balls $B_r(x_0) := \{x\in \R^n \ | \ |x-x_0|<r\}$. By default $B_r := B_r(0)$.} $B_r(x_0)\cap \W$ can be used to determine the continuity of $u$ around $x_0$. In fact, $u$ is continuous at $x_0$ if and only if $\lim_{r\to 0^+}\osc_{B_r(x_0)\cap \W} u = 0$. Moreover, $u$ is uniformly continuous if and only if $\lim_{r\to0^+}\sup_{x_0\in\W} \osc_{B_r(x_0)\cap \W} u = 0$.

Any continuous non-decreasing function $\w\colon[0,\8)\to [0,\8)$ such that $\w(0)=0$ is said to be a \textit{modulus of continuity} for $u$ at $x_0$ if and only if $\w(r) \geq \osc_{B_r(x_0)\cap \W} u$. If we also have $\w(r) \geq \sup_{x_0\in\W} \osc_{B_r(x_0)} u$, then we just said that it is a modulus of continuity for $u$. The idea is that $\w$ quantifies the continuity of $u$, either punctually or globally.

\begin{figure}
    \centering
    \includegraphics[width=0.8\textwidth]{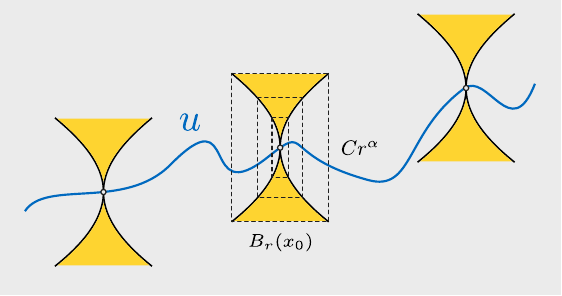}
    \caption{An $\a$-Hölder continuous function is one whose graph can be squeezed at every point $(x_0,u(x_0))$ between the graphs of $u(x_0) - C|x-x_0|^\a$ and $u(x_0) + C|x-x_0|^\a$ for some $C>0$. Equivalently, $\osc_{B_r(x_0)\cap \W} u \simeq \sup_{B_r(x_0)\cap \W} |u-u(x_0)|\leq Cr^\a$.}
    \label{fig:mod_cont}
\end{figure}

One important example of a modulus of continuity is the \textit{Hölder modulus}, it takes the form $\w(r) := Cr^\a$ for some $C>0$ and $\a\in(0,1]$, see Figure \ref{fig:mod_cont}. One of the most important features of this modulus of continuity is the scaling symmetry that it possesses, which leads to the \textit{diminish of oscillation} characterization.

\begin{lemma}[Characterization of Hölder continuity by diminish of oscillation]
\label{lem:dim_osc0}
Consider $\r,\theta\in(0,1)$, and $u\colon B_1 \to \R$ a bounded function. If $u$ satisfies the following diminish of oscillation identity for every radius $r \in(0,1]$
\[
\osc_{B_{\r r}} u\leq (1-\theta)\osc_{B_r} u,
\]
then $u$ admits the following Hölder modulus of continuity at the origin
\[
\osc_{B_r} u \leq Cr^\a\osc_{B_1} u,
\]
where $\a := \ln(1-\theta)/\ln \r$, $C := (1-\theta)^{-1}$, and $r \in(0,1]$.
\end{lemma}

We will usually find that $\theta$ and $\r$ are small positive constants. In particular, $1-\theta \geq \r$ would make $\a \in(0,1]$.

\begin{proof}
By an inductive argument we have that for any integer $k\geq 0$
\[
\osc_{B_{\r^k}} u \leq (1-\theta)^k\osc_{B_1} u.
\]
From this identity we can see why $\a$ was set to be $\ln(1-\theta)/\ln \r$. It is the exponent for which we have that $(1-\theta)^k = (\r^k)^\a$, for any $k$.

Let $r\in(0,1]$ and $k=\lfloor \ln r/\ln \r\rfloor\geq 0$ such that $\r^{k+1} < r \leq \r^k$. Then
\[
\osc_{B_r} u \leq \osc_{B_{\r^{k}}} u \leq (1-\theta)^{k}\osc_{B_1} u = C\r^{\a(k+1)}\osc_{B_1} u \leq Cr^\a\osc_{B_1} u.
\]
This concludes the proof.
\end{proof}

Given $\a\in(0,1]$, we measure the \textit{$\a$-Hölder semi-norm} of $u\colon \W\to \R$ over $E\ss\W$ by
\[
[u]_{C^{0,\a}(E)} := \sup_{x,y\in E} \frac{|u(y)-u(x)|}{|y-x|^\a} = \sup_{\substack{x_0\in E\\r>0}} r^{-\a}\osc_{B_r(x_0)\cap E} u.
\]
For $\b\in \R$, we denote the \textit{weighted Hölder semi-norm} in $\W$ as
\[
[u]_{C^{0,\a}(\W)}^{(\b)} := \sup_{B_r(x_0)\ss\W} r^\b[u]_{C^{0,\a}(B_{r/2}(x_0))}.
\]
Notice that $[u]_{C^{0,\a}(\W)}^{(\b)}$ controls the $\a$-Hölder modulus of $u$ at $x_0 \in \W$ in terms of the distance from $x_0$ to the boundary of $\W$: For $d:= \dist(x_0,\p\W)$,
\[
\sup_{r \in(0,d/2)} r^{-\a} \osc_{B_r(x_0)} u \leq [u]_{C^{0,\a}(B_{d/2}(x_0))} \leq d^{-\b}[u]_{C^{0,\a}(\W)}^{(\b)}.
\]

One could also consider for $\mu\in(0,1)$ a similar construction for $[u]_{C^{0,\a}(\W)}^{(\b)}$ replacing $B_{r/2}(x_0)$ with $B_{\mu r}(x_0)$. The following lemma shows that this quantity is comparable with the original one.

\begin{lemma}\label{lem:holder_weight}
Given $\a\in(0,1]$ and $\mu\in(0,1)$, there exists $C\geq 1$ such that for $u\in C(\W)$
\[
C^{-1}[u]_{C^{0,\a}(\W)}^{(\a)} \leq \sup_{B_\r(x)\ss\W} \r^\a[u]_{C^{0,\a}(B_{\mu\r}(x))} \leq C[u]_{C^{0,\a}(\W)}^{(\a)} .
\]
\end{lemma}

\begin{proof}
    Let $B_r(x_0)\ss \W$. We will show that for some $C\geq 1$, depending only on $\a$ and $\m$,
    \[
    r^\a[u]_{C^{0,\a}(B_{r/2}(x_0))} \leq C\sup_{B_\r(x)\ss\W} \r^\a[u]_{C^{0,\a}(B_{\m\r}(x))}.
    \]
    The proof for the lower bound can be obtained in a similar way and is omitted.
    
    For $y\in B_{r/2}(x_0)$ we have $B_{r/2}(y)\ss B_{r}(x_0)\ss\W$ and then
    \[
    r^\a[u]_{C^{0,\a}(B_{\mu r/2}(y))} = 2^\a (r/2)^\a[u]_{C^{0,\a}(B_{\mu r/2}(y))} \leq 2^\a\sup_{B_{\r}(x)\ss\W} \r^\a[u]_{C^{0,\a}(B_{\mu\r}(x))}.
    \]

    Let $N := \lfloor 2/\mu\rfloor +1$. Given $y_0,y_1\in B_{r/2}(x_0)$, consider $y_t := (1-t)y_0+ty_1$ so that
    \begin{align*}
    r^\a|u(y_0)-u(y_1)|
    &\leq r^\a\sum_{i=1}^N |u(y_{(i-1)/N})-u(y_{i/N})|\\
    &\leq \sum_{i=1}^N r^\a[u]_{C^{0,\a}(B_{\m r/2}(y_{(i-1/2)/N}))}\\
    &\leq 2^\a\sup_{B_{\r}(x)\ss\W} \r^\a[u]_{C^{0,\a}(B_{\m\r}(x))}\sum_{i=1}^N |y_{(i-1)/N}-y_{i/N}|^\a\\
    &\leq 2^\a N^{1-\a}\sup_{B_{\r}(x)\ss\W} \r^\a[u]_{C^{0,\a}(B_{\m\r}(x))}|y_0-y_1|^\a.
    \end{align*}
    We conclude after dividing by $|y_0-y_1|^\a$ and taking the supremum in $y_0,y_1\in B_{r/2}(x_0)$.  
\end{proof}

We will also consider the $L^p$-norms\footnote{For a Lebesgue measurable set $E\ss\R^d$, we denote its measure by $|E|$.}
\begin{align*}
    &\|u\|_{L^p(E)} := \1\int_E |u|^p\2^{1/p}, \qquad (p\in[1,\8)),\\
    &\|u\|_{L^\8(E)} := \inf\{M >0 \ | \ |\{x \in E \ | \ |u(x)|>M\}|=0\}.
\end{align*}
as well as its weighted version with exponent $\b\in \R$ 
\[\|u\|_{L^p(\W)}^{(\b)} :=  \sup_{B_r(x_0)\ss\W} r^\b\|u\|_{L^p(B_{r/2}(x_0))}.
\]

Finally, let us show an interpolation lemma.

\begin{lemma}\label{lem:inter}
Given $\a\in(0,1]$ and $\d\in(0,1/2)$ there exists $C\geq 1$ such that for $u \in C^1(B_1)$
\[
\|Du\|_{L^\8(B_{1/2})} \leq C\osc_{B_1}u + \d^\a[Du]_{C^\a(B_1)}.
\]
\end{lemma}

\begin{proof}
Given $x_0 \in B_{1/2}$ let
\[
\varphi(x) := \inf_{B_{\d}(x_0)} u + \d^{-2}\osc_{B_1}u\1\d^2-|x-x_0|^2\2.
\]
Let us see that there exists some $y_0 \in \argmin_{B_{\d}(x_0)}(u-\varphi)$, see Figure \ref{fig:inter}.

\begin{figure}
\centering
\includegraphics[width=0.8\textwidth]{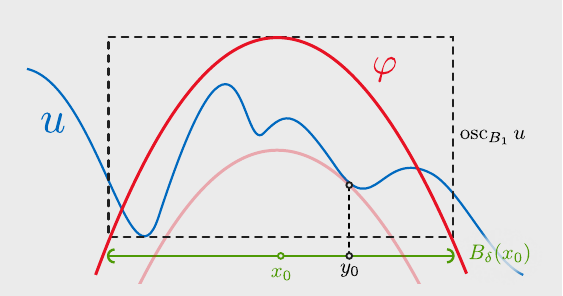}
\caption{The following construction could be considered a multidimensional version of Rolle's theorem. The paraboloid $\varphi$ satisfies that $\varphi(x_0)\geq u(x_0)$ and $\varphi\leq u$ on $\p B_{\d}(x_0)$. Either $\varphi$ touches $u$ at the center $x_0$ or $u$ must cross the graph of $\varphi$ in the interior of the ball $B_{\d}(x_0)$. Then a downwards translation of $\varphi$ becomes tangent with $u$ at some point $y_0 \in B_{\d}(x_0)$. At such contact point the gradient of $u$ equals the gradient of $\varphi$, which is bounded by $2\d^{-1}\osc_{B_1} u$.}
\label{fig:inter}
\end{figure}  

Notice that $m := \min_{\overline{B_{\d}(x_0)}}(u-\varphi)\leq 0$ because $\varphi(x_0)\geq u(x_0)$. If $m = 0$ we just take $y_0=x_0$. Otherwise, if $m<0$, then the non-empty set $\argmin_{\overline{B_{\d}(x_0)}}(u-\varphi)$ must be disjoint from $\p B_{\d}(x_0)$ where $\varphi \leq u$.

By applying the first derivative test at $y_0 \in \argmin_{B_{\d}(x_0)}(u-\varphi)$ we obtain that
\[
|Du(x_0)| \leq |Du(y_0)| + |Du(x_0)-Du(y_0)|\leq 2\d^{-1}\osc_{\W}u + \d^\a  [Du]_{C^{0,\a}(B_{1})}.
\]
The conclusion now follows by taking the supremum over $x_0\in B_{1/2}$.
\end{proof}

\subsection{Interior Hölder Estimate for the Laplacian}

We now have the ingredients to state the main regularity theorem of this section.

\begin{theorem}[Interior Hölder estimate - Laplacian]
    \label{thm:int_hold_est}
    Given $p\in[1,\8]\cap(n/2,\8]$, there exist $\a\in(0,1)$ and $C\geq 1$ such that for $\W\ss\R^n$ open, and $u\in C^2(\W)$ we have that
    \[
    [u]_{C^{0,\a}(\W)}^{(\a)} \leq C\1\osc_{\W}u + \|\D u\|_{L^p(\W)}^{(2-n/p)}\2.
    \]
\end{theorem}

Due to the definitions of the weighted semi-norms, it suffices to establish a Hölder estimate in balls contained in $\W$. Taking an appropriate system of coordinates, we can assume that the solution is defined over $B_1$, and the goal will be to establish an estimate on the Hölder modulus of continuity at the origin. Having in mind Lemma \ref{lem:dim_osc0}, we focus on showing that the oscillation of $u$ has a geometric decay in a sequence of concentric balls with geometrically decreasing radii. The mean value property plays a crucial role in this approach.

\begin{lemma}[Mean value property]
    \label{thm:mvf}
    Given $p\in[1,\8]\cap(n/2,\8]$, there exists $C\geq 1$ such that for $u\in C^2(B_1)$ we have that\footnote{Given $a\in \R$, we denote its positive and negative parts respectively as $a_+ = \max\{a,0\}$ and $a_-=\max\{-a,0\}$, such that $a=a_+-a_-$.}
    \[
    \fint_{B_1} u \leq u(0) + C\|(\D u)_+\|_{L^p(B_1)}.
    \]
\end{lemma}

\begin{figure}
    \centering
    \includegraphics[width=0.8\textwidth]{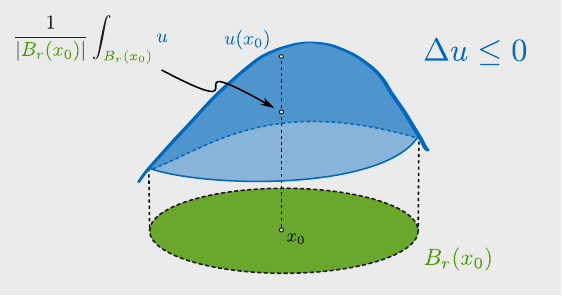}
    \caption{Given $\D u \leq 0$ in $\W$, we get that for any ball $B_r(x_0)$ contained in $\W$, the average of $u$ over $B_r(x_0)$ is less than the value at the center.}
    \label{fig:mvf}
\end{figure}

\begin{proof}
    Consider for $r\in(0,1)$, the average of $u$ over $\p B_r$. By the change of variable formula
    \[
    \fint_{\p B_r} u = \fint_{\p B_1} u(rx)dS(x).
    \]
    Then we compute its derivative with respect to $r$
    \[
    \frac{d}{dr} \fint_{\p B_r} u = \fint_{\p B_1} Du(rx)\cdot x dS(x) = \frac{r}{n}\fint_{B_1} \D u(rx) dx=\frac{r}{n}\fint_{B_r} \D u.
    \]
    The second equality follows by the divergence theorem. Given that $p\geq1$, Hölder's inequality allow us to get
    \[
    \frac{d}{dr} \fint_{\p B_r} u \leq Cr^{1-n/p}\|(\D u)_+\|_{L^p(B_1)}.
    \]

    Notice that the function $r\mapsto r^{1-n/p}$ is integrable in the interval $(0,1)$ if and only if $p>n/2$. By integrating the inequality from $r=0$ to $r=\r\in(0,1)$
    \[
    \fint_{\p B_\r} u \leq u(0) + C\r^{2-n/p}\|(\D u)_+\|_{L^p(B_1)}.
    \]
    We finally recover the desired formula by integrating from $\r=0$ to $\r=1$.
\end{proof}

The previous result can be extended to in any ball $B_r(x_0)$, namely
\[
\fint_{B_r(x_0)} u \leq u(x_0) + Cr^{2-n/p}\|(\D u)_+\|_{L^p(B_r(x_0))}.
\]
This can be shown by applying the lemma to $v(x) := u(rx+x_0)$.

The homogeneous case, $(\D u)_+ = 0$ or equivalently $\D u\leq 0$, gives a particular case of Jensen's inequality for concave functions, as depicted in Figure \ref{fig:mvf}. The functions that satisfy the inequality $\D u\leq 0$ are called \textit{super-harmonic}, and those that satisfy $\D u\geq 0$ are called \textit{sub-harmonic}. Although the naming convention might initially appear counterintuitive given the direction of the inequalities, it is actually based on the geometric characteristics of these functions. Specifically, concave functions are super-harmonic and their graphs curve \textit{upwards}. Conversely, convex functions are sub-harmonic functions and their graphs curve \textit{downwards}. See \hyperlink{ex:cvx}{Problem 2} for some further analogies.

If we apply Lemma \ref{thm:mvf} to a non-negative function defined in $B_1$, and over any possible ball $B_{2/3}(x_0)$ with $x_0\in B_{1/3}$ we deduce the following result. Figure \ref{fig:weak_har} illustrates the idea in the super-harmonic case. The general case can be recovered with a similar argument.

\begin{figure}[t]
    \centering
    \includegraphics[width=0.8\textwidth]{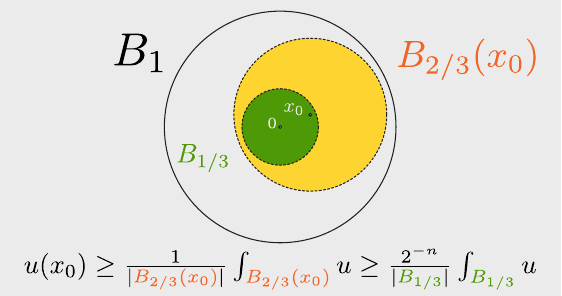}
    \caption{For any $x_0\in B_{1/3}$, we have the inclusions $B_{1/3}\ss B_{2/3}(x_0)\ss B_1$. We used that $u\geq 0$ in the second inequality.}
    \label{fig:weak_har}
\end{figure}

\begin{lemma}[Weak Harnack inequality - Laplacian]\label{cor:wharnack}
    Given $p\in[1,\8]\cap(n/2,\8]$, there exists $C\geq 1$ such that for $u\in C^2(B_1)$ non-negative
    \[
    \fint_{B_{1/3}}u \leq C\1\inf_{B_{1/3}} u+\|(\D u)_+\|_{L^p(B_1)}\2.
    \]
\end{lemma}

The following lemma is the key estimate that establishes a Hölder modulus of continuity by iterating the given estimate over balls of radii $3^{-k}$.

\begin{lemma}[Diminish of oscillation - Laplacian]
    \label{lem:dim_osc}
    Given $p\in[1,\8]\cap(n/2,\8]$, there exist $\d,\theta\in(0,1)$, such that for $u\in C^2(B_1)$ with $\|\D u\|_{L^p(B_1)}\leq \d$ we have that
    \[
    \osc_{B_1} u\leq 1 \qquad\Rightarrow\qquad \osc_{B_{1/3}} u \leq (1-\theta).
    \]
\end{lemma}

The hypothesis $\|\D u\|_{L^p(B_1)}\leq \d$, for $\d>0$ sufficiently small, can be considered as a perturbation of the Laplace equation $\D u = 0$ in $B_1$.

\begin{proof}
    Assume, without loss of generality, that $0\leq u \leq 1$. We now consider two possible alternatives: either $\fint_{B_{1/3}}u$ is $\geq 1/2$ or $\leq 1/2$. In the former case, we get by the weak Harnack inequality that $\inf_{B_{1/3}} u \geq 1/(2C) - \d$. Otherwise, we apply the same reasoning to $(1-u)$ to conclude that $\sup_{B_{1/3}} u \leq 1-(1/(2C) - \d)$. In any case, we get the desired diminish of oscillation if we choose $\d=\theta=1/(4C)$.
\end{proof}

Now the idea is to apply the previous lemma over a family of concentric balls in order to obtain a geometric decay of the oscillation as illustrated in Figure \ref{fig:dim_osc} (for some $\r\in (0,1/3]$).

\begin{lemma}[H\"older modulus of continuity - Laplacian]
\label{cor:dim_osc}
    Given $p\in[1,\8]\cap(n/2,\8]$, there exist $\a\in(0,1)$ and $C\geq 1$ such that for $u\in C^2(B_1)$
    \[
    \sup_{r\in(0,1)}r^{-\a}\osc_{B_r}u\leq C\1\osc_{B_1}u+\|\D u\|_{L^p(B_1)}\2.
    \]
\end{lemma}

\begin{figure}
    \centering
    \includegraphics[width=0.8\textwidth]{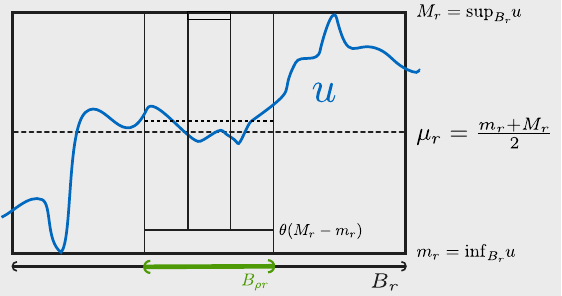}
    \caption{Diminish of oscillation: The lower bound of the solution in $B_{\r r}$ improves proportionally to the oscillation on $B_r$ if $\fint_{B_r} u\geq \m_r$. Otherwise the upper bound improves by the same amount. Then we repeat the argument from $B_{\r r}$ to $B_{\r^2 r}$ and so on.}
    \label{fig:dim_osc}
\end{figure}

\begin{proof}
    Our goal is to show that for some $\r,\theta \in (0,1)$ we see that $\osc_{B_{\r^k}} u \leq (1-\theta)^k$ holds true for any integer $k\geq 0$, so that Lemma \ref{lem:dim_osc0} settles the proof.
    
    Let $\d,\theta\in(0,1)$ as in Lemma \ref{lem:dim_osc}. Assume that $\osc_{B_1}u\leq 1$ and $\|\D u\|_{L^p(B_1)} \leq \d$, otherwise apply the following proof to the function $\bar u:= u/\1\osc_{B_1}u+\d^{-1}\|\D u\|_{L^p(B_1)}\2$ that satisfies the required conditions.
    
    Assume by induction that for $\r:= \min\{1/3,(1-\theta)^{1/(2-n/p)}\}$ and some $k\geq 1$, we have $\osc_{B_{\r^k}} u \leq (1-\theta)^k$. The base case $k=1$ is just the conclusion of Lemma \ref{lem:dim_osc}. Then consider $v(x) := (1-\theta)^{-k}u(\r^kx)$ such that
    \[
    \osc_{B_1} v =(1-\theta)^{-k}\osc_{B_{\r^k}} u \leq 1.
    \]
    Using that $\r \leq (1-\theta)^{1/(2-n/p)}$
    \[
    \|\D v\|_{L^p(B_1)} = (1-\theta)^{-k}\r^{k(2-n/p)}\|\D u\|_{L^p(B_{\r^k})} \leq \d.
    \]
    Then by Lemma \ref{lem:dim_osc}, $\osc_{B_\r} v \leq \osc_{B_{1/3}} v \leq (1-\theta)$. Equivalently,
    \[
    \osc_{B_{\r^{k+1}}} u = (1-\theta)^{k}\osc_{B_\r} v \leq (1-\theta)^{k+1},
    \]
    which now shows that the geometric decay is valid for every $k\geq 1$.
\end{proof}

The proof of Theorem \ref{thm:int_hold_est} now follows from Lemma \ref{cor:dim_osc} by an idea that is usually referred to in the literature as a \textit{covering argument}.

\begin{proof}[Proof of Theorem \ref{thm:int_hold_est}]
    Let $B_{r}(x_0)\ss\W$, $x,y \in B_{r/6}(x_0)$ distinct, and $\r=|x-y| \in (0,r/3)$. Notice that $B_{r/6}(x_0) \ss B_{r/3}(x) \ss B_{r/2}(x_0)\ss \W$. Consider $v(z) := u((r/3)z+x)$ such that as $z\in B_1$ we have that $(r/3)z+x \in B_{r/3}(x)$.
    
    By applying Lemma \ref{cor:dim_osc} to $v$
    \begin{align*}
    \frac{|u(y) - u(x)|}{|y-x|^\a} &\leq \r^{-\a}\osc_{B_{\r}(x)}u\\
    &= \r^{-\a}\osc_{B_{3\r/r}} v\\
    &\leq Cr^{-\a} \1\osc_{B_1} v + \|\D v\|_{L^p(B_1)}\2\\
    &= Cr^{-\a} \1\osc_{B_{r/3}(x)} u + r^{2-n/p}\|\D u\|_{L^p(B_{r/3}(x))}\2\\
    &\leq Cr^{-\a} \1\osc_{B_{r/2}(x_0)} u + r^{2-n/p}\|\D u\|_{L^p(B_{r/2}(x_0))}\2.
    \end{align*}
    By taking the supremum over $x,y \in B_{r/6}(x_0)$ and multiplying by $r^\a$
    \[
    r^\a[u]_{C^{0,\a}(B_{r/6}(x_0))} \leq C\1\osc_{B_{r/2}(x_0)} u + r^{2-n/p}\|\D u\|_{L^p(B_{r/2}(x_0))}\2.
    \]
    Finally, we conclude by taking the supremum in $B_r(x_0)$ and using Lemma \ref{lem:holder_weight}.
\end{proof}

\begin{remark}
    For $\a>0$, the function $u(x) := |\ln |x||^{-\a}$ over $B_{1/2}\ss\R^n$ is not Hölder continuous. We check that, 
    \begin{align*}
    |\D u(x)| &= \a|(\a+1)|\ln|x||^{-(\a+2)} - (n-2)|\ln|x||^{-(\a+1)}||x|^{-2} \leq  C|\ln |x||^{-(\a+1)}|x|^{-2}.
    \end{align*}
    Hence, we have $|\D u|\in L^{n/2}(B_{1/2})$ if $\a > \min\{2/n-1,0\}$
    \[
    \int_{B_{1/2}} |\D u|^{n/2} \leq C\int_0^{1/2} |\ln r|^{-(\a+1)n/2}\frac{dr}{r} \leq C\int_{\ln 2}^{\8} \ell^{-(\a+1)n/2}d\ell < \8. 
    \]
    This shows that the critical exponent $p=n/2$ cannot be reached in the Hölder estimate.
\end{remark}

\begin{remark}
    Notice how the Hölder estimate obtained from $\D u \in L^p(B_1)$ with $p\in[1,\8]\cap(n/2,\8]$ is related to the sequential application of Sobolev's embedding and Morrey's embedding. Specifically, the embeddings $W^{2,p}(B_1)\ss W^{1,q}(B_1) \ss C^{0,\a}(B_1)$ hold for $p \in (n/2,n)$, where $q = np/(n-p) > n$, and $\a = 1-n/q$. If instead $p\geq n$, we have that $W^{2,p}(B_1)\ss C^{0,\a}(B_1)$ for any $\a \in (0,1)$.
    
    The surprising fact is that, meanwhile the Sobolev/Morrey inequalities require to control each one of the coefficients in the Hessian, the Hölder estimate presented in Theorem \ref{thm:int_hold_est} achieves the same qualitative result with just a control on the Laplacian. This is remarkable given that the Laplacian, being the trace of the Hessian, retains only partial information on the Hessian.
\end{remark}

\begin{corollary}[Interior Higher Order Estimates - Laplacian]\label{cor:high_ord_est}
Given $p\in[1,\8]\cap(n/2,\8]$ and a non-negative integer $k$, there exist $\a\in(0,1)$ and $C\geq 1$, such that for $\W\ss\R^n$ open, and $u\in C^{k+2}(\W)$ we have that
\[
[D^ku]^{(k+\a)}_{C^{0,\a}(\W)} \leq C\1\osc_{\W}u + \|D^k(\D u)\|_{L^p(\W)}^{(k+2-n/p)}\2.
\]
\end{corollary}

Let us present the proof for $k=1$, as the general case follows from a similar reasoning.
    
\begin{proof}
    For any directional derivative $v := \p_e u$, with respect to a unit direction $e\in \p B_1$, we have that $\D v = \p_e(\D u)$ in $\W$. By applying the proof of Theorem \ref{thm:int_hold_est} to $v$ we get that for any $B_r(x_0)\ss\W$
    \begin{align*}
    r^\a[v]_{C^{0,\a}(B_{r/12}(x_0))} &\leq C\1\osc_{B_{r/4}(x_0)}v + r^{2-n/p}\|\p_e(\D u)\|_{L^p(B_{r/4}(x_0))}\2\\
    &\leq C\1\|Du\|_{L^\8(B_{r/4}(x_0))}+r^{-1}\|D(\D u)\|_{L^p(\W)}^{(3-n/p)}\2.
    \end{align*}
    By taking the supremum over $e \in \p B_1$ on the left-hand side and multiplying by $r$
    \begin{align}\label{eq:6}
    r^{1+\a}[Du]_{C^{0,\a}(B_{r/12}(x_0))} \leq C\1r\|Du\|_{L^\8(B_{r/4}(x_0))}+\|D(\D u)\|_{L^p(\W)}^{(3-n/p)}\2.
    \end{align}
    
    By applying Lemma \ref{lem:inter} to $w(x) = u((r/2)x+x_0)$ with some $\d\in(0,1/2)$ to be fixed sufficiently small
    \begin{align*}
    r\|Du\|_{L^\8(B_{r/4}(x_0))} &= \|Dv\|_{L^\8(B_{1/2})}\\
    &\leq C\osc_{B_1}v + \d^\a[Dv]_{C^{0,\a}(B_1)}\\
    &\leq C\osc_{\W} u + \d^\a r^{1+\a} [Du]_{C^{0,\a}(B_{r/2}(x_0))}\\
    &\leq C\osc_{\W} u + \d^\a [Du]_{C^{0,\a}(\W)}^{(1+\a)}.
    \end{align*}
    Substituting this interpolation in \eqref{eq:6}
    \[
    r^{1+\a}[Du]_{C^{0,\a}(B_{r/12}(x_0))} \leq C\1\osc_{\W} u + \d^\a [Du]_{C^{0,\a}(\W)}^{(1+\a)} + \|D(\D u)\|^{(3-n/p)}_{L^p(\W)}\2.
    \]
    Taking the supremum over $B_r(x_0)\ss\W$ and using Lemma \ref{lem:holder_weight}
    \[
    [Du]^{(1+\a)}_{C^{0,\a}(\W)} \leq C\1\osc_{\W} u + \d^\a [Du]^{(1+\a)}_{C^{0,\a}(\W)} + \|D(\D u)\|^{(3-n/p)}_{L^p(\W)}\2.
    \]
    By finally choosing $\d$ sufficiently small, the term $\d^\a[Du]^{(1+\a)}_{C^{0,\a}(\W)}$ gets absorbed by the left-hand side of the inequality.
\end{proof}


\begin{exercise}[\hypertarget{ex:harnack}{\hyperlink{sol:harnack}{\textbf{Harnack's inequality.}}}]
  \item Show that for a non-negative harmonic function $u\in C^2(B_1)$ and $r\in(0,1/3)$,
\[
\sup_{B_r}u\leq \1\frac{1-r}{1-3r}\2^n \inf_{B_r} u.
\]
  \item Use this previous inequality to show a diminish of oscillation and interior Hölder's estimate for harmonic functions considering as alternatives whether $u(0)$ is smaller or greater than $(\inf_{B_{1/4}}u + \sup_{B_{1/4}}u)/2$.

\end{exercise}

\begin{remark}
    Other proofs of the Harnack inequality include: The one by Serrin in \cite{MR0081415} using barriers in two dimensions, also presented in \cite[Theorem 3.10]{MR737190}. The one by Li and Yau in \cite{MR0834612} using the Bernstein method, also presented in \cite[Lemma 1.32]{MR2777537}.
\end{remark}

\begin{exercise}[\hypertarget{ex:cvx}{\hyperlink{sol:cvx}{\textbf{Concave/Convex functions.}}} Many properties for sub/super-harmonic functions have easier counterparts for convex/concave functions. Recall that $u\colon \W\to \R$ is convex if and only if $\{(x,y) \in \W\times\R \ | \ y\geq u(x)\}$ is a convex set. It is concave if instead $-u$ is convex.]

\item (Harnack) For $u\in C(B_1)$ non-negative and concave, and $r\in (0,1)$
\[
\sup_{B_{1}}u\leq \frac{2}{1-r}\inf_{B_r}u.
\]

\item (Local maximum principle) For every $u\in C(B_1)\cap L^1(B_{1})$ non-negative and convex, and $r\in(0,1)$
\[
\sup_{B_r}u \leq \frac{2^n}{|B_1|(1-r)^n} \|u\|_{L^1(B_{1}\sm B_r)}.
\]
\end{exercise}

\begin{exercise}[\hypertarget{ex:weyl}{\hyperlink{sol:weyl}{\textbf{Weyl's lemma.}}} From Lemma \ref{thm:mvf} we get that a harmonic function $u$ must satisfy the mean value property
\[
u(x_0) = \fint_{B_r(x_0)}u.
\]
On the other hand, this identity is rigid enough to imply smoothness and harmonicity. We say that $u \in L^1(\W)$ is harmonic in a weak sense if and only if the mean value property holds for every ball $B_r(x_0)\ss\W$.]
\item Let $\eta \in C^\8_c(B_1)$ non-negative, radially symmetric, and such that $\int_{B_1} \eta = 1$; and let $\eta_\e(x) := \e^{-n}\eta(\e^{-1}x)$. Show that if $u \in L^1(\W)$ is weakly harmonic, then
\[
u = \eta_\e\ast u \text{ in } \W_\e := \{x\in \W  \ | \ B_\e(x)\ss\W\}.
\]
This implies that $u$ is smooth over $\W_\e$ for every $\e>0$.

\item Show that there exists $C\geq 1$ such that for every $u\in C^2(B_\e)$
\[
\lim_{\r\to0^+} \r^{-2}\fint_{B_\r} (u-u(0)) = C\D u(0).
\]
This shows that if $u$ is weakly harmonic, then it is also be harmonic in the classical sense.
\end{exercise}

\begin{exercise}[\hypertarget{ex:max_prin}{\hyperlink{sol:max_prin}{\textbf{Maximum principle.}}}]

\item Show that for every $p\in[1,\8]\cap(n/2,\8]$, there exists $C>0$ such that for $u\in C(\overline{B_1})$ with $u\in C^2(\{u>0\}\cap B_1)$ and $v = \max\{u,0\}$ we have that
\begin{align}
    \label{eq:7}
    v(0) \leq \fint_{\p B_1} v + C\|(\D u)_-\|_{L^p(\{u>0\}\cap B_1)}.
\end{align}

\item Show that for every $p\in[1,\8]\cap(n/2,\8]$, there exists $C>0$ such that for $\W\ss B_R$ and $u\in C^2(\W)\cap C(\R^n)$ with $u=0$ on $\R^n\sm \W$
\[
    \|u\|_{L^\8(\W)} \leq CR^{2-n/p}\|\D u\|_{L^p(\W)}.
\]

\end{exercise}

\begin{exercise}[\hypertarget{ex:bdry_reg}{\hyperlink{sol:bdry_reg}{\textbf{Boundary regularity.}}}]

\item Show that for every $p\in[1,\8]\cap(n/2,\8]$ and $\eta \in (0,1)$ there exist $\a\in(0,1)$ and $C\geq 1$ such that the following holds: Let $\W\ss \R^n$ such that
\[
\inf_{r\in(0,1)}\frac{|(\R^n \sm \W)\cap B_r|}{|B_r|} \geq \eta.
\]
Then for any $u\in C^2(\W)\cap C(B_1)$ with $u = 0$ in $B_1\sm \W$ and $\D u \in L^p(\W)$, we have that
\[
\sup_{r\in(0,1)} r^{-\a}\osc_{B_r} u \leq C\1\osc_{B_1} u + \|\D u\|_{L^p(\W)}\2.
\]

\item
Show that for every $p\in[1,\8]\cap(n/2,\8]$, $\eta \in(0,1)$, and $R>0$ there exist $\a\in(0,1)$ and $C\geq 1$ such such that the following holds: Let $\W\ss B_R\ss\R^n$ be such that
\[
\inf_{\substack{x_0 \in\p\W\\r\in(0,1)}} \frac{|(\R^n \sm \W)\cap B_r(x_0)|}{|B_r|} \geq \eta.
\]
Then for any $u\in C^2(\W)\cap C(\R^n)$ with $u = 0$ in $\R^n\sm \W$ and $\D u \in L^p(\W)$, we have that
\[
[u]_{C^{0,\a}(\W)} \leq C\|\D u\|_{L^p(\W)}.
\]
\end{exercise}

\begin{exercise}[\hypertarget{ex:lmp}{\hyperlink{sol:lmp}{\textbf{Local maximum principle.}}}]

\item Show that for every $p\in[1,\8]\cap(n/2,\8]$ and $\e>0$, there exist $\theta\in(0,1)$ and $M\geq 1$ such that the following holds: For $u\in C^2(B_1)$ with $\|u_+\|_{L^\e(B_1)} \leq 1$, and $\|(\D u)_-\|_{L^p(B_1)} \leq 1$, we have that
\[
    u(0)\geq M \qquad\Rightarrow\qquad \sup_{B_1}u\geq (1+\theta)u(0).
\]
\item Show that for every $p\in[1,\8]\cap(n/2,\8]$ and $\e>0$, there exists some $C\geq 1$ such that the following maximum principle holds for $u\in C^2(B_1)$
\[
\sup_{B_{1/2}} u_+ \leq C\1\|u_+\|_{L^\e(B_1)} + \|(\D u)_-\|_{L^p(B_1)}\2.
\]
This implies that the interior Hölder estimate can be strengthened to\footnote{For $\e\in(0,1)$, $\|\cdot \|_{L^\e}$ is not a norm because it is not sub-additive, however it is sub-additive up to a multiplicative constant.}
\[
[u]_{C^{0,\a}(\W)}^{(\a)} \leq C\1\|u\|_{L^\e(\W)}^{(-n/\e)} + \|\D u\|_{L^p(\W)}^{(2-n/p)}\2.
\]
\end{exercise}

\begin{exercise}[\hypertarget{ex:he}{\hyperlink{sol:he}{\textbf{Heat equation.}}} Let $H_t(x) := (4\pi t)^{-n/2}e^{-|x|^2/(4t)}$ be the heat kernel. For these problems it may be convenient to know that $u =u(x,t)\colon\R^n\times(-\8,0\rbrack\to \R$ defined from $f = f_t(x) \in C_c(\R^n\times(-\8,0\rbrack)$ by the global representation formula
\[
u(x,t) := \int_{-\8}^t (H_{t-s}\ast f_s)(x)ds,
\]
is the unique bounded solution of the heat equation $\p_t u - \D u = f$ in $\R^n\times(0,-\8\rbrack$, with $u(x,-T)=0$ for any $T>0$ such that $\spt f \ss \R^n\times(-T,0\rbrack$.]

\item (Mean value property) Show that there exist some non-negative kernels $K_1,K_2 \in C_c(B_1\times (-1,0])$, such that for any solution $u\in (C^2_x\cap C^1_t)(B_1\times (-1,0])$ of the heat equation $\p_t u - \D u = f$ in $B_1\times(0,1]$ it holds that
\[
u(0,0) = \iint_{B_1\times(-1,0]} (uK_1 + fK_2).
\]

\item (Weak Harnack) Show that for every $p,q \in[1,\8]$ such that $2>n/p+2/q$, there exist a radius $\r\in (0,1/4)$ and a constant $C\geq 1$ such that for any non-negative solution $u\in (C^2_x\cap C^1_t)(B_1\times (-1,0])$ of $\p_t u - \D u \geq f$ in $B_1\times(-1,0]$ it holds that
\[
\iint_{Q_-}u \leq C\1\inf_{Q_+} u + \|f_-\|_{L^q_t((-1,0]\to L^p_x(B_1))}\2,
\]
where $Q_-:=B_{\r}\times(-4\r,-3\r]$ and $Q_+:=B_{\r}\times(-\r,0]$. Similarly to the stationary problem, this estimate yields an interior Hölder estimate for the solution of a heat equation of the form
\[
\sup_{r \in(0,1)}r^{-\a}\osc_{B_r\times(-r^2,0]}u \leq C\1\osc_{B_1\times(-1,0]}u+\|f\|_{L^q_t((-1,0]\to L^p_x(B_1))}\2.
\]
\end{exercise}



\section{Uniformly Elliptic Equations}

In this section we present the Krylov-Safonov theory for non-linear uniformly elliptic equations. The main result is Theorem \ref{thm:ihe}.

\subsection{Preliminaries}

Let $\W\ss\R^n$ be open and $F=F(M,p,z,x)\in C(\R^{n\times n}_{\text{sym}}\times\R^n\times\R\times \W)$. This function defines a second-order nonlinear PDE of the form:
\[
F(D^2u(x), Du(x),u(x),x)= 0 \text{ for } x\in\W.
\]
To simplify the notation we may just write the expression as
\[
F(D^2u,Du,u,x)=0\text{ in $\W$}.
\]

In analogy with the definition of super-harmonic functions, we call super-solutions those functions that satisfy the inequality $F(D^2u, Du,u,x)\leq 0$. Similarly, sub-solutions are those that satisfy the inequality $F(D^2u, Du,u,x)\geq 0$.

\subsubsection{Linearizations}\label{sec:lin}

A common strategy to analyze non-linear equations consists of linearizing the problem at hand. In the second-order setting this lead us to a problem of the form\footnote{We assume the following inner product in $\R^{n\times n}$
\[
A:B := \tr(A^TB) = \sum_{i,j=1}^n a_{ij}b_{ij}.
\]}
\[
Lv := A:D^2v + b\cdot Dv + cv = f.
\]
Occasionally, we will use $L = A:M+b\cdot p+cz$ to refer to a linear operator as above.

The coefficients given by the \textit{diffusion matrix} $A$, the \textit{drift vector} $b$, the \textit{scalar} $c$, and the \textit{forcing term} $f$, are generally functions that depend on the spatial variable $x$. In such instances, we describe the operator as having variable coefficients. If $f=0$ we say that the equation is homogeneous. The simplest non-trivial case $A=I$, $b=0$ and $c=0$ makes $L=\D$.

We will consider three ways to obtain this linearization, depending on whether we look for estimates on $u$, the gradient of $u$, or the Hessian of $u$. In the following, we assume that $F\in C^1(\R^{n\times n}_{\text{sym}}\times\R^n\times\R\times \W)$ is differentiable and $u \in C^2(\W)$ satisfies
\[
F(D^2u,Du,u,x) = 0 \text{ in } \W.
\]

\noindent\textit{Linearization for the Original Function:} The first way follows from the fundamental theorem of calculus. By integrating the derivative $(d/dt)F(tD^2u,tDu,tu,x)$ from $t=0$ to $t=1$,
\begin{align*}
-F(0,0,0,x) &= F(D^2u,Du,u,x)-F(0,0,0,x)\\
&= \int_0^1 \frac{d}{dt}F(tD^2u,tDu,tu,x)dt\\
&= \sum_{i,j=1}^n\1\int_0^1\p_{m_{ij}}Fdt\2 \p_{ij}u + \sum_{i=1}^n\1\int_0^1 \p_{p_i}Fdt\2 \p_iu + \1\int_0^1 \p_zFdt\2u.
\end{align*}
We obtain in this way that $v_0=u$ satisfies the linear equation $Lv_0=f$, with coefficients given by
\begin{align*}
    A(x) &:= \int_0^1D_MF(tD^2u(x),tDu(x),tu(x),x)dt,\\
    b(x) &:= \int_0^1D_pF(tD^2u(x),tDu(x),tu(x),x)dt,\\
    c(x) &:= \int_0^1\p_zF(tD^2u(x),tDu(x),tu(x),x)dt,\\
    f(x) &:= -F(0,0,0,x).
\end{align*}
The functions $D_MF = (\p_{m_{ij}}F) \in \R^{n\times n}_{\text{sym}}$, $D_pF = (\p_{p_1}F,\ldots,\p_{p_n}F) \in \R^n$, and $\p_zF\in \R$ denote the partial derivatives of $F$ with respect to the Hessian variable $M=(m_{ij})\in \R^{n\times n}_{\text{sym}}$, the gradient variable $p=(p_1,\cdots,p_n)\in \R^n$, and $z\in \R$, respectively.

\noindent\textit{Linearization for a First Derivative:} A second way to obtain a similar linearization consists of taking a derivative of the equation in some unit direction $e\in \p B_1$. This leads to the following linear equation for $v_1 := \p_e u$
\[
D_M F:D^2v_1 + D_p F\cdot Dv_1 + \p_z F v_1 = - e \cdot D_xF.
\]
In other words, $v_1$ satisfies $Lv_1=f$ for the following coefficients
\begin{align*}
    A(x) &:= D_M F(D^2u(x),Du(x),u(x),x),\\
    b(x) &:= D_p F(D^2u(x),Du(x),u(x),x),\\
    c(x) &:= \p_z F(D^2u(x),Du(x),u(x),x)\p_e u(x),\\
    f(x) &:= - e \cdot D_xF(D^2u(x),Du(x),u(x),x).
\end{align*}

\noindent\textit{Linearization for a Second Derivative:} We can also obtain a linear equation for a pure second derivative $v_2 = \p_e^2u$ by differentiating the equation for $u$ twice, or equivalently the equation for $v_1 = \p_e u$ once. Keep in mind that information about mixed derivatives can be recovered from the polarization identity
\[
2\p_{e'}\p_e = \p_{e+e'}^2 - \p_e^2 - \p_{e'}^2.
\]

The derivative of the leading order term of the equation for $v_1$ is
\[
\p_e\1\sum_{i,j=1}^n \p_{m_{ij}}F\p_{ij}v_1\2 = \sum_{i,j=1}^n \p_{m_{ij}}F\p_{ij}v_2 + \sum_{i,j,k,l=1}^n (\p_{m_{ij}}\p_{m_{kl}}F)\p_{ij}v_1\p_{kl}v_1.
\]
The second term has a sign if we assume that $F$ is concave or convex in $M$.

We can repeat the same procedure to obtain linear equations for higher-order derivatives of $u$. However, once we have established Hölder estimates for $v_2$, higher-order estimates for $u$ become readily available thanks to the Schauder theory, not covered in this note and for which we recommend \cite[Chapter 2]{MR4560756}. This technique is known as \textit{elliptic bootstrapping}, before describing it we must understand the most important hypotheses, \textit{ellipticity} and \textit{uniform ellipticity}.

\subsubsection{Ellipticity and Uniform Ellipticity}

In the previous linearizations, the coefficients depend on the solution $u$ as well. However, under suitable hypotheses on the derivatives of $F$ and $u$, we can overlook this dependence and understand the equation in a general sense. For example, it may be the case that the coefficients are all measurable and bounded.

A common assumption in many applications is the ellipticity of the operator. This appears as a monotonicity property given by\footnote{We assume the following ordering in $\R^{n\times n}_{\text{sym}}$: $M\leq N$ if and only if $x\cdot Mx \leq x\cdot Nx$ for every $x \in \R^n$. Equivalently, the eigenvalues of $(N-M)$ are non-negative.} $D_MF\geq 0$. Uniform ellipticity requires instead a quantitative form of monotonicity with respect to the Hessian variable
\begin{align}
    \label{eq:ellipd}
    \l I \leq D_M F \leq \L I,
\end{align}
for some parameters $[\l,\L]\ss(0,\8)$.

We can also extend the notions of ellipticity and uniform ellipticity to non-differentiable operators. Ellipticity means that $F$ is monotone in the Hessian variable, which means that for $M,N\in \R^{n\times n}_{\text{sym}}$ and $(p,z,x)\in \R^n\times\R\times \W$ we have
\[
M\leq N\qquad\Rightarrow\qquad F(M,p,z,x) \leq F(N,p,z,x).
\]
The main consequence of this property is the \textit{comparison principle}.

\begin{figure}
    \centering
    \includegraphics[width=0.8\textwidth]{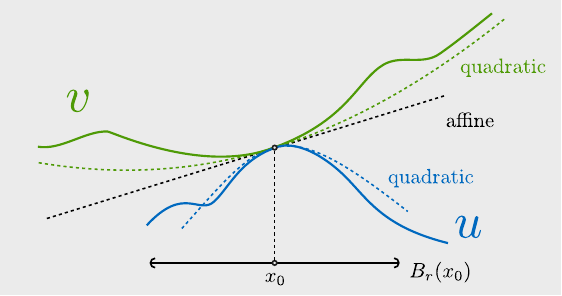}
    \caption{The graph of $u$ touches the graph of $v$ from below. At the contact point $x_0$ we obtain that $u(x_0)=v(x_0)$, $Du(x_0)=Dv(x_0)$ (first derivative test), and $D^2u(x_0)\leq D^2v(x_0)$ (second derivative test).}
    \label{fig:comp_prin}
\end{figure}

\noindent\textit{The Comparison Principle:} Given two functions $u,v\colon B_r(x_0)\to \R$, we say that $u$ touches $v$ from below at $x_0\in \W$ if $u\leq v$ in $B_r(x_0)$, and $u(x_0)=v(x_0)$ (Figure \ref{fig:comp_prin}). Whenever the functions are differentiable at $x_0$, the first derivative test implies that their gradients at the contact point are the same. If we additionally assume that the functions have second-order derivatives at $x_0$, then the second derivative test says that their Hessians are ordered, $D^2u(x_0)\leq D^2v(x_0)$. For $F$ elliptic it then must happen that
\[
F(D^2u(x_0),Du(x_0),u(x_0),x_0) \leq F(D^2v(x_0),Dv(x_0),v(x_0),x_0).
\]

The comparison principle motivates the notion of \textit{viscosity solutions}, a theory of weak solutions developed during the 1980s for first-order and second-order problems and widely used nowadays. Although they will not be used in these notes, it is useful to describe the general construction and keep it in mind for the next section about the contact set for a whole family of test functions (Section \ref{sec:abp}).

In the case of \textit{viscosity super-solutions}, this notion declares that $u \in C(\W)$ satisfies
\[
F(D^2u,Du,u,x)\leq 0 \text{ in the viscosity sense $\W$},
\]
by checking that for any $B_r(x_0)\ss\W$, and any \textit{test function} $\varphi\in C^2(B_r(x_0))$ touching $u$ from below at $x_0$, it holds that at the contact point
\[
F(D^2\varphi(x_0),D\varphi(x_0),\varphi(x_0),x_0)\leq 0.
\]
The ellipticity of $F$ means that this criterion is automatically satisfied if $u \in C^2(\W)$ is a super-solution in the classical sense.

\textit{Viscosity sub-solutions} require instead a test function $\varphi$ touching $u$ from above at some $x_0\in \W$, and checking that $F(D^2\varphi(x_0),D\varphi(x_0),\varphi(x_0),x_0)\geq 0$. The equality in the viscosity sense holds when the function is simultaneously a viscosity super-solution and a sub-solution. More details on this theory can be found in \cite{MR1118699,MR1351007,calder}.

The comparison principle allows us to compare sub and super-solutions. Consider $\W\ss\R^n$ be open and bounded, and $F \in C(\R^{n\times n}_{\text{sym}}\times\R^n\times\R\times\W)$ be elliptic and non-increasing in the $z$ variable, this last property is known as \textit{properness}. Given $u,v \in C^2(\W)\cap C(\overline{\W})$ we get that
\[
\begin{cases}
    F(D^2u,Du,u,x) > F(D^2v,Dv,v,x) \text{ in }\W,\\
    u\leq v \text{ in } \p\W,
\end{cases} \qquad\Rightarrow\qquad u< v \text{ in } \W.
\]

\begin{figure}
    \centering
    \includegraphics[width=0.8\textwidth]{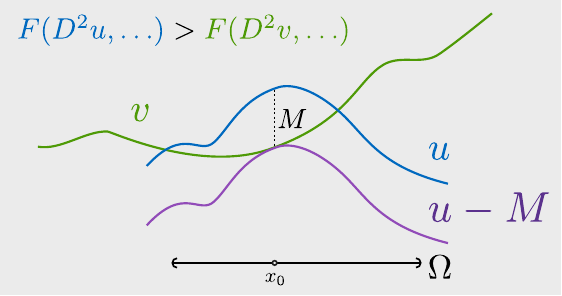}
    \caption{If the graphs of $u$ and $v$ cross or touch each other, we get that some downwards translation $u-M$ touches $v$ from below at some $x_0\in \W$. Given that $M\geq 0$ and $F$ is elliptic and proper, this leads to the contradiction $F(D^2u(x_0),Du(x_0),u(x_0),x_0) \leq  F(D^2u(x_0),Du(x_0),u(x_0)-M,x_0) \leq F(D^2v(x_0),Dv(x_0),v(x_0),x_0)$.}
    \label{fig:comp_prin2}
\end{figure}

In contrast to the direct reasoning for classical sub/super-solutions (Figure \ref{fig:comp_prin2}), the comparison principle for viscosity solutions is a highly non-trivial result that took several years to be fully established.

\noindent\textit{The Extremal Pucci Operators:} Let $F\in C^1(\R^{n\times n}_{\text{sym}}\times\R^n\times\R\times\W)$ be uniformly elliptic. By the fundamental theorem of calculus
\[
F(M,p,z,x)- F(N,p,z,x) = \1\int_0^1 D_MF(N+t(M-N),p,z,x)dt\2:(M-N)
\]
Using that $\l I \leq D_MF\leq \L I$, we deduce the identity
\begin{align}
\label{eq:ellip}
    \min_{\l I\leq A\leq \L I} A:(M-N) \leq F(M,p,z,x)-F(N,p,z,x)\leq \max_{\l I\leq A\leq \L I} A:(M-N).
\end{align}
This property is actually equivalent to $\l I \leq D_MF\leq \L I$ and does not requires differentiability on $F$. Therefore, it can be used to extend the notion of uniform ellipticity when $F$ is not differentiable.

The operators that appear on the left and right-hand side of \eqref{eq:ellip} are the extremal Pucci operators and can be computed as the following combinations in terms of the eigenvalues of the symmetric matrix $M$
\begin{align*}
\mathcal M^-_{\l,\L}(M) &:= \min_{\l I\leq A\leq \L I} A:M = \sum_{e\in \operatorname{eig}(M)} (\l e_+ - \L e_-),\\
\mathcal M^+_{\l,\L}(M) &:= \max_{\l I\leq A\leq \L I} A:M = \sum_{e\in \operatorname{eig}(M)} (\L e_+ - \l e_-).
\end{align*}

\begin{remark}
    \label{rmk:lin}
    The Pucci operators provide an indirect way to address equations of the form $A:D^2u=f$ without an explicit characterization of the coefficients given by the matrix-valued function $A$, in addition to uniform ellipticity. Indeed, for $f\in C(\W)$ we have that $u\in C^2(\W)$ satisfies
    \[
    \mathcal M^-_{\l,\L}(D^2u)\leq f \leq \mathcal M^+_{\l,\L}(D^2u) \text{ in $\W$}
    \]
    if and only if there exists $A\colon\W\to\R^{n\times n}_{\text{sym}}$ such that $\l I\leq A\leq \L I$ and $A:D^2u=f$ in $\W$.
\end{remark}

\begin{figure}
    \centering
    \includegraphics[width=0.8\textwidth]{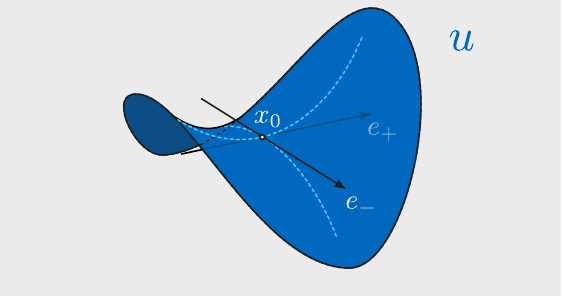}
    \caption{For the homogeneous equation $\mathcal M_{\l,\L}^-(D^2u)\leq 0 \leq \mathcal M_{\l,\L}^+(D^2u)$ we get that the graph of $u$ has a saddle at every point. Moreover, the sum of the positive eigenvalues of the Hessian and the sum of the negative eigenvalues of the Hessian are comparable: $(\l/\L)\sum_{e\in \operatorname{eig}(D^2u)} e_- \leq \sum_{e\in \operatorname{eig}(D^2u)} e_+ \leq  (\L/\l) \sum_{e\in \operatorname{eig}(D^2u)} e_-$.}
    \label{fig:eig_ue}
\end{figure}

The homogeneous equation $\mathcal M_{\l,\L}^-(D^2u)\leq 0$ states that the positive eigenvalues of the Hessian are controlled by the negative ones. Meanwhile, equation $\mathcal M_{\l,\L}^+(D^2u)\geq 0$ states that the negative eigenvalues are controlled by the positive ones (Figure \ref{fig:eig_ue}).

\begin{remark}\label{rmk:lin3}
Let $f\colon\W\to[0,\8)$. The equations
\begin{align}\label{eq:11}
\begin{cases}
\mathcal M^-_{\l,\L}(D^2u) \leq f \text{ in } \W,\\
\mathcal M^+_{\l,\L}(D^2u) \geq -f \text{ in } \W,
\end{cases}
\end{align}
are equivalent to $|A:D^2u| \leq f$, for some $A\colon\W\to\R^{n\times n}_{\text{sym}}$ with $\l I\leq A\leq \L I$. In particular, the equations appearing in Remark \ref{rmk:lin} imply \eqref{eq:11}

To justify this claim, consider $A^\pm\colon\W\to\R^{n\times n}_{\text{sym}}$ such that $A^\pm:D^2u = \mathcal M^\pm_{\l,\L}(D^2u)$. Then \eqref{eq:11} means that the intervals $[-f,f]$ and $[A^-:D^2u,A^+:D^2u]$ must intercept at least at one point. For some $t\in[0,1]$ we see that $A = (1-t)A^-+tA^+$ satisfies $|A:D^2u| \leq f$.
\end{remark}

\begin{remark}
\label{rmk:lin2}
For a uniformly elliptic operator $F=F(M,x)$, not necessarily differentiable, we can ``linearize'' the equation $F(D^2u,x)=0$ using the Pucci operators instead of the fundamental theorem of calculus. In fact, if we let $f(x) := -F(0,x)$, then any solution of $F(D^2u,x)=0$ satisfies
\[
\mathcal M^-_{\l,\L}(D^2u)\leq \underbrace{F(D^2u,x) - F(0,x)}_{=f} \leq \mathcal M^+_{\l,\L}(D^2u).
\]
We already pointed out that these inequalities are equivalent to $A:D^2u=f$ for some $A\colon\W\to\R^{n\times n}_{\text{sym}}$ such that $\l I\leq A\leq \L I$.
\end{remark}

If $F = F(M)$ is also independent of $x$, then we can find a linearization for a function approximating a directional derivative $v_1 = \p_e u$ in the following way:
\[
v_{1,h} := \frac{u_{h,e}-u}{h}, \qquad u_{h,e}(x) := u(x+he), \qquad h>0, \qquad e\in \p B_1.
\]
We notice that if $u$ is a solution of $F(D^2u)=0$ in $\W$, then also the translation $u_{h,e}$ is a solution of the same equation in $\W-he$. At the intersection we get that
\[
\mathcal M^-_{\l,\L}(D^2v_{1,h}) \leq \underbrace{\frac{F(D^2u_{h,e})-F(D^2u)}{h}}_{=0} \leq \mathcal M^+_{\l,\L}(D^2v_{1,h}). 
\]
If we assume that $u \in C^3(\W)$, then as we send $h\to 0^+$ we get $\mathcal M^-_{\l,\L}(D^2v_1) \leq 0 \leq \mathcal M^+_{\l,\L}(D^2v_1)$.


Finally, if $F=F(M)$ is uniformly elliptic and convex, we get a super-solution inequality for
\[
v_{2,h} := \frac{u_{h,e}+u_{-h,e}-2u}{h^2} \sim v_2 = \p_e^2u.
\]
By uniform ellipticity
\begin{align*}
\frac{h^2}{2}\mathcal M^-_{\l,\L}\1D^2v_{2,h}\2 &= \mathcal M^-_{\l,\L}\1\frac{D^2u_{h,e}+D^2u_{-h,e}}{2}-D^2u\2\\
&\leq F\1\frac{D^2u_{h,e}+D^2u_{-h,e}}{2}\2-F(D^2u)\\
&= F\1\frac{D^2u_{h,e}+D^2u_{-h,e}}{2}\2.
\end{align*}
By convexity and translation invariance
\[
F\1\frac{D^2u_{h,e}+D^2u_{-h,e}}{2}\2 \leq \frac{F(D^2u_{h,e})+F(D^2u_{-h,e})}{2} = 0.
\]
We conclude in this way that if $u\in C^4(\W)$, then $\mathcal M^-_{\l,\L}(D^2v_2)\leq 0$.

Let us finish this discussion by illustrating the bootstrapping technique for higher-order estimates. Assume that $u$ is a $C^3$-regular solution of the uniformly elliptic equation $F(D^2u)=0$ with $F$ smooth. Then, any directional derivative $v_1 = \p_eu$ satisfies a linear uniformly elliptic equation with Lipschitz coefficients. Schauder estimates (\cite[Chapter 2]{MR4560756} or \cite[Chapter 6]{MR737190}) then imply that $v_1$ is actually $C^{2,\a}$-regular (for any $\a\in(0,1)$) and then $u$ must be $C^{3,\a}$-regular. This now implies that the coefficients of the equation satisfied by $v_1$ are now even better, $C^{1,\a}$-regular, which then implies that the solution $v_1$ must be $C^{3,\a}$-regular, and $u$ is now promoted to be $C^{4,\a}$-regular. This procedure can be inductively continued, showing that $u$ is smooth.


\subsubsection{Symmetries}

In our analysis for the Laplacian, we used some fundamental symmetries of the operator. These symmetries allowed us to establish an estimate in a specific (renormalized) configuration and then extend the result to translated and rescaled configurations. For example, the mean value property (Lemma \ref{thm:mvf}) was proved for the unit ball and then observed that it can be extended to arbitrary balls by a change of variables.

Let us give a few examples of the fundamental symmetries considered for second-order operators of the form $u\mapsto F(D^2u,Du,u,x)$. They have also been used in the previous considerations about linearizations.

\noindent\textit{Translation invariance:} Holds whenever $F = F(M,p,z)$ is independent of the $x$ variable. In this case, we find that if $f(x) := F(D^2u(x),Du(x),u(x))$ and $\bar u(x) := u(x+x_0)$ is a translation of $u$ for some $x_0\in \R^n$, then we can compute $F(D^2\bar u(x),D\bar u(x),\bar u(x))= f(x+x_0)$, with the same translation.

We also consider vertical translations of the graph of $u$ given by $\bar u(x) := u(x)+\theta$, for some $\theta\in \R$. In this case $F = F(M,p,x)$, independent of the $z$ variable, satisfies $F(D^2\bar u(x),D\bar u(x),x)=F(D^2u(x),Du(x),x)$. Whenever $F$ is proper, that is, nonincreasing in $z$, then we see that for $\theta\geq 0$, it holds
\[
F(D^2\bar u(x),D\bar u(x),\bar u(x),x) \leq F(D^2u(x),Du(x),u(x),x).
\]
This means that upward translations preserve super-solutions of proper equations. Similarly, downward translations preserve sub-solutions.

\noindent\textit{Homogeneity:} Holds if for any constant $c\geq 0$, we have $F(cM,cp,cz,x) = cF(M,p,z,x)$. In this case, we find that if $f(x) := F(D^2u(x),Du(x),u(x))$ and $\bar u(x) := cu(x)$ are a vertical scaling of $u$, then we can compute $F(D^2\bar u(x),D\bar u(x),\bar u(x))= cf(x)$, with the same scaling.

\noindent\textit{Dilation invariance:} Let $\a\in(0,1]$. The $\a$-H\"older continuity of a function $u\colon B_1\to \R$ at the origin can be characterized by the scaling transformations given by $u_r(x):= r^{-\a}u(rx)$, for $r>0$. We are interested in the dilations that arise when $r$ belongs to $(0,1)$, so the graph of $u_r$ consists in dilating the graph of $u$ horizontally with a factor $r^{-1}$, and vertically with a factor $r^{-\a}$ (Figure \ref{fig:scaling}).

\begin{figure}
    \centering
    \includegraphics[width=0.8\textwidth]{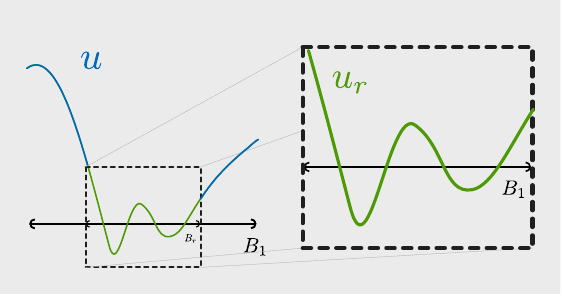}
    \caption{For $r\in (0,1)$, the scaling $u_r(x) = r^{-\a}u(rx)$ zooms-in the graph of $u$ by identifying the values of $u$ over the ball of radius $r$, to the values of $u_r$ over the ball of radius 1. In the figure above we illustrate the case when $\a=1$, which corresponds to a similar vertical stretching of the graph.}
    \label{fig:scaling}
\end{figure}

An operator $F=F(M,p,z,x)$, defined over $\W=\R^n$, is \textit{invariant by dilations} if for some exponents $\a,\b \in \R$ we have that
\[
F(r^{\a-2}M,r^{\a-1}p,r^\a z,rx) = r^\b F(M,p,z,x).
\]
In this case, if we let $f(x) := F(D^2u(x),Du(x),u(x),x)$ and $u_r(x) = r^{-\a}u(rx)$ a dilation of $u$, then we are also able to compute $F(D^2u_r(x),Du_r(x),u_r(x),x)=r^{-\b}f(rx)$ in terms of a scaling.

For example, the Laplacian is invariant by dilations for any pair of exponents $\a,\b \in \R$ such that $\b=\a-2$. This is the main consideration behind the change of variable arguments in the diminish of oscillation strategy. Notice also that any pure second-order linear operator with constant coefficients, as well as the Pucci extremal operators, is invariant by dilations, once again for $\b=\a-2$.

It is also useful to consider the notion of a whole family of operators preserved by some given symmetries. A family of operators $\mathcal F = \{F\}$ is invariant by dilations and in terms of the exponents $\a,\b \in \R$, if for any $F \in \mathcal F$ and $r\in(0,1)$, there exists $G \in \mathcal F$, such that $F(r^{\a-2}M,r^{\a-1}p,r^\a z,rx) = r^\b G(M,p,z,x)$.

This previous notion is useful for characterizing a family of linear operators $L = A:M+b\cdot p+cz$ with the hypothesis of uniform ellipticity $\l I\leq A\leq \L I$, and $\|b\|_{L^{q_1}(\R^n)},\|c\|_{L^{q_0}(\R^n)}\leq \L$, for $q_1\geq n$ and $q_0 \geq n/2$. In general, each of these operators is not invariant by dilations because each term has a different scaling invariance. However, the whole family is invariant by dilations.

All these symmetries will be combined in order to simplify some arguments in our proofs. The idea starts by showing a result in a normalized domain (usually the unit ball), with some normalized hypothesis (for example, oscillation bounded by one, or a sufficiently small forcing term for the equation). Then the hypotheses and conclusions can be extended to other configurations by applying a combination of translations and rescalings.

\subsection{Interior Hölder Estimates for Uniformly Elliptic Equations}

Here is the main result of this section and the fundamental estimate of the Krylov-Safonov regularity theory.

\begin{theorem}[Interior Hölder estimate - uniformly elliptic]
\label{thm:ihe}
Given $[\l,\L]\ss(0,\8)$, there exist $\a\in(0,1)$ and $C\geq 1$, such that the following holds: Let $\W\ss\R^n$ open and $F\in C(\R^{n\times n}_{\text{sym}}\times \W)$ be uniformly elliptic with respect to $[\l,\L]$. Given $u\in C^2(\W)$ and $f(x) := F(D^2u,x)-F(0,x)$ we have
\[
[u]^{(\a)}_{C^{0,\a}(\W)} \leq C\1\osc_{\W}u + \|f\|^{(1)}_{L^n(\W)}\2.
\]
\end{theorem}

Due to the definition of the weighted norms using balls $B_r(x_0)\ss\W$, and the considerations presented in Remark \ref{rmk:lin}, Remark \ref{rmk:lin3}, and Remark \ref{rmk:lin2}, we can drop the operator given by $F$ and assume instead that $u\in C^2(B_r(x_0))$ and $f\in L^n(B_r(x_0))$ non-negative satisfy
\[
\begin{cases}
\mathcal M^-_{\l,\L}(D^2u) \leq f \text{ in } B_r(x_0),\\
\mathcal M^+_{\l,\L}(D^2u) \geq -f \text{ in } B_r(x_0).
\end{cases}
\]

The symmetries of the Pucci operators allow us to extend the strategy for the Laplacian in this setting. Let us recapitulate the main lemmas:
\begin{itemize}
    \item Mean value formula: Lemma \ref{thm:mvf}.
    \item Weak Harnack inequality: Lemma \ref{cor:wharnack}.
    \item Diminish of oscillation: Lemma \ref{lem:dim_osc}.
    \item Hölder modulus of continuity: Lemma \ref{cor:dim_osc}.
\end{itemize}

For each of the steps above, we present the following analogous results:
\begin{itemize}
    \item Measure estimates: Section \ref{sec:abp2}.
    \item Weak Harnack inequality: Section \ref{sec:dim_dist} and Lemma \ref{lem:wharnack}.
    \item Diminish of oscillation: Lemma \ref{lem:dim_osc2}.
    \item Hölder modulus of continuity: Lemma \ref{thm:h}.
\end{itemize}

Here is the statement for last step in the list.

\begin{lemma}[Hölder modulus of continuity - uniformly elliptic]
\label{thm:h}
Given $[\l,\L]\ss(0,\8)$, there exist $\a\in(0,1)$ and $C\geq 1$, such that for $u\in C^2(B_1)$, and $f\in L^n(B_1)$ non-negative with
\[
\begin{cases}
\mathcal M^-_{\l,\L}(D^2u) \leq f \text{ in } B_1,\\
\mathcal M^+_{\l,\L}(D^2u) \geq -f \text{ in } B_1,
\end{cases}
\]
we have that
\[
\sup_{r\in(0,1)}r^{-\a}\osc_{B_r} u \leq C\1\osc_{B_1} u+\|f\|_{L^n(B_1)}\2.
\]
\end{lemma}

The second to last step in our list is given by the following lemma. This result will be the main technical goal of this section.

\begin{lemma}[Diminish of oscillation - uniformly elliptic]\label{lem:dim_osc2}
Given $[\l,\L]\ss(0,\8)$, there exist $\d,\r,\theta\in(0,1)$, such that for $u\in C^2(B_1)$, $f\in L^n(B_1)$ non-negative with $\|f\|_{L^n(B_1)}\leq \d$, and
\[
\begin{cases}
\mathcal M^-_{\l,\L}(D^2u)\leq f \text{ in } B_1,\\
\mathcal M^+_{\l,\L}(D^2u)\geq -f \text{ in } B_1,
\end{cases}
\]
we have that
\[
\osc_{B_1} u\leq 1 \qquad\Rightarrow\qquad \osc_{B_{\r}} u \leq (1-\theta).
\]
\end{lemma}

Lemma \ref{lem:dim_osc2} implies Lemma \ref{thm:h} in the same way that Lemma \ref{lem:dim_osc} implies Lemma \ref{cor:dim_osc}; using the scaling invariance of the equation and Lemma \ref{lem:dim_osc0} characterizing the Hölder modulus of continuity. As it was already noticed, the hypotheses in Theorem \ref{thm:ihe} imply those of Lemma \ref{thm:h} in an appropriated system of coordinates. The conclusion of Lemma \ref{thm:h} then implies the conclusion of Theorem \ref{thm:ihe} using the covering argument in the proof of Theorem \ref{thm:int_hold_est}

The rest of the strategy towards the proof of Lemma \ref{lem:dim_osc2} requires novel insights. Some technical steps, with proofs resembling those presented in the previous section for the Laplacian, will be postponed to Section \ref{sec:sum} in order to highlight the new ideas. We also encourage the reader to reconstruct those arguments as exercises.

\subsubsection{The Probabilistic Insight}

The original proof of Theorem \ref{thm:ihe}, due to Krylov and Safonov \cite{MR525227}, was originally presented in terms of stochastic processes. In this section, we will briefly describe this perspective which motivates the measure estimates in the coming sections.

Consider $u\in C^1(\R^n)$ satisfying the first-order equation $b \cdot Du = 0$ for some vector-valued function $b$. The curve $X=X(t)\in C^1([0,\8)\to \R^n)$ is \textit{characteristic} for the linear and homogeneous PDE if and only if $X'=b(X)$, in which case we obtain by the chain rule that $u(X(t))$ remains constant along the trajectory.

We can find a similar connection with second-order PDEs once we allow for random dynamics. The simplest case appears when $X(t)$ is a Brownian motion and $u$ is a harmonic function. By Itô's calculus we get that the expected value $\mathbb E[u(X(t))]$ is also constant\footnote{And more precisely, one can also say that $u(X(t))$ is a \textit{martingale}.}. We can use this property to compute $u(x)$ by considering a Brownian motion starting at $x$ up to the stopping time when it hits the region where the values of $u$ are prescribed.

Consider $\W\ss\R^n$ open and bounded and $g\in C(\R^n\sm\W)$. If we define $u(x)$ as the expected value of $g(X(\t))$ where $X(t)$ is a Brownian motion starting at $x$, and $\t$ is the exit time from $\W$, then $u$ satisfies $\D u=0$ in $\W$ with the complementary condition $u = g$ in $\R^n\sm \W$\footnote{The continuity of $u$ across the boundary $\p\W$, and the uniqueness of this characterization, depend on the regularity of the boundary $\p\W$; Lipschitz would be enough.}.

\begin{figure}
    \centering
    \includegraphics[width=0.8\textwidth]{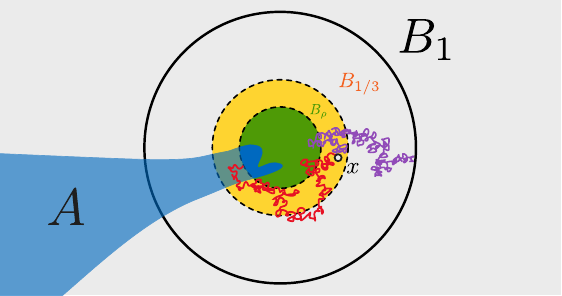}
    \caption{The Harnack inequality provides a lower bound for the probability that a Brownian motion starting at $x \in B_{1/3}$, hits $A\cap B_\r$ before exiting $B_1$. This bound depends linearly on the measure of $A\cap B_\r$, but not on the starting point or other geometric characteristics of the set $A$. The figure above illustrates two realizations of the Brownian motion.}
    \label{fig:prob}
\end{figure}

Consider $\r\in(0,1/3)$, a closed set $A$, and let $v_{\r,A}(x)$ be the probability that a Brownian motion starting at $x$ reaches the set $A\cap \overline{B_\r}$ before exiting $B_1$. Therefore
\[
\D v_{\r,A} = 0 \text{ in } B_1\sm (A\cap \overline{B_\r})
\]
with complementary values
\[
\begin{cases}
v=1 \text{ on } A\cap \overline{B_\r},\\
v=0 \text{ on } \R^n\sm B_1.
\end{cases}
\]
The probabilistic Harnack inequality (Figure \ref{fig:prob}) states that for some constant $c>0$ depending only on the dimension, it holds that
\[
\inf_{B_{1/3}} v_{\r,A} \geq c|A\cap B_\r|.
\]

In particular, if $u$ is a non-negative super-harmonic function and we let $A=\{u\geq 1\}$, we get by the comparison principle that $u\geq v_{\r,A}$ and then
\[
\inf_{B_{1/3}} u \geq c|\{u\geq 1\}\cap B_\r|.
\]

Notice that the same inequality can also be deduced from Lemma \ref{cor:wharnack} and the Markov inequality. Moreover, the proof of the diminish of oscillation (Lemma \ref{lem:dim_osc}) can also be achieved if we instead consider the alternatives whether $|\{u\geq 1/2\}\cap B_\r|$ is at least half the measure of $B_\r$, or is less than half the measure of the same ball.  

This probabilistic approach to the Harnack principle can actually be extended to general diffusions. Given that the identities now relate the distribution of $u$ with point-wise values, these inequalities are known as \textit{measure estimates}. They extend the mean value property for uniformly elliptic equations.

\subsubsection{The Geometric Insight}

Another significant insight into the theory was provided by the Minkowski problem. It consists of constructing a convex body with prescribed Gaussian curvature over its boundary. In the case where the boundary is parameterized by the graph of the convex function $u$ we find the PDE
\[
\det(D^2u) = K(x)(1+|Du|^2)^{(n+2)/2}.
\]
In this context, the leading term is referred to as the Monge-Ampère operator and has important connections with the field of optimal transport. Notice that it remains elliptic under the assumption of convexity; however, it lacks uniform ellipticity.

A key geometrical observation for these problems is the Aleksandrov lemma \cite{aleksandrov1960,MR0214900}. Similarly to the mean value property, it allows one to control the point-wise values of $u$ in terms of the integral of a second-order operator. Notice, however, that in this scenario the natural integration theorem is no longer the divergence theorem, but the change of variables formula.

\begin{lemma}
    There exists a constant $C\geq 1$ depending only on the dimension such that the following holds: Let $\W\ss\R^n$ be an open and convex set and $u\in C^2(\W) \cap C(\overline{\W})$ be a convex function with $u=0$ on $\p\W$, then
    \[
    \sup_{x\in \W}\frac{|u(x)|^n}{\dist(x,\p\W)} \leq C\diam(\W)^{n-1}\int_\W\det(D^2u)
    \]
\end{lemma}

We suggest that the reader attempt the proof of this lemma after reading this and the next section. See \hyperlink{ex:mp}{Problem 9} and its \hyperlink{sol:mp}{hints} at the end.

The techniques presented next could be considered refined versions of the convex geometry arguments developed by Aleksandrov, Bakelmann, and Pucci to establish a maximum principle for linear equations without divergence structure. We recommend the lecture \cite{MR2465040} for a complete survey on the subject.

\subsubsection{Touching Solutions with a Family of Test Functions}\label{sec:abp}

One of the main tools available for elliptic equations is the comparison principle. In the case of a super-solution $u$, a test function touching $u$ from below inherits the same equation as $u$ at the contact point. If we now consider a whole family of test functions, it is reasonable to expect that the corresponding contact set must also retain information from the equation.

To fix ideas, consider a family of test functions obtained by translations of a given profile
\[
\varphi_{y_0}(x) := \varphi_0(x-y_0).
\]
We will call $y_0$ the center of the function $\varphi_{y_0}$.

Given $u\in C(\W)$, consider the contact set as $y_0$ varies in some set $B$ as
\[
A := \bigcup_{y_0\in B} \argmin(u-\varphi_{y_0}). 
\]
(Figure \ref{fig:abp2}).

\begin{figure}
    \centering
    \includegraphics[width=0.8\textwidth]{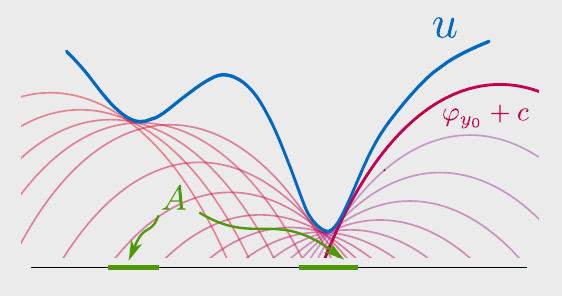}
    \caption{The contact set for a function $u$ is the set of points in the domain that admit a lower supporting graph of the form $\varphi_{y_0}+c$.}
    \label{fig:abp2}
\end{figure}

Here are three general principles about the contact set $A$:

\textbf{The set $A$ captures information about $u$:} For every $x_0 \in A$ one has $u(x_0) = \varphi_{y_0}(x_0)$, $Du(x_0) = D\varphi_{y_0}(x_0)$, and $D^2u(x_0) \geq \varphi_{y_0}(x_0)$. Depending on the choice of test functions, one can obtain different bounds on $u$, $Du$ or $D^2u$ over $A$.

For this previous information on $u$ to be significant, we need to ensure that $A$ is not too small.

\textbf{The measure of $B$ provides a lower bound on the measure of $A$:} For this we notice that if the test function $\varphi_{y_0}$ touches $u$ from below at $x_0$, then
\[
Du(x_0) = D\varphi_{y_0}(x_0) = D\varphi_0(x_0-y_0).
\]
If $D\varphi_0$ has an inverse, then one can compute the center of the test function touching $u$ from below at $x_0$ as $y_0 = T(x_0) := x_0 - D\varphi_{0}^{-1}(Du(x_0))$. If $T\colon A\to B$ is surjective and Lipschitz, we get by the area formula that
\[
|B| \leq \int_{A} |\det(DT)|.
\]
Notice that $DT$ ultimately depends on $D^2u$. Using the previous identity, we observe that $\int_{A} |\det(DT)|$ can be bounded by the measure of $A$ if there is some bound on $D^2u$ over $A$.

\textbf{Uniform ellipticity controls the eigenvalues of $D^2u$ over $A$:} The negative eigenvalues of $D^2u$ are bounded by the eigenvalues of $D^2\varphi_{y_0}$ due to the second derivative test. Then we can use $\mathcal M_{\l,\L}^-(D^2u) = \sum_{e\in \operatorname{eig}(D^2u)} \l e_+ - \L e_-$, to bound the positive eigenvalues of $D^2u(x_0)$ in terms of $\mathcal M_{\l,\L}^-(D^2u)$ and the already controlled negative eigenvalues.

\subsubsection{Measure Estimates (ABP-type Lemmas)}\label{sec:abp2}

The following lemma provides a concrete application of this idea when $\varphi_0(x) = -\frac{1}{2}|x|^2$ is a concave paraboloid, and therefore $T(x) = x + Du(x)$. This approach is know as \textit{the method of the sliding paraboloids} and was originally developed in \cite{MR1447056} and \cite{MR2334822}.

\begin{lemma}[Measure estimate]\label{lem:abp}
    Given $[\l,\L]\ss(0,\8)$, there exist $\d,\theta,\eta\in(0,1)$, such that the following holds: Let $u\in C^2(B_1)$, and $f \in L^n(B_1)$ with $\|f\|_{L^n(B_1)}\leq \d$, both non-negative functions such that $\mathcal M^-_{\l,\L}(D^2u)\leq f$. Then
    \[
    u(0) \leq \theta \qquad\Rightarrow\qquad |\{u\leq 1\}\cap B_1|\geq \eta.
    \]
\end{lemma}



\begin{proof}[Proof of Lemma \ref{lem:abp}]
    Consider the family of paraboloids
    \[
    \varphi_{y_0}(x) := -\frac{1}{2}|x-y_0|^2 + \frac{1}{2}\1\frac{3}{4}\2^2 \leq 1.
    \]
    Letting $\theta := 1/4$, we get that for every $y_0\in B_{1/4}$ we have $\varphi_{y_0}(0) > \theta$ and $\varphi_{y_0} < 0$ in $\R^n \sm B_1$. Together with the hypotheses $u(0) \leq \theta$ and $\inf_{B_1} u\geq 0$, they imply that for every $y_0\in B_{1/4}$
    \begin{align}\label{eq:2}
    \emptyset \neq \argmin (u - \varphi_{y_0}) \ss \{u\leq 1\}\cap B_1.        
    \end{align}
    (Figure \ref{fig:abp}). Therefore, the result would follow from a lower bound on the measure of the contact set $A = \bigcup_{y_0\in B_{1/4}}\argmin (u - \varphi_{y_0})$.

\begin{figure}
    \centering
    \includegraphics[width=0.8\textwidth]{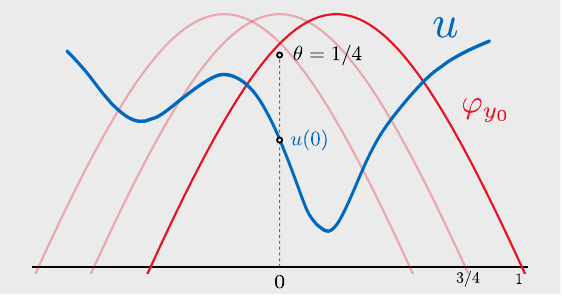}
    \caption{The paraboloids are arranged such that for $y_0 \in B_{1/4}$ it happens that $\varphi_{y_0}(0) > \theta$, and $\varphi_{y_0}<0$ in $\R^n\sm B_1$. Hence a positive function that dips below $\theta$ at the origin must necessarily cross all of these paraboloids. By translating each paraboloid downwards we obtain a contact point in $\{u\leq 1\}\cap B_1$.}
    \label{fig:abp}
\end{figure}

    Consider now the mapping $T\colon A\to B_{1/4}$ defined such that $T(x_0)=x_0+Du(x_0)$. Equivalently, $y_0=T(x_0)$ is the center of the paraboloid $\varphi_{y_0}$ such that $x_0 \in \argmin(u-\varphi_{y_0})$. By \eqref{eq:2} it also follows that $T$ is surjective. We finally observe that $DT = I+D^2u\geq 0$ on the contact set $A$, just by definition of the contact set and the second derivative test. 
    
    Now we will use the uniform ellipticity over the contact set. For every $x_0\in A$ we already have that the eigenvalues of $D^2u(x_0)$ are bounded from below by $-1$, then for any positive eigenvalue $e\in \operatorname{eig}(D^2u(x_0))$ we must have $\l e \leq (n-1)\L + f(x_0)$. This means that $\det(D^2u+I) \leq C(1+f^n)$ in $A$ for some $C\geq1$ depending on the dimension $n$ and the ellipticity constants.
    
    By the area formula and the hypothesis $\|f\|_{L^n(B_1)}\leq \d$,
    \[
    |B_{1/4}| \leq \int_{A}\det(D^2u+I) \leq C|A| + C\d^n.
    \]
    By choosing $\d$ sufficiently small we can absorb the term $C\d^n$ on the left-hand side and get the desired estimate for $\eta := |B_{1/4}|/(2C)$.
\end{proof}

The purpose of the following lemma is to localize the measure estimate close to the origin.

\begin{lemma}[Localization]\label{lem:loc}
    Given $[\l,\L]\ss(0,\8)$ and $\r\in(0,1/2)$, there exist $\d\in(0,1)$ and $M\geq 1$, such that the following holds: Let $u\in C^2(B_1)$ and $f\in L^n(B_1)$ with $\|f\|_{L^n(B_1)}\leq \d$, both non-negative functions such that $\mathcal M^-_{\l,\L}(D^2u)\leq f$. Then
    \[
    \inf_{B_{1/2}} u \leq 1 \qquad\Rightarrow\qquad \inf_{B_{\r}}u \leq M.
    \]
\end{lemma}

\begin{figure}
    \centering
    \includegraphics[width=0.8\textwidth]{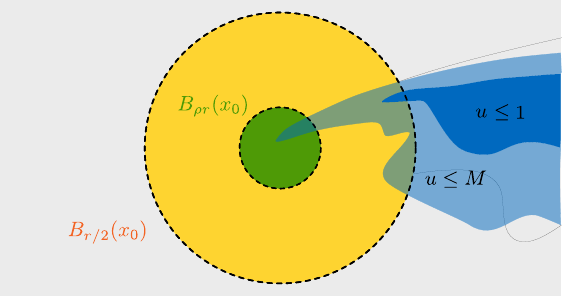}
    \caption{For any $B_r(x_0)$ in the domain of the equation, we have that if the set $\{u\leq 1\}$ enters $\overline{B_{r/2}(x_0)}$, then the set $\{u\leq M\}$ has positive density in $B_{\r r}(x_0)$.}
    \label{fig:mest}
\end{figure}

Before proving Lemma \ref{lem:loc}, let us point out a key consequence. By combining the measure estimate (Lemma \ref{lem:abp}), the localization (Lemma \ref{lem:loc}), and using the scaling symmetries of the equation, we obtain the following corollary. In Figure \ref{fig:mest} we illustrate a rescaled version of this result, which will be an important tool in the next section.

\begin{corollary}[Localized measure estimate]\label{cor:mest}
Given $[\l,\L]\ss(0,\8)$ and $\r \in(0,1/2)$, there exist $\d,\eta\in(0,1)$ and $M\geq 1$, such that the following holds: Let $u\in C^2(B_1)$ and $f\in L^n(B_1)$ with $\|f\|_{L^n(B_1)}\leq \d$, both non-negative functions such that $\mathcal M^-_{\l,\L}(D^2u)\leq f$. Then
\[
    \inf_{B_{1/2}}u \leq 1 \qquad\Rightarrow\qquad \frac{|\{u\leq M\}\cap B_{\r}|}{|B_{\r}|}\geq \eta.
\]
\end{corollary}

The detailed proof of Corollary \ref{cor:mest} can be found in Section \ref{sec:sum}. 

\begin{proof}[Proof of Lemma \ref{lem:loc}]
    Consider the family of test functions (Figure \ref{fig:abp4})
    \[
    \varphi_{y_0}(x) := C_0\frac{q(|x-y_0|)-q(1-\r/2)}{q(1/2+\r/2)-q(1-\r/2)}, \qquad q(r) := \min\{r^{-\a},(\r/2)^{-\a}\},
    \]
    where $C_0,\a\geq 1$ will be chosen sufficiently large. We see that for any $y_0 \in B_{\r/2}$, we have $\inf_{B_{1/2}}\varphi_{y_0} \geq 1$, and $\sup_{\R^n\sm B_1}\varphi_{y_0} \leq 0$.
    
    Let $M := \max_{\R^n} \varphi_{y_0}$. If we assume by contradiction that $\inf_{B_{\r}}u > M$ we obtain that for each $y_0\in B_{\r/2}$
    \begin{align}
        \label{eq:3}
    \emptyset\neq \argmin(u-\varphi_{y_0}) \ss B_1\sm \overline{B_{\r}}.
    \end{align}

    \begin{figure}
        \centering
        \includegraphics[width=0.8\textwidth]{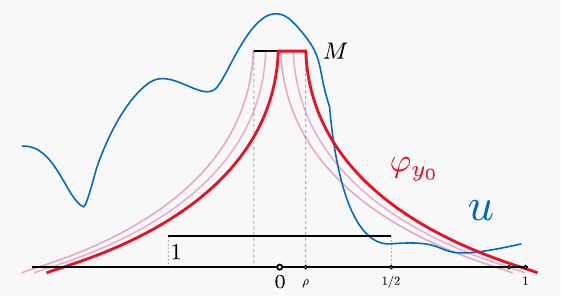}
        \caption{The test functions are arranged such that for $y_0\in B_{\r/2}$, they are smooth in $B_1\sm \overline{B_\r}$, remain above one in $\overline{B_{1/2}}$, and below zero in $\R^n\sm B_1$. Hence, a positive function that goes above their maximum $M$ in $\overline{B_\r}$, and dips below one in $\overline{B_{1/2}}$ must necessarily cross all the functions in $B_1\sm \overline{B_\r}$. As in Lemma \ref{lem:abp}, we can translate the test functions downwards until the touch $u$ at some point in $B_1\sm \overline{B_\r}$.}
        \label{fig:abp4}
    \end{figure}
    
    Let $A = \bigcup_{y_0 \in B_{\r/2}} \argmin(u-\varphi_{y_0})$, and $T:A\to B_{\r/2}$ such that $T(x_0)=y_0$ if $x_0 \in \argmin (u-\varphi_{y_0})$. This mapping is surjective due to \eqref{eq:3} and gets computed from the tangency relation at the contact point
    \[
    T(x_0) = x_0 - D\varphi_0^{-1}(Du(x_0)) = x_0 + C_0C_1\a|Du(x_0)|^{-(2+\a)/(1+\a)}Du(x_0).
    \]
    where $C_1:= (q(1/2+\r/2)-q(1-\r/2))^{-1}$.

    Let $\Phi(x) := D^2\varphi_0(x-T(x))$. By the area formula
    \begin{align}\label{eq:area}
        |B_{\r/2}| \leq \int_A |\det(DT)| = \int_A |\det(I - \Phi^{-1} D^2u)|.
    \end{align}

    To estimate the integrand we use the following fact from linear algebra: For $N_1,N_2 \in \R^{n\times n}_{\text{sym}}$ such that $\det N_1\neq 0$\footnote{For $M\in \R^{n\times n}$, we denote $|M|_{\text{op}} := \sup_{|x|=1}|Mx|$. If $M$ is symmetric then we have that
    \[
    |M|_{\text{op}} = \max\{|e| \ | \ e\in \operatorname{eig}(M)\}.
    \]}
    \begin{align*}
    |\det(I - N_1^{-1}N_2)| &= \sqrt{\det((I-N_1^{-1}N_2)^T(I-N_1^{-1}N_2))}\\
    &\leq |I - N_2N_1^{-1} - N_1^{-1}N_2 + N_2N_1^{-2}N_2|_{\text{op}}^{n/2}\\
    &\leq (1+2|N_1^{-1}|_{\text{op}}|N_2|_{\text{op}}+|N_1^{-1}|_{\text{op}}^2|N_2|_{\text{op}}^2)^{n/2}\\
    &\leq C(1+|N_1^{-1}|_{\text{op}}^n|N_2|_{\text{op}}^n)\\
    &= C\11+\1\frac{\max_{e\in \operatorname{eig}(N_2)}|e|}{\min_{e\in \operatorname{eig}(N_1)}|e|}\2^n\2.
    \end{align*}
    
    For $z \in \R^n\sm \overline{B_{\r/2}}$ we compute
    \[
    D^2\varphi_0(z) = C_0C_1\a|z|^{-\a-2}\1(\a+2)\frac{z\otimes z}{|z|^2} - I\2.
    \]
    From this we get that the eigenvalues of $D^2\varphi_0(z)$ are: $C_0C_1\a|z|^{-\a-2}(\a+1)$ with single multiplicity, and $-C_0C_1\a|z|^{-\a-2}$ with multiplicity $(n-1)$. 
    
    For $x_0 \in A$ and $y_0=T(x_0)$
    \begin{align*}
    f(x_0) \geq \mathcal M^-_{\l,\L}(D^2\varphi_{y_0}(x_0)) = C_0C_1\a|x_0-y_0|^{-\a-2}(\l (\a+1)-\L(n-1)) .
    \end{align*}
    We can make the right-hand side positive by fixing $\a:=\L n/\l$. Then we can also make the right-hand side $\geq 1$ by fixing $C_0$ large enough, depending on $\r$ and $\a$.

    Once we have fixed the parameters $\a$, $C_0$, and hence also $M$; we notice that $|D^2\varphi_{y_0}|_{\text{op}}$ and $|(D^2\varphi_{y_0})^{-1}|_{\text{op}}$ are bounded over $B_1\sm \overline{B_{\r}}\supseteq A$, uniformly with respect to $y_0\in B_{\r/2}$. The eigenvalues of $D^2u(x_0)$ are bounded from below by $-|D^2\varphi_{y_0}(x_0)|_{\text{op}}$. On the other hand, the equation provides an upper bound for the positive eigenvalues $e \in \operatorname{eig}(D^2u(x_0))$ given by $\l e \leq \L (n-1)|D^2\varphi_{y_0}(x_0)|_{\text{op}} + f(x_0)$.
    
    Putting together the information about the eigenvalues of $N_1= \Phi = D^2\varphi_{y_0}$ and $N_2= D^2u$, and using that $A\ss \{f\geq 1\}$, we conclude that $|DT| = |\det(I - \Phi^{-1} D^2u)| \leq Cf^n$ in $A$. Due to \eqref{eq:area} we obtain $|B_{\r/2}|\leq C\|f\|_{L^n(B_1)}^n \leq C\d^n$, however this is a contradiction if we finally fix $\d$ sufficiently small.
\end{proof}




\begin{exercise}[\hypertarget{ex:mp}{\hyperlink{sol:mp}{\textbf{ABP maximum principle.}}} The techniques used in the previous two lemmas were originally used to prove the Aleksandrov-Bakelman-Pucci maximum principle, for this reason they are also known as \textit{ABP-type lemmas}. In this case the test functions are linear functions parameterized by its gradient.]
    \item Given $[\l,\L]\ss(0,\8)$ there exists $C>1$ such that the following holds: Let $u \in C^2(B_1)\cap C(\overline{B_1})$ and $f\in L^n(B_1)$ non-negative such that
    \[
    \begin{cases}
        \mathcal M_{\l,\L}^-(D^2u) \leq f \text{ in } B_1,\\
        u\geq 0 \text{ on } \p B_1.
    \end{cases}
    \]
    Consider the contact set
    \[
    A := \bigcup_{p\in \R^n} \argmin (u - \varphi_p), \qquad \varphi_p(x):= p\cdot x.
    \]
    Then $\sup_{B_1} u_- \leq C \|f\|_{L^n(A)}$.
    \item There exists some constant $C\geq 1$ depending on the dimension such that the following holds. Let $\W\ss\R^n$ be an open, convex and bounded set and $u\in C^2(\W)\cap C(\overline{\W})$ a convex function with $u=0$ on $\p\W$. Then, for every $x_0\in \W$
    \[
    |u(x_0)|^n \leq C\diam(\W)^{n-1}\dist(x_0,\p\W)\int_\W \det(D^2u),
    \]
    This gives a $1/n$-Hölder modulus of continuity for convex functions at the boundary.
\end{exercise}

    \begin{figure}
        \centering
        \includegraphics[width=0.8\textwidth]{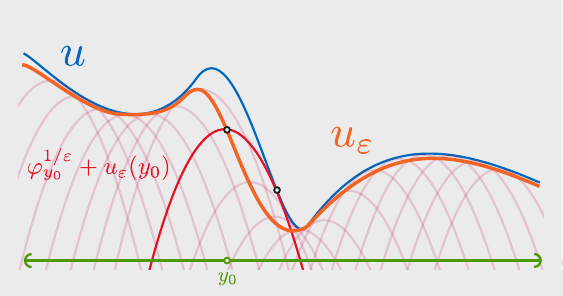}
        \caption{Geometrically, $(y_0,u_\e(y_0))$ gives the vertex of a concave paraboloid that touches $u$ from below. Similar to the approximations of the identity, we obtain that $u_\e$ approximates $u$ as $\e\to0$, and gains some regularity. This construction is essential in the theory of viscosity solutions.}
        \label{fig:inf_conv}
    \end{figure}

\begin{exercise}[\hypertarget{ex:envelope}{\hyperlink{sol:envelope}{\textbf{Inf/Sup-convolutions.}}} Given $\W\ss\R^n$ open, $u \in C(\W)$, and $\e>0$, we define the \textit{inf-convolution} of $u$ as $u_\e\colon \R^n \to \R$ such that
\[
u_\e(y_0) := \inf_\W \1u-\varphi^{1/\e}_{y_0}\2, \qquad \varphi^{1/\e}_{y_0}(x) := -\frac{1}{2\e}|x-y_0|^2.
\]
In a similar way, we define the \textit{sup-convolution} as
\[
u^\e(y_0) := \sup_\W \1u-\varphi^{-1/\e}_{y_0}\2 = -(-u)_\e(y_0).
\]
See Figure \ref{fig:inf_conv}.
]
\item Show that if $0 < \e_1 <\e_2$, then $u_{\e_2}\leq u_{\e_1}\leq u$.
\item Show that for any $\overline{B_r(x_0)}\ss\W$, $u_\e$ converges to $u$ uniformly in $\overline{B_r(x_0)}$.
\item Show that $x\mapsto u_\e(x)+|x|^2/(2\e)$ is concave.
\item Show the semi-group property: $(u_{1/\a_1})_{1/\a_2} = u_{1/(\a_1+\a_2)}$.
\item Show that $\G_\e := (u_\e)^\e$ is the the following upper-envelope of paraboloids
\[
\G_\e(x) = \sup\{\varphi^{1/\e}_{y_0}(x)+C \ | \ \varphi^{1/\e}_{y_0}+C\leq u\}.
\]
\item Show that
\[
\{u = \G_\e\} = \bigcup_{y_0 \in \R^n} \argmin(u-\varphi_{y_0}^{1/\e}).
\]
\end{exercise}

\begin{exercise}[\hypertarget{ex:abp_drift}{\hyperlink{sol:abp_drift}{\textbf{Measure estimates for operators with a drift term.}}}]
    \item (Measure estimate with small drift) Show that for $[\l,\L]\ss(0,\8)$, there exist $\d,\theta,\eta\in(0,1)$, such that the following holds: Let $u\in C^2(B_1)$, and $b,f \in L^n(B_1)$ with $\|b\|_{L^n(B_1)}\leq \d$, $\|f\|_{L^n(B_1)}\leq \d$, all of them non-negative functions such that
    \[
    \mathcal M^-_{\l,\L}(D^2u) - b|Du| \leq f \text{ in } B_1.
    \]
    Then
    \[
    u(0) \leq \theta \qquad\Rightarrow\qquad |\{u\leq 1\}\cap B_1|\geq \eta.
    \]
    \item (Localization with small drift) Prove a similar extension for Lemma \ref{lem:loc} and its corollary.
    \item (ABP maximum principle) Prove a similar extension for the ABP maximum principle without the smallness condition on $\|b\|_{L^n(B_1)}$: Given $[\l,\L]\ss(0,\8)$ there exists $C>1$ such that the following holds: Let $u \in C^2(B_1)\cap C(\overline{B_1})$, and $b,f \in L^n(B_1)$ non-negative with $\|b\|_{L^n(B_1)}\leq \L$ such that
    \[
    \begin{cases}
        \mathcal M_{\l,\L}^-(D^2u) - b|Du| \leq f \text{ in } B_1,\\
        u\geq 0 \text{ on } \p B_1.
    \end{cases}
    \]
    Consider the contact set
    \[
    A := \bigcup_{p\in \R^n} \argmin (u - \varphi_p), \qquad \varphi_p(x):= p\cdot x.
    \]
    Then
    \[
    \sup_{B_1}u_- \leq C \|f\|_{L^n(A)}.
    \]
    \item (Localized measure estimate) Prove a similar extension for Corollary \ref{cor:mest} without the smallness condition on $\|b\|_{L^n(B_1)}$: Given $[\l,\L]\ss(0,\8)$ and $\r \in(0,1/2)$, there exist $\d,\eta\in(0,1)$ and $M\geq 1$, such that the following holds: Let $u\in C^2(B_1)$, and $b,f \in L^n(B_1)$ with $\|b\|_{L^n(B_1)}\leq \L$ and $\|f\|_{L^n(B_1)}\leq \d$, all of them non-negative functions such that
    \[
    \mathcal M^-_{\l,\L}(D^2u) - b|Du| \leq f \text{ in } B_1.
    \]
    Then
    \[
    \inf_{B_{1/2}}u \leq 1 \qquad\Rightarrow\qquad \frac{|\{u\leq M\}\cap B_{\r}|}{|B_{\r}|}\geq \eta.
    \]
\end{exercise}

\begin{exercise}[\hypertarget{ex:par_eq1}{\hyperlink{sol:par_eq1}{\textbf{Measure estimates for parabolic equations.}}} For dynamic problems, we consider contact points $(x_0,t_0) \in B_1\times(-1,0\rbrack$ with respect to a function $u=u(x,t) \in C(B_1\times(-1,0\rbrack)$ and a test function $\varphi=\varphi(x,t)$ from below. The time $t_0$ is the first time when the graph of $\varphi$ touches the graph of $u$ from below, and the point $x_0$ is any point where the contact occurs at time $t_0$.]
    \item (Maximum principle) Show that for $[\l,\L]\ss(0,\8)$ there exists $C>1$ such that the following holds: Let $u \in C^2(B_1\times(-1,0])\cap C(\overline{B_1}\times[-1,0])$ and $f \in L^{n+1}(B_1\times(-1,0])$ non-negative with\footnote{$\p_{\text{par}}(B_1\times(-1,0]) = \p B_1\times(-1,0] \cup B_1\times\{-1\}$ denotes the parabolic boundary.}
    \[
    \begin{cases}
        \p_t u - \mathcal M_{\l,\L}^- (D^2u) \geq -f \text{ in } B_1\times(-1,0],\\
        u\geq 0 \text{ on } \p_{\text{par}}(B_1\times(-1,0]).
    \end{cases}
    \]
    Consider for $p\in \R^n$, $h\in \R$, and $\varphi_{p,h}(x) := p\cdot x-h$, the contact set
    \begin{align*}
    A := \{(x_0,t_0) \in B_1\times(-1,0] \ | \ &\text{For some $(p_0,h_0) \in \R^{n+1}$,}\\
    &\text{$u > \varphi_{p_0,h_0}$ in $B_1\times(-1,t_0)$,}\\
    &\text{and $u(x_0,t_0) = \varphi_{p_0,h_0}(x_0,t_0)$}\}.
    \end{align*}
    Then
    \[
    \sup_{B_1\times(-1,0]} u_- \leq C \|f\|_{L^{n+1}(A)}.
    \]
    \item (Measure estimate) Given $[\l,\L]\ss(0,\8)$ there exist $\d,\theta,\eta\in(0,1)$ such that the following holds: Let $u \in C^2(B_1\times(-1,0])$ and $f \in L^{n+1}(B_1\times(-1,0])$ with $\|f\|_{L^{n+1}(B_1\times(-1,0])}\leq \d$, both non-negative functions such that
    \[
    \p_t u - \mathcal M^-_{\l,\L}(D^2u) \geq -f \text{ in } B_1\times(-1,0].
    \]
    Then
    \[
    u(0,0) \leq \theta \qquad\Rightarrow\qquad |\{u\leq 1\}\cap B_1\times(-1,0]| \geq \eta.
    \]
    \item (Localization) Given $[\l,\L]\ss(0,\8)$ and $\r\in(0,1/2)$ there exist $\d,\eta\in(0,1)$ and $M\geq 1$, such that the following holds: Let $u \in C^2(B_1\times(-1,0])$ and $f \in L^{n+1}(B_1\times(-1,0])$ with $\|f\|_{L^{n+1}(B_1\times(-1,0])}\leq \d$, both non-negative functions such that
    \[
    \p_t u - \mathcal M^-_{\l,\L}(D^2u) \geq -f \text{ in } B_1\times(-1,0].
    \]
    Then
    \begin{align*}
    \inf_{B_{1/2}\times(-1/2,0]} u\leq 1 \qquad\Rightarrow\qquad \inf_{B_{\r}\times(-1,-1+\r]} u \leq M.
    \end{align*}
    \item (Localized measure estimate) Given $[\l,\L]\ss(0,\8)$ and $\r\in(0,1/2)$ there exist $\d,\eta\in(0,1)$ and $M\geq 1$, such that the following holds: Let $u \in C^2(B_1\times(-1,0])$ and $f \in L^{n+1}(B_1\times(-1,0])$ with $\|f\|_{L^{n+1}(B_1\times(-1,0])}\leq \d$, both non-negative functions such that
    \[
    \p_t u - \mathcal M^-_{\l,\L}(D^2u) \geq -f \text{ in } B_1\times(-1,0].
    \]
    Then
    \begin{align*}
    \inf_{B_{1/2}\times(-1/2,0]} u \leq 1 \qquad\Rightarrow\qquad \frac{|\{u\leq M\}\cap B_\r\times(-1/2-\r,-1/2]|}{|B_\r\times(-1/2-\r,-1/2]|} \geq \eta.
    \end{align*}
\end{exercise}

\begin{exercise}[\hypertarget{ex:w2p}{\hyperlink{sol:w2p}{\textbf{Measure estimates for the Hessian I.}}} We have seen in this section and the previous problems that the contact set captures information about $u$ and also its gradient, due to the first derivative test. It is natural to expect that it also captures information about its Hessian, due to the second derivative test. For $M>0$, $B\ss\R^n$, and $u\in C^2(B_1)$, consider the contact sets
\[
A^M_B := \bigcup_{y_0\in B}\argmin\1u-\varphi_{y_0}^M\2 \ss \{|(D^2u)_-|_{\text{op}}\leq M\}, \qquad \varphi_{y_0}^M(x):= -\frac{M}{2}|x-y_0|^2.
\]]
\item Given $[\l,\L]\ss(0,\8)$ and $\r \in (0,1/2)$, there exist $\d,\eta\in(0,1)$ and $M\geq 1$, such that for $u\in C^2(B_1)$ with $\|(\mathcal M^-_{\l,\L}(D^2u))_+\|_{L^n(B_1)}\leq \d$,
\[
    A^1_{\{0\}}\cap \overline{B_{1/2}} \neq \emptyset \qquad\Rightarrow\qquad \frac{|A^M_{\R^n}\cap B_\r|}{|B_\r|}\geq \eta.
\] 
\end{exercise}

\subsubsection{Diminish of Distribution}\label{sec:dim_dist}

The localized measure estimate proved in the previous section concludes that
\[
|\{u>M\}\cap B_{\r}| \leq (1-\eta)|B_\r|.
\]
In this section, we will iterate this estimate and obtain the following geometric decay for the distribution
\[
|\{u>M^k\}\cap B_{\r}| \leq (1-\eta)|\{u>M^{k-1}\}\cap B_{\r}| \leq \ldots \leq (1-\eta)^k|B_\r|.
\]
As soon as $(1-\eta)^k$ becomes less than $1/2$, we are in a position to prove the diminish of oscillation (Lemma \ref{lem:dim_osc2}).

By the homogeneity of the equation, it suffices to show the step from $k=0$ to $k=1$
\[
|\{u> M\}\cap B_{\r}|\leq (1-\eta)|\{u>1\}\cap B_{\r}|,
\]
which by rearranging the terms is equivalent to the following estimate (Figure \ref{fig:isoper})
\begin{align}
    \label{eq:4}
    |\{1 < u \leq M\}\cap B_{\r}| \geq \eta|\{u> 1\}\cap B_{\r}|.
\end{align}

At this point, we pose the equation in $B_{3\sqrt n}$, replace $B_\r$ with the cube $Q_1 = (-1/2,1/2)^n$ of unit length, and $\overline{B_{1/2}}$ for $\overline{Q_3} = [-3/2,3/2]^n$. This modification will be convenient for implementing a covering argument using dyadic cubes.

    \begin{figure}
        \centering
        \includegraphics[width=0.8\textwidth]{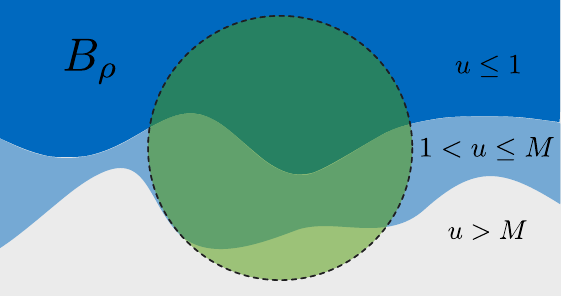}
        \caption{The estimate $|\{1 < u \leq M\}\cap B_{\r}| \geq \eta|\{u> 1\}\cap B_{\r}|$ resembles the isoperimetric inequality where the set $\{1 < u \leq M\}$ plays the role of the perimeter for the set $\{u>1\}$.}
        \label{fig:isoper}
    \end{figure}

\begin{lemma}[Diminish of distribution]\label{lem:dimdist}
    Given $[\l,\L]\ss(0,\8)$, there exist $\d,\eta\in(0,1)$ and $M\geq 1$ such that the following holds: Let $u\in C^2(B_{3\sqrt n})$ and $f \in L^n(B_{3\sqrt n})$ with $\|f\|_{L^n(B_{3\sqrt n})}\leq \d$, both non-negative functions such that $\mathcal M^-_{\l,\L}(D^2u)\leq f$. Then
    \[
    \inf_{Q_3} u \leq 1 \qquad\Rightarrow\qquad |\{1 < u \leq M\}\cap Q_1| \geq \eta|\{u> 1\}\cap Q_1|.
    \]
\end{lemma}

The proof of this lemma will be guided in a moment after discussing its most important consequence, the weak Harnack inequality.

Taking $\e = -\ln(1-\eta)/\ln M$ we obtain the following result with an argument similar to the one in the proof of Lemma \ref{cor:dim_osc}.

\begin{lemma}[Weak Harnack inequality - uniformly elliptic]
    \label{lem:wharnack}
    Given $[\l,\L]\ss(0,\8)$, there exist $\e\in(0,1)$ and $C\geq 1$ such that the following holds: Let $u\in C^2(B_{3\sqrt n})$ and $f \in L^n(B_{3\sqrt n})$ both non-negative functions such that $\mathcal M^-_{\l,\L}(D^2u)\leq f$. Then
    \[
    \sup_{\m>0}\m^{\e}|\{u \geq \m\}\cap Q_1|\leq C\1\inf_{Q_3} u+\|f\|_{L^n(B_{3\sqrt n})}\2^\e.
    \]
\end{lemma}

The outline of the proof of Lemma \ref{lem:wharnack} is the iteration discussed at the beginning of this section. Details are given in Section \ref{sec:sum}.

To prove Lemma \ref{lem:dimdist} we will integrate the local information given by Corollary \ref{cor:mest}. To do this, we use the \textit{dyadic decomposition} of a relatively open set of the unit cube.

The \textit{dyadic cubes} in $\overline{Q_1} = [-1/2,1/2]^n$ are a collection of closed cubes constructed by an iterative procedure: Each dyadic cube has the form
\[
\overline{Q_{2^{-k}}(x_0)} = \prod_{i=1}^n [(x_0)_i-2^{-(k+1)},(x_0)_i+2^{-(k+1)}],
\]
where $k$ is a non-negative integer, $x_0 = ((x_0)_1,\ldots,(x_0)_n)\in \R^n$ is its center, and $2^{-k}$ is the length of its sides. The unit cube $\overline{Q_1}$ is dyadic and forms the zeroth generation. Each successive generation is obtained from the previous one by subdividing each cube $Q$ into $2^n$ congruent cubes of half the length. We call these cubes the \textit{descendants} of $Q$, and for each descendant of $Q$ we say that $Q$ is its \textit{progenitor}.

For a relatively open set $E \ss \overline{Q_1}$, its dyadic decomposition is the set of dyadic cubes $Q\ss E$ such that the progenitor of $Q$ is not fully contained in $E$, see \cite[Theorem 1.4]{MR2129625}. This collection forms an almost-disjoint covering of $E$ (any pair of cubes has disjoint interiors). Figure \ref{fig:covering} illustrates some stages in the decomposition of $\{u>1\}\cap \overline{Q_1}$.

    \begin{figure}
        \centering
        \includegraphics[width=0.8\textwidth]{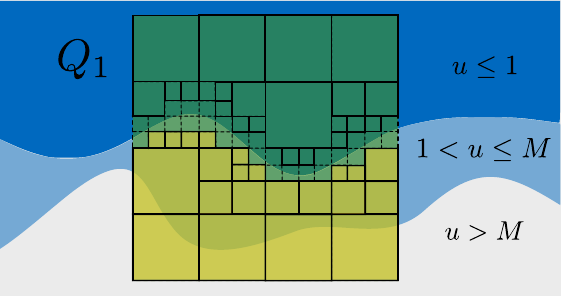}
        \caption{Each cube $Q$ in the dyadic decomposition of $E = \{u>1\}\cap \overline{Q_1}$ is selected by the following algorithm: Starting from the zeroth generation, a dyadic cube will be discarded (with all its descendants) if it does not intersect $E$, it will be selected if it is contained in $E$, otherwise its descendants will be passed to the next iteration.}
        \label{fig:covering}
    \end{figure}

\begin{proof}[Proof of Lemma \ref{lem:dimdist}]
Let $\d_0,\eta_0\in(0,1)$ and $M_0\geq 1$ be the constants from Corollary \ref{cor:mest} with respect to $\r := 1/(6\sqrt n)$ and the same ellipticity constants. We will show that Lemma \ref{lem:dimdist} holds for $\d := \d_0/(3\sqrt{n})$, $\eta:= \eta_0|B_{1/2}|$, and $M:=M_0$.

Let $Q := \overline{Q_\ell(x_0)}$ be an arbitrary dyadic cube in the decomposition of $\{u>1\}\cap \overline{Q_1}$. Let $r := \ell/(2\r) = 3\sqrt{n}\ell$, using that $\d = \d_0/(3\sqrt{n})$, we see that for the rescalings $v(x) := u(r x+x_0)$ and $g(x):= r^2f(r x+x_0) \geq \mathcal M_{\l,\L}^-(D^2v(x))$ we get
\[
\|g\|_{L^n(B_1)} = r\|f\|_{L^n(B_{r}(x_0))} \leq \d_0.
\]
Given that the predecessor of $Q$ intersects $\{u\leq 1\}$ 
\[
\inf_{B_{1/2}}v = \inf_{B_{3\sqrt{n}\ell/2}(x_0)}u \leq \inf_{Q_{3\ell}(x_0)}u \leq 1.
\]

Now we are in shape to apply Corollary \ref{cor:mest} to $v$ and $g$, obtaining in this way the following estimate for the distribution of $u$ over $Q$
\[
\frac{|\{u\leq M\}\cap Q|}{|Q|} \geq |B_{1/2}|\frac{|\{u\leq M\}\cap B_{\ell/2}(x_0)|}{|B_{\ell/2}(x_0)|}= |B_{1/2}|\frac{|\{v\leq M\}\cap B_{\r}|}{|B_{\r}|} \geq |B_{1/2}|\eta_0 = \eta.
\]

By adding this estimate to all the dyadic cubes in the decomposition of $\{u>1\}\cap \overline{Q_1}$
\[
|\{1 < u \leq M\}\cap Q_1| = \sum_Q |\{u\leq M\}\cap Q| \geq \eta \sum_Q |Q| = \eta|\{u>1\}\cap Q_1|,
\]
which concludes the proof.
\end{proof}



\begin{figure}
    \centering
    \includegraphics[width=0.8\textwidth]{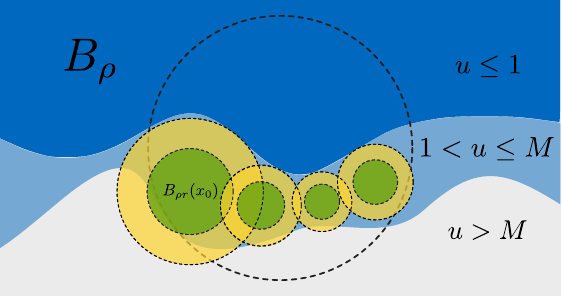}
    \caption{For $x_0 \in \{u>1\}$, and $r/2 = \dist(x_0,\{u\leq 1\})$, we get that the set $\{1<u\leq M\}$ has positive density in $B_{\r r}(x_0)$. The idea is then to add up this local information. The caveat is that the ball $B_{\r r}(x_0)$ may spill out of $B_\r$.}
    \label{fig:covering3}
\end{figure}
    
\noindent\textit{Growing Ink-spots.} It is also possible to prove Lemma \ref{lem:dimdist} using Vitali's covering lemma. In the Figure \ref{fig:covering3} we illustrate this approach and a geometric challenge.

\begin{lemma}[Growing ink-spots]\label{lem:grow_ink}
    Let $\r_0:=1/6$ and $\r_1:=1/7$. Given $\eta\in (0,1)$ and $F\ss E\ss \R^n$ closed such that $\overline{B_{\r_1}} \cap F \neq \emptyset$, and for every ball $B_r(x_0) \ss B_1$
    \[
    \overline{B_{r/2}(x_0)} \cap F \neq \emptyset \qquad\Rightarrow\qquad \frac{|E\cap B_{\r_0 r}(x_0)|}{|B_{\r_0 r}(x_0)|}\geq \eta,
    \]
    it holds $|(E\sm F)\cap B_{\r_1}|\geq 10^{-n}\eta|B_{\r_1}\sm F|$.
\end{lemma}

The hypotheses in the previous lemma are satisfied by $F = \{u\leq 1\}$ and $E=\{u\leq M\}$, see Figure \ref{fig:mest}. This approach was used in \cite{MR2334822} to deduce the Lemma \ref{lem:dimdist}.

\begin{proof}
    Given $x_0 \in B_{\r_1}\sm F$, let $r = r(x_0) := 2\dist(x_0,F)$. Consider as well $x_1\in B_{r/2}(x_0) \cap B_{\r_1}$ such that
    \[
    B_{r/4}(x_1) \ss B_{r/2}(x_0) \cap B_{\r_1} \qquad \text{ and } \qquad B_{r/2}(x_0) \ss B_{3r/4}(x_1). 
    \]
    Figure \ref{fig:growing_ink} details the construction of $x_1$.

    \begin{figure}
    \centering
    \includegraphics[width=0.8\textwidth]{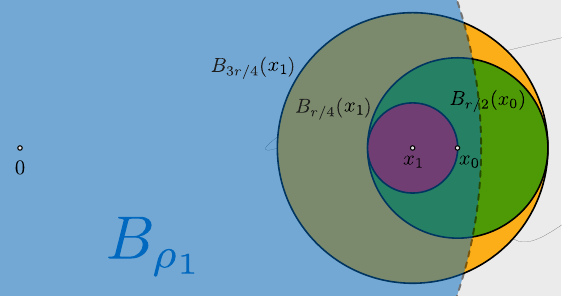}
    \caption{We may attempt to use Vitali's covering lemma on the set of balls given by $\{B_{\r_0 r}(x_0) \ | \ x_0 \in B_{\r_1}\sm F, r = r(x_0) := 2\dist(x_0,F)\}$ in order to estimate the measure of $(E\sm F)\cap B_{\r_1}$ from below, using the density hypothesis of the lemma. However, we find a problem because $B_{\r_0r}(x_0)$ may spill outside $B_{\r_1}$. We fix this issue by considering a new center $x_1$ over the segment from $0$ to $x_0$ with $|x_1-x_0|=r/4$, such that $B_{r/4}(x_1) \ss B_{r/2}(x_0) \cap B_{\r_1}$ and $B_{r/2}(x_0) \ss B_{3r/4}(x_1)$. In the exceptional case that $x_0=0$ we take $x_1=0$.}
    \label{fig:growing_ink}
    \end{figure}
    
    Given that $\overline{B_{\r_1}} \cap F \neq \emptyset$, we get $r\leq 4\r_1$. Then $B_{3r/2}(x_1)\ss B_{7\r_1}\ss B_1$, since $3r/2 + \r_1 \leq 7\r_1=1$. Because $B_{r/2}(x_0) \ss B_{3r/4}(x_1)$ and $\overline{B_{r/2}(x_0)} \cap F \neq \emptyset$, we also get $\overline{B_{3r/4}(x_1)} \cap F \neq \emptyset$. Then we can apply the hypothesis of the lemma to $B_{3r/2}(x_1)$ and deduce that
    \[
    \frac{|E\cap (B_{r/2}(x_0)\cap B_{\r_1})|}{|B_{r/2}|}\geq \frac{|E\cap B_{r/4}(x_1)|}{|B_{r/4}|}\frac{|B_{r/4}|}{|B_{r/2}|} \geq 2^{-n}\eta.
    \]
    
    By Vitali's covering lemma, we can extract a countable collection of disjoint balls $\mathcal B \ss \{B_{r(x_0)/2}(x_0) \ | \ x_0 \in B_{\r_1}\sm F\}$ such that $\sum_{B\in \mathcal B} |B| \geq 5^{-n}|B_{\r_1}\sm F|$. Then
    \[
    |(E\sm F)\cap B_{\r_1}| = \sum_{B\in \mathcal B} |E\cap (B\cap B_{\r_1})| \geq 2^{-n}\eta\sum_{B\in \mathcal B}|B| \geq 10^{-n}\eta|B_{\r_1}\sm F|,
    \]
    which concludes the proof.
\end{proof}

The localized measure estimate in \hyperlink{ex:par_eq1}{Problem 11} has a time shift in relation to the sets where $u$ dips below one ($\overline{B_{1/2}}\times\lbrack-1/2,0\rbrack$), and the set where the measure of $\{u\leq M\}$ is bounded from below ($B_\r\times(-1/2-\r,-1/2\rbrack$). This requires a technical modification of the covering arguments to obtain the corresponding weak Harnack inequality.

The dyadic parabolic decomposition was used by Lihe Wang in \cite{MR1135923} and more recently in the notes \cite{imbert:hal-00798300}. We define the dyadic cylinders generated from $\overline{Z_1} := \overline{Q_1}\times[-1,0]$ with a similar procedure that sub-divides the spatial lengths in halves and the temporal lengths in fourths, being consistent in this way with the scaling of the parabolic equation. A useful geometric result in this setting is the staking lemma in the first part of the next problem.

An alternative approach, using Vitali's covering lemma, was presented by Yu Wang \cite{MR3158522}, we also recommend the notes \cite{mooneyp}. In this setting, we require an analogue of the growing ink-spot lemma, in particular the construction in the Figure \ref{fig:growing_ink} using cylinders instead of balls. This is the content of the second part in the following problem.

\begin{exercise}[\hypertarget{ex:par_eq2}{\hyperlink{sol:par_eq2}{\textbf{Covering lemmas for parabolic equations.}}}]
    \item (Staking lemma) For every cylinder $Z = Q\times[t-r,t]$ and $m\geq 1$, let $Z^m = Q\times[t,t+mr]$. Show that for any collection $\mathcal Z$ of dyadic cylinders,
    \[
    \left|\bigcup_{Z\in \mathcal Z} Z^m\right| \geq \frac{m}{m+1}\left|\bigcup_{Z\in \mathcal Z} (Z^m\cup Z)\right|.
    \]

    \item (Construction for the parabolic growing ink-spot lemma) Given $(x_0,t_0)\in\R^n\times\R$ and $r>0$, let
    \[
    E^{\pm}_r(x_0,t_0) := \{(x,t) \in \R^n\times\R \ | \ \pm(t-t_0) > |x-x_0|^2, |t-t_0|<r^2\}.
    \]
    Show that there exists $\r\in(0,1)$ such that for every $(x_0,t_0) \in E^-_1(0,0)$ and $r\in(0,\sqrt{|t_0|})$, there exists $(x_1,t_1) \in E^+_r(x_0,t_0)\cap E^-_1(0,0)$ for which
    \[
    E^+_{\r r}(x_1,t_1) \ss E^+_r(x_0,t_0)\cap E^-_1(0,0).
    \]
\end{exercise}

\begin{exercise}[\hypertarget{ex:w2p2}{\hyperlink{sol:w2p2}{\textbf{Measure estimates for the Hessian II.}}} The diminish of distribution strategy can also be used on the distribution for the second derivative estimates stated in \hyperlink{ex:w2p}{Problem 12}. Recall the definition of the contact set $A_B^M$ from \hyperlink{ex:w2p}{Problem 12}, and let us introduce the Hardy–Littlewood maximal operator of $f\colon\W\to\R$ as 
\[
m(f)(x) := \sup_{r>0} \frac{1}{|B_r|}\int_{B_r(x)\cap\W}|f|.
\]\label{ex:w2p2}]
\item Given $[\l,\L]\ss(0,\8)$, there exist $\d,\eta\in(0,1)$ and $M\geq 1$, such that the following holds: Let $u\in C^2(B_{3\sqrt n})$, $f \in L^n(B_{3\sqrt n})$ with $\|f\|_{L^n(B_{3\sqrt n})}\leq \d$ such that $\mathcal M^-_{\l,\L}(D^2 u)\leq f$ and $A^1_{\{0\}}\cap \overline{Q_3} \neq \emptyset$. Then, for any cube $Q$ in the dyadic decomposition of the set
\[
(\overline{Q_1}\sm A^1_{\R^n})\cup \{x \in \overline{Q_1} \ | \ m(f^n) > \d^n\},
\]
we have $|Q\cap A^M_{\R^n}|\geq \eta|Q|$.
\item Given $[\l,\L]\ss(0,\8)$, there exist $\e\in(0,1)$ and $C\geq 1$ such that the following holds: Let $u\in C^2(B_{3\sqrt n})$, $f \in L^n(B_{3\sqrt n})$ such that $\mathcal M^-_{\l,\L}(D^2 u)\leq f$. Then
    \[
    \sup_{\m>0}\m^{\e}|\{|(D^2u)_-|_{\text{op}} \geq \m\}\cap Q_1|\leq C\1\osc_{B_{3\sqrt n}}u+\|f\|_{L^n(B_{3\sqrt n})}\2^\e.
    \]
\end{exercise}

\subsubsection{Summary of the Proofs}\label{sec:sum}

At this point, one already has all the results required to prove the diminish of oscillation lemma for the Pucci operators, Lemma \ref{lem:dim_osc2}. With it we finally settle the proof of the Krylov-Safonov theorem giving the interior Hölder estimates for non-linear equations, Theorem \ref{thm:ihe}. We have skipped some details in order to highlight the fundamental ideas. To gain a better view of the theory, let us recapitulate the main steps and compare them with the analogous results for the Laplacian. Finally, we quickly assemble the missing proofs.

\noindent\textit{Measure estimates:} For the Laplacian it is represented by the mean value property, proved as an application of the divergence theorem. For the Pucci operators, we proved in Lemma \ref{lem:abp} that for a super-solution, the measure of the lower contact set with a family of tests functions is bounded from below. It relies on a different integration theorem, the area formula. These proofs are the only ones where the uniform ellipticity is actually used. Corollary \ref{cor:mest} combines Lemma \ref{lem:abp} and Lemma \ref{lem:loc} to obtain a measure estimate localized at the ball $B_\r$. 

\begin{proof}[Proof of Corollary \ref{cor:mest}]
    Let $M_0$ be the constant from the localization Lemma \ref{lem:loc} with respect to the radius $\r_0=\r/2$. By this lemma, we get $u(x_0) \leq M_0$ for some $x_0 \in \overline{B_{\r/2}}$.
    
    Let $\theta$ and $\eta_0$ from Lemma \ref{lem:abp}. Applying this lemma to $\theta M_0^{-1} u((\r/2) x+x_0)$, we get $|\{u\leq \theta^{-1}M_0\}\cap B_{\r/2}(x_0)|\geq \eta_0(\r/2)^n$. Therefore, the desired result follows for $M := \theta^{-1}M_0$ and $\eta := \eta_0/|B_2|$.
\end{proof}

\noindent\textit{Weak Harnack inequality:} For the Laplacian it followed from the mean value property by a geometric observation (Figure \ref{fig:weak_har}). In the more general case, we found a way to iterate the measure estimate (Corollary \ref{cor:mest}) in order to get a geometric decay of the distribution $|\{u>M^k\}\cap Q_1|$ in Lemma \ref{lem:dimdist}.

\begin{proof}[Proof of Lemma \ref{lem:wharnack}]
Follows as in the proof of Lemma \ref{cor:dim_osc}, assuming first that for $\d\in(0,1)$ from Lemma \ref{lem:dimdist}, we have $\|f\|_{L^n(B_{3\sqrt n})} \leq \d$ and $\inf_{Q_3} u \leq 1$. The goal is to show that for every $\m>0$
\[
|\{u \geq \m\}\cap Q_1|\leq C\m^{-\e}.
\]
Otherwise, we use the homogeneity of the Pucci operator and apply the following reasoning to the renormalized function
\[
\bar u := \frac{u}{\inf_{Q_3}u+\d^{-1}\|f\|_{L^n(B_{3\sqrt n})}}.
\]

For $M$ and $\eta$ as in Lemma \ref{lem:dimdist}, it follows inductively and once again by the homogeneity of the operator, that for every integer $k\geq 1$,
\[
|\{u>M^k\}\cap Q_1| \leq (1-\eta)^k.
\]
To be precise, the inductive step follows by applying Lemma \ref{lem:dimdist} to $M^{-k}u$. Notice that for $\e := -\ln(1-\eta)/\ln M$, we can replace $(1-\eta)^{k}$ on the right-hand side with $(M^k)^{-\e}$.

Given $\m\geq 1$, let $k$ be a non-negative integer such that $M^k \leq \m <M^{k+1}$. Then, once we fix $C := M^\e$, we obtain
\[
|\{u>\m\}\cap Q_1| \leq |\{u>M^k\}\cap Q_1| \leq (1-\eta)^k = C(M^{k+1})^{-\e} \leq C\m^{-\e}.
\]
If $\m \in (0,1)$, the desired estimate in the distribution is true because $C\geq 1$.
\end{proof}

\noindent\textit{Diminish of the oscillation:} The proof of the interior Hölder estimate (Lemma \ref{thm:h}) follows from the dilation invariance of the equation, in the same way as for the Laplacian (proof of the Lemma \ref{cor:dim_osc}), but using the diminish of oscillation Lemma \ref{lem:dim_osc2} instead. Lemma \ref{lem:dim_osc2} relies on the weak Harnack inequality (Lemma \ref{lem:wharnack}).

\begin{proof}[Proof of Lemma \ref{lem:dim_osc2}]
    Assume without loss of generality that $0\leq u\leq 1$ in $B_1$. If
    \[
    \frac{|\{u\geq 1/2\}\cap Q_{1/(3\sqrt n)}|}{|Q_{1/(3\sqrt n)}|} \geq \frac{1}{2},
    \]
    we apply a rescaling of Lemma \ref{lem:wharnack} to get $u\geq \theta := C^{-1/\e}2^{-(1+2\e)/\e}$ in $B_{\r}\ss Q_{1/\sqrt n}$, provided that $\d := C^{-1/\e}2^{-(1+2\e)/\e}$ and $\r:=1/(2\sqrt n)$. Otherwise, we apply Lemma \ref{lem:wharnack} to $(1-u)$ to get a similar improvement of the upper bound of $u$.
\end{proof}

\begin{exercise}[\hypertarget{ex:fl}{\hyperlink{sol:fl}{\textbf{Fractional Laplacian.}}} Elliptic integro-differential operators, such as the fractional Laplacians, present in their own definition some sort of mean value property. The present problem aims to illustrate this claim from the comparison principle. For $\s\in(0,2)$ we define the fractional Laplacian $\D^{\s/2} = -(-\D)^{\s/2}$ in terms of its Fourier symbol $\widehat{\D^{\s/2}} = -|\xi|^{\s}$. From this definition we can evaluate the operator on $u \in L^1(\R^n)$ at a point $x$ where $u$ is regular enough ($C^2$ is enough)
\[
\D^{\s/2} u(x) = C\int_{\R^n} \frac{u(y+x)+u(y-x)-2u(x)}{|y|^{n+2\s}}dy.
\]
The constant $C$ that appears above is a positive number that depends only on the parameter $\s$ and the dimension $n$.]

\item (Comparison principle) Show that for every $u,v \in L^1(\R^n)$ such that $u\leq v$ and $u(x_0)=v(x_0)$ it holds that $\D^{\s/2} u(x_0)\leq \D^{\s/2} v(x_0)$.

\item (Weak Harnack) Show that for every $\s\in(0,2)$ there exists $C\geq 1$ such that for $u \in L^1(\R^n) \cap C^2(B_{3/4})$ non-negative with $\|(\D^{\s/2} u)_+\|_{L^\8(B_{3/4})} \leq 1$
\[
\inf_{B_{1/2}}u \leq 1 \qquad\Rightarrow\qquad
\int_{B_1\sm B_{3/4}} u \leq C.
\]

\item (H\"older estimate) Show that for every $\s\in(0,2)$ there exist $\a\in(0,1)$ and $C\geq 1$ such that the following H\"older estimate holds for $u \in L^1(\R^n) \cap C^2(B_{3/4})$
\[
\sup_{r\in(0,1)}r^{-\a}\osc_{B_r}u \leq C\1\osc_{B_1} u + \|u\|_{L^1(\R^n\sm B_1)} + \|\D^{\s/2} u\|_{L^\8(B_{3/4})}\2.
\]
\end{exercise}

\begin{exercise}[\hypertarget{ex:int_grad_est}{\hyperlink{sol:int_grad_est}{\textbf{Interior gradient estimate.}}}]
\item Show that for $[\l,\L]\ss(0,\8)$, there exist $\a\in(0,1)$ and $C\geq 1$, such that the following holds: Let $F \in C(\R^{n\times n}_{\text{sym}})$ be translation invariant and uniformly elliptic with respect to $[\l,\L]$, $\W\ss\R^n$ open and $u\in C^3(\W)$ with $F(D^2u)=0$ in $\W$. Then
\[
[Du]^{(1+\a)}_{C^{0,\a}(\W)} \leq C\osc_{\W}u.
\]
\end{exercise}

\begin{remark}
    Hypothesis $u\in C^3(\W)$ can be replaced by $u\in C^2(\W)$ in the previous problem. In fact, one can use the translation invariance and the interior Hölder estimate (Lemma \ref{thm:h}) to show that the difference quotient $v_\a = (u_{h,e}-u)/h^\a \in C^{0,\a}_{\text{loc}}(\W)$, uniformly in $h>0$. By the interpolation result in \cite[Lemma 5.6]{MR1351007} we recover that $u \in C^{0,2\a}_{\text{loc}}(\W)$ if $\a \leq 1/2$, and otherwise we get $u \in C^{1,2\a-1}_{\text{loc}}(\W)$. This procedure can be iterated all the way to an interior Hölder estimate for the gradient.
\end{remark} 

\begin{exercise}[\hypertarget{ex:lmp2}{\hyperlink{sol:lmp2}{\textbf{Local maximum principle - uniformly elliptic.}}}]
    \item Show that for $[\l,\L]\ss(0,\8)$ and $\e>0$, there exists some $C\geq 1$ such that the following maximum principle holds for $u\in C^2(B_1)$
\[
\sup_{B_{1/2}} u_+ \leq C\1\|u_+\|_{L^\e(B_1)} + \|(\mathcal M^+_{\l,\L} u)_-\|_{L^n(B_1)}\2.
\]
\end{exercise}

\begin{exercise}[\hypertarget{ex:harnack2}{\hyperlink{sol:harnack2}{\textbf{Harnack inequality - uniformly elliptic.}}}]
    \item Show that for $[\l,\L]\ss(0,\8)$ there exists some $C\geq 1$ such that for $u\in C^2(B_1)$, and $f \in L^n(B_1)$, both non-negative functions with
    \[
    \begin{cases}
        \mathcal M_{\l,\L}^-(D^2u) \leq f \text{ in } B_1,\\
        \mathcal M_{\l,\L}^+(D^2u) \geq -f \text{ in } B_1,
    \end{cases}
    \]
    we have
\[
\sup_{B_{1/2}}u \leq C\1\inf_{B_{1/2}}u + \|f\|_{L^n(B_1)}\2.
\]
\end{exercise}

\begin{exercise}[\hypertarget{ex:hess_est_convx}{\hyperlink{sol:hess_est_convx}{\textbf{Interior Hessian estimates for convex operators.}}}]
\item Show that for $[\l,\L]\ss(0,\8)$ there exists some $C\geq 1$ such that the following holds: Let $F \in C(\R^{n\times n}_{\text{sym}})$ convex, uniformly elliptic with respect to $[\l,\L]\ss(0,\8)$, and with $F(0)=0$. Let $u\in C^4(B_1)$ a solution of $F(D^2u)=0$ in $B_1$, $e\in \p B_1$, and $v_2 = \p_e^2 u$. Then
\[
\sup_{B_{1/2}} (v_2)_- \leq C\osc_{B_1} u.
\]
This is effectively a bound on the negative eigenvalues of the Hessian of $u$. Given that $u$ also satisfies $\mathcal M^-_{\l,\L}(D^2u)\leq 0$, the bound on the negative eigenvalues also implies a similar bound on the positive ones. All in all, we get the $C^{1,1}$ estimate
\[
\|D^2u\|_{L^\8(B_{1/2})} \leq C\osc_{B_1} u.
\]
\item Let $F \in C(\R^{n\times n}_{\text{sym}})$ convex and uniformly elliptic with respect to $[\l,\L]\ss(0,\8)$. Given $A \in \R^{n\times n}_{\text{sym}}$ and $u\in C^4(B_1)$ a solution of $F(D^2u)=0$ in $B_1$, consider $v_A := \div(AA^TDu) = (AA^T):D^2u$. Show that $v_A$ satisfies $\mathcal M_{\l,\L}^-(D^2v_A)\leq 0$ in $B_1$.
\item Show that for $[\l,\L]\ss(0,\8)$ there exists some $\theta \in(0,1)$ such that the following holds: Let $F \in C(\R^{n\times n}_{\text{sym}})$ convex, uniformly elliptic with respect to $[\l,\L]\ss(0,\8)$, and with $F(0)=0$. Let $u\in C^4(B_1)$ a solution of $F(D^2u)=0$ in $B_1$, with $D^2u(0)=0$. Then
\[
\min_{\substack{P \in \R^{n\times n}_{\text{sym}}\\P^2=P}} v_P \geq -1 \text{ in } B_1 \qquad\Rightarrow\qquad \min_{\substack{P \in \R^{n\times n}_{\text{sym}}\\P^2=P}} v_P \geq -(1-\theta)\text{ in } B_{1/2}.
\]
This is effectively a Hölder modulus of continuity for $|(D^2u)_-|_{\text{op}}$ at the origin. Once again, since $u$ also satisfies $\mathcal M^-_{\l,\L}(D^2u)\leq 0$, we also get a Hölder modulus of continuity for $|(D^2u)_+|_{\text{op}}$ at the origin. All in all, we conclude the $C^{2,\a}$ estimate
\[
[D^2u]_{C^{0,\a}(B_{1/2})} \leq C\osc_{B_1} u.
\]
\end{exercise}

\section{Degenerate Problems}

A natural question one may ask is: When is it possible to extend the theory whenever uniform ellipticity is not available? This actually the case of some of some of the most important models in elliptic equations such as the minimal surface equation
\begin{align}
    \label{eq:min_sur}
\div\1\frac{Du}{\sqrt{1+|Du|^2}}\2 = 0,
\end{align}
the $p$-Laplace equation ($p\geq1$)
\[
\div(|Du|^{p-2}Du) = 0,
\]
the Monge-Ampère equation ($D^2u\geq 0$)
\[
\det(D^2u) = f,
\]
and in the dynamic case, the porous media equation ($u\geq 0$)
\[
\p_t u = u\D u + |Du|^2.
\]

Free boundary problems such as the obstacle problem
\[
\min\{\D u,\varphi-u\}= 0,
\]
are also considered as degenerate elliptic, given that the uniform ellipticity only has an effect outside of the \textit{contact region} $\{u=\varphi\}$. 

In all these cases, we observe that $F$ is still monotone increasing on the Hessian; however, its rate of growth may go to zero or infinity depending on other features of the solution. Nevertheless, in all of the previous models we find some additional mechanism which complements the uniform ellipticity in order to establish suitable regularity estimates.

The result we present in this final section considers the case where the uniform ellipticity holds if the gradient is sufficiently large.

\subsection{Elliptic Equations That Hold Only Where the Gradient is Large}

This degeneracy consideration splits the domain of the equation into two regions. One where the gradient is bounded and the solution is automatically Lipschitz, and another where the solution satisfies a uniformly elliptic equation which should enforce some regularity. The problem is that we do not have any control over the boundary between these two regions, where the regularity may degenerate.

The following theorem was originally proven in \cite{MR3500837}.

\begin{theorem}\label{thm:int_hold_est_deg}
Given $[\l,\L]\ss(0,\8)$ and $\gamma\geq 0$, there exist $\a\in(0,1)$ and $C\geq 1$, such that the following holds: Let $u\in C^1(B_1)\cap C^2(\{|Du|>\gamma\}\cap B_1)$, and $f\in C(B_1)$ non-negative such that
\[
\begin{cases}
\mathcal M^-_{\l,\L}(D^2u) \leq f \text{ in } \{|Du|>\gamma\} \cap B_1,\\
\mathcal M^+_{\l,\L}(D^2u) \geq -f \text{ in } \{|Du|>\gamma\} \cap B_1,
\end{cases}
\]
Then
\[
\sup_{r\in(0,1)}r^{-\a}\osc_{B_r} u \leq C\1\osc_{B_1} u+\|f\|_{L^n(\{|Du|>\gamma\} \cap B_1)}\2.
\]
\end{theorem}

The equation given in the previous theorem is invariant by dilations. When we consider $u_r(x) := r^{-\a}u(rx)$, $\gamma_r := r^{1-\a}\gamma$, and $f_r(x) := r^{2-\a}f(rx)$, we obtain that
\[
\begin{cases}
\mathcal M^-_{\l,\L}(D^2u_r) \leq f_r \text{ in } \{|Du_r|>\gamma_r\} \cap B_{1/r},\\
\mathcal M^+_{\l,\L}(D^2u_r) \geq -f_r \text{ in } \{|Du_r|>\gamma_r\} \cap B_{1/r},\\
\|f_r\|_{L^n(\{|Du_r|>\gamma_r\} \cap B_{1/r})} = r^{1-\a}\|f\|_{L^n(\{|Du|>\gamma\} \cap B_1)}.
\end{cases}
\]
For $\a,r\in(0,1)$ the dilations make the threshold $\gamma_r$ and the forcing term $f_r$ smaller than the original ones.

This means that the proof of this theorem ultimately relies on the proof of a measure estimate, as the one appearing in Corollary \ref{cor:mest}, which followed from Lemma \ref{lem:abp} and Lemma \ref{lem:loc}. The remaining steps towards the interior Hölder estimate can be proved with the same arguments we have already presented.

Notice as well that, by performing a sufficiently large dilation, we can assume that $\gamma$ is as small as needed. In Lemma \ref{lem:loc} we observe that the test functions have gradients universally bounded away from zero. Hence the exact same proof applies under a smallness condition on $\gamma$. On the other hand, the result analogous to Lemma \ref{lem:abp} requires a new idea. Let us present the proof given in \cite{MR3295593}.

\begin{lemma}\label{lem:abp_deg1}
    Given $[\l,\L]\ss(0,\8)$ there exist $\d,\theta,\eta\in(0,1)$ and $M\geq 1$, such that the following holds: Let $\gamma \in [0,1/12]$, $u\in C^1(B_{1})\cap C^2(\{|Du|>\gamma\}\cap B_{1})$ and $f\in L^n(B_1)$ with $\|f\|_{L^n(B_{1})}\leq \d$, both non-negative functions with
    \[
    \mathcal M^-_{\l,\L}(D^2u) \leq f \text{ in } \{|Du|>\gamma\}\cap B_{1}.
    \]
    Then
    \[
        u(0) \leq \theta \qquad \Rightarrow\qquad|\{u\leq M\} \cap B_{1}|\geq \eta.
    \]
\end{lemma}

As before, we use test functions of the form $\varphi_{y_0}^M(x) = -\frac{M}{2}|x-y_0|^2$ and their corresponding contact set will be denoted by
\[
A^M_B = \bigcup_{y_0\in B} \argmin(u-\varphi_{y_0}^M).
\]
Notice that for $\m\in(0,1)$ and $\theta=(1-\m)^2/2$, we find that the hypotheses given in Lemma \ref{lem:abp_deg1} imply that for some $c \in [0,\theta]$, $\varphi_0^1+c$ touches $u$ from below at some point in $\overline{B_{1-\m}}$. We first consider the scenario where the contact point falls in the region where the gradient is large and the uniform ellipticity is available.

\begin{lemma}\label{lem:abp_deg2}
    Given $[\l,\L]\ss(0,\8)$ and $\m\in(0,1/2)$ there exist $\d,\eta\in(0,1)$, such that the following holds: Let $\gamma \in[0,\m/2]$, $u\in C^1(B_1)\cap C^2(\{|Du|>\gamma\}\cap B_1)$ with $u\geq \varphi_0^1$, $f\in L^n(B_1)$ non-negative with $\|f\|_{L^n(B_1)}\leq \d$ such that $\mathcal M^-_{\l,\L}(D^2u)\leq f$. Then for $R:=\overline{B_{1-\m}}\sm B_{\m}$
    \[
        A_{\{0\}}^1\cap R \neq \emptyset  \qquad \Rightarrow\qquad|\{u\leq 1\} \cap B_1|\geq \eta.
    \]
\end{lemma}

The following proof is closely related with the idea behind the solution of measure estimate for the Hessian (\hyperlink{ex:w2p}{Problem 12} and its \hyperlink{sol:w2p}{solution}). See also Figure \ref{fig:abp7}.

\begin{proof}[Proof of Lemma \ref{lem:abp_deg2}]
Let $x_0 \in A^1_{\{0\}}\cap R$ and consider for $y_0 \in B_{\m/10}(x_0)$
\[
\varphi_0^1 + (\varphi_{y_0}^1+\m^2/200) = \varphi_{y_0/2}^2 + c_{y_0}.
\]
For $y_0 \in B_{\m/10}(x_0)$ we have $\varphi_{y_0}^{1}(x_0) + \m^2/200>0$ and $\varphi_{y_0}^{1} + \m^2/200<0$ in $\R^n\sm B_{\m/5}(x_0)$. Therefore
\[
\emptyset \neq \argmin(u-\varphi_{y_0/2}^{2}) \ss \{u\leq \m^2/200\} \cap B_{\m/5}(x_0)\ss B_1.
\]
For every $x_1 \in \argmin(u-\varphi_{y_0/2}^{2})$ we get that
\begin{align*}
|Du(x_1)| &= 2|x_1-y_0/2|\\
&\geq 2(|x_0-x_0/2|-|x_0-x_1|-|y_0/2-x_0/2|)\\
&\geq 2(|x_0|/2 - \mu/5 -\mu/20)\\
&> \mu/2\\
&\geq \gamma.
\end{align*}
Finally, by the same argument as in the proof of Lemma \ref{lem:abp}, using the uniform ellipticity at the contact points and the area formula, we conclude that
\[
|\{u\leq 1\}\cap B_1|\geq |A^2_{B_{\m/20}(x_0/2)}| \geq \eta.
\]
\end{proof}

Now we consider the alternative scenario where the contact point falls in the region where the gradient is small. The following result could be considered as a localization-type estimate and it does not require any assumption on the derivatives of the function.

\begin{lemma}\label{lem:abp_deg3}
    For $\m=1/6$, $r=1/12$, $u\in C(B_1)$ with $u\geq \varphi_0^1$, and $x_0 \in \p B_{r}$ it holds that
    \[
        A_{\{0\}}^1\cap \overline{B_\m}\neq\emptyset \qquad \Rightarrow\qquad
        \emptyset\neq A_{\{x_0\}}^2 \ss \{u\leq 1\} \cap \overline{B_{(1-\m)/2}(x_0)}.
    \]
\end{lemma}

\begin{proof}[Proof of Lemma \ref{lem:abp_deg3}]
    Let us see that the given choice of $\m$ and $r$ guarantees that
    \begin{align}
        \label{eq:5}
    m &:= \max\{(\varphi_0^1-\varphi_{x_0}^2)(x) \ | \ x\in \overline{B_\m}\}\\
    \notag &\leq \min\{(\varphi_0^1-\varphi_{x_0}^2)(x) \ | \ x\in \R^n\sm\overline{B_{(1-\m)/2}(x_0)}\}\\
    \notag &\leq 1.
    \end{align}
    Once this is done, we observe that the paraboloid $\varphi^2_{x_0}+m$ is above $\varphi_0^1$ over $\overline{B_\m}$ and below $\varphi_0^1$ outside $\overline{B_{(1-\m)/2}(x_0)}$. Hence, the hypotheses on $u$ imply that $\varphi_{x_0}^2+m$ must cross $u$ in $\overline{B_{(1-\m)/2}(x_0)}$. By lowering the paraboloid we find a contact point in $\overline{B_{(1-\m)/2}(x_0)}$. See Figure \ref{fig:abp_deg2}.

\begin{figure}
    \centering
    \includegraphics[width=0.8\textwidth]{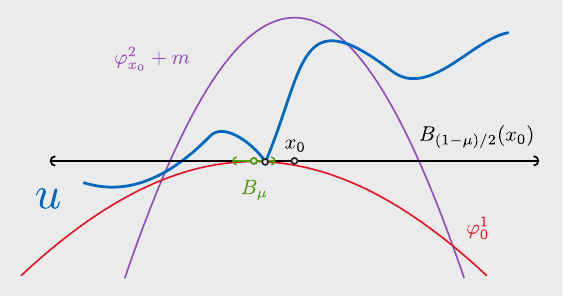}
    \caption{We arrange $\m$, $r$, and $m$ such that the graph of $u$ enters the region $\{(x,y) \in \R^n \times \R \ | \ \varphi_0^1(x) \leq y\leq  \varphi_{x_0}^2(x)+m\} \ss B_{(1-\m)/2}(x_0)\times (-\8,1]$.}
    \label{fig:abp_deg2}
\end{figure}

    We compute
    \begin{align*}
    \varphi_0^1(x)-\varphi_{x_0}^2(x) &= -\frac{1}{2}|x|^2 + |x-x_0|^2 = \frac{1}{2}|x|^2 - 2x\cdot x_0 + r^2 = \frac{1}{2}|x-2x_0|^2-r^2.
    \end{align*}
    
    The maximum of this convex paraboloid over $\overline{B_\m}$ is achieved at the furthest boundary point from $2x_0$
    \[
    \max\{(\varphi_0^1-\varphi_{x_0}^2)(x) \ | \ x\in \overline{B_\m}\} = \frac{1}{2}(\m+2r)^2-r^2.
    \]
    
    Given that $r \leq (1-\m)/2$, the vertex $2x_0$ falls in $\overline{B_{(1-\m)/2}(x_0)}$, then the minimum of $(\varphi_0^1-\varphi_{x_0}^2)$ over the complement of $\overline{B_{(1-\m)/2}(x_0)}$ is attained at the closest boundary point from $2x_0$. In particular it is less than 1
    \[
    \min\{(\varphi_0^1-\varphi_{x_0}^2)(x) \ | \ x\in \R^n\sm\overline{B_{(1-\m)/2}(x_0)}\} \leq \1\frac{1-\mu}{2}\2^2<1.
    \]
    Moreover,
    \[
    \min\{(\varphi_0^1-\varphi_{x_0}^2)(x) \ | \ x\in \R^n\sm\overline{B_{(1-\m)/2}(x_0)}\} = \frac{1}{2}\1\frac{1-\m}{2}-r\2^2-r^2.
    \]
    
    Finally we see that $\m+2r = 1/3 = \frac{1-\m}{2}-r$ guarantees \eqref{eq:5}.
\end{proof}

\begin{proof}[Proof of Lemma \ref{lem:abp_deg1}]
    We show instead the result over the domain $B_R$ with $R=3\sqrt{n}$, instead of $B_1$, so we assume $\|f\|_{L^n(B_R)}\leq \d$. As in Lemma \ref{lem:abp_deg3}, we fix $\mu=1/6$ and $r=1/12$.

    Let
    \[
    m_k(x_0) := \inf_{B_R} (u-\varphi_{x_0}^{2^k}), \qquad R_k(x_0):=\overline{B_{(1-\m)2^{-k} R}(x_0)}\sm B_{r2^{-k}R}(x_0), \qquad \theta := (1-\m)^2 R^2/2.
    \]
    From $u(0)\leq \theta$, we get $m_0(0) \leq \theta$ and $A_{\{0\}}^1\cap \overline{B_{(1-\m)R}} \neq \emptyset$.
    
    We select a set of dyadic cubes according to the following algorithm: At the $k^{th}$ stage of the algorithm, which starts from $k=0$, we consider a given subset of dyadic cubes from the $k^{th}$ generation of dyadic cubes starting from $Q_1$. One of these cubes, say $Q_{2^{-k}}(x_0)$, is selected if $A_{\{x_0\}}^{2^k}\cap R_k(x_0)\neq \emptyset$. Once a cube is selected, none of its descendants will be considered in the following stages. Otherwise, if the cube is not selected, its immediate descendants will be considered in the next stage.

    By an inductive argument using a rescaling of Lemma \ref{lem:abp_deg3}, we obtain that for every cube $Q_{2^{-k}}(x_0)$ that is considered in the $k^{th}$ iteration we also have that $A_{\{x_0\}}^{2^k}\cap \overline{B_{(1-\m)2^{-k}R}(x_0)} \neq \emptyset$.
    
    Also by induction and using Lemma \ref{lem:abp_deg3}, we see that if $Q_{2^{-(k+1)}}(x_1)$ is a cube considered in the $(k+1)^{th}$ iteration, descendant of $Q_{2^{-k}}(x_0)$, then
    \[
    m_{k+1}(x_1) \leq m_k(x_0) + 2^{-k}R^2 \leq \theta+2R^2.
    \]
    If $y_0\in Q_1$ is never covered by the selected cubes we find a nested sequence of dyadic cubes $Q_1 \supseteq Q_{1/2}(x_1) \supseteq \ldots$ such that $\bigcap_{k\geq 1}Q_{2^{-k}}(x_k) = \{y_0\}$ and
    \[
    \min\{(u-(\varphi_{x_k}^{2^k}+m_k(x_k)))(x) \ | \ x\in \overline{B_{\mu 2^{-k}R}(x_k)}\} = 0.
    \]
    This implies by the continuity of $u$ that $u(y_0)\leq \theta+2R^2$. If the measure of the set of points not covered is at least $1/2$, the result would now follow for any $\eta\leq 1/2$ and $M\geq \theta+2R^2$.
    
    On the other hand, using Lemma \ref{lem:abp_deg2} we see that if $Q_{2^{-k}}(x_0)$ is one of the selected cubes and now fix $M:=\theta+3R^2$, then
    \[
    |\{u\leq M\} \cap B_{2^{-k}R}(x_0)|\geq \eta_0|B_{2^{-k}R}(x_0)|.
    \]
    Assume that the measure of the union of all the selected cubes is at least 1/2 and let
    \[
    \mathcal B := \{B_{2^{-k}R}(x_0) \ | \ \text{$Q_{2^{-k}}(x_0)$ is one of the cubes selected by the algorithm}\}.
    \]
    By Vitali's covering lemma there exists a disjoint collection of balls $\mathcal B'\ss\mathcal B$ so that $\sum_{B\in \mathcal B'} |B| \geq 5^{-n}\left|\bigcup_{B\in \mathcal B}B\right| \geq 5^{-n}/2$. Given that for each $B\in \mathcal B'$ one has $|\{u\leq M\} \cap B|\geq \eta_0|B|$, we obtain that
    \[
    |\{u\leq M\} \cap B_R| \geq \sum_{B\in \mathcal B'} |\{u\leq M\} \cap B| \geq \eta := 5^{-n}\eta_0/2,
    \]
    which conbcludes the proof.
\end{proof}

\begin{exercise}[\hypertarget{ex:imbsil}{\hyperlink{sol:imbsil}{\textbf{Imbert-Silvestre's Approach.}}}]
    \item Assume the hypotheses of Lemma \ref{lem:abp_deg1} with $\theta=1$. Consider the family of test functions of the form
    \[
    \varphi_{y_0}(x) := -C_0|x-y_0|^{1/2},
    \]
    and the contact set
    \[
    A := \bigcup_{y_0 \in \{u > M\}\cap B_{1/4}} \argmin(u-\varphi_{y_0}).
    \]
    Show that for some appropriate choice of $C_0,M\geq 1$ and $\gamma_0,\d,\eta\in(0,1)$, we have $A \ss \{u \leq M\}$ with $|A|\geq \eta$.
\end{exercise}

\subsubsection{Further Developments}

The methods described in \cite{MR3500837, MR3295593} have proven to be remarkably versatile, extending their applicability to various equations. The cusp test functions introduced in \cite{MR3500837} were used by Silvestre and Schwab to extend the interior regularity estimates for parabolic integro-differential equations in \cite{MR3518535}. Recently, Pimentel, Santos, and Teixeira further advanced this idea to derive higher-order fractional estimates in \cite{MR4462186}. Moreover, in collaboration with Santos, we revisited the regularity theory for the porous medium equation, as presented in \cite{MR4557323}, by adapting Mooney's proof to a parabolic setting.

\subsection{Quasi-Harnack Inequality}

Degeneracy can also manifest itself across scales. For example, when modeling a PDE using finite-difference schemes, the continuous formulation of uniform ellipticity breaks down at the level of the discretization. However, if the numerical scheme approximates a uniformly elliptic equation, we expect that the discrete solution will approximate the continuous solution over large scales, inheriting with it the classical manifestations of uniform ellipticity \cite{MR1371593}.

A recent work by De Silva and Savin in \cite{MR4201786} proposes a weak notion of solution for equations where the uniform ellipticity manifests itself from a given scale onward. Their main result is a weak Harnack-type inequality. In this case, the alternative mechanism that compensates for the lack of uniform ellipticity is a measure estimate at microscopic scales. We discuss this result in more detail in the recent survey \cite{MR4658084}.

In addition to the applications already discussed to numerical schemes, the quasi-Harnack inequality can also be applied in the homogenization of elliptic problems with degeneracies, as studied in \cite{MR3265174}. It was also shown in \cite{MR4201786} that uniformly elliptic integro-differential equations of order $\s$ close to two fit also within the framework of the quasi-Harnack inequality. In this way, it provides a new proof for the Harnack inequality of Caffarelli and Silvestre \cite{MR2494809}.

A particular attractive feature of the quasi-Harnack inequality is that it does not require an actual partial differential equation. This is the idea behind the above mentioned applications. This versatility allowed to implement this approach to almost minimizers in calculus of variations in \cite{MR4126326}.

\subsection{Small Perturbation Solutions for Elliptic Equations}

The previous framework for degenerate equations does not apply in some fundamental equations coming from the calculus of variations, such as the minimal surface equation \eqref{eq:min_sur} or the $p$-Laplacian. In both of these cases, the ellipticity degenerates as the gradient grows.

Caffarelli introduced a perturbative method in \cite{MR1005611} to establish higher regularity estimates for solutions of uniformly elliptic equations, often referred to as \textit{regularity by compactness} or the \textit{improvement of flatness}. Here, ``flatness" means that a solution is assumed to be uniformly close to a prescribed profile. This strategy was inspired by De Giorgi's regularity theorem for minimal surfaces \cite{DeGiorgi1960}. The core idea is that if a solution is uniformly close to a smooth solution, it inherits the estimates of the corresponding linearization.

This approach quickly provided alternative proofs for regularity estimates in certain degenerate equations. Caffarelli and Córdoba applied it to the minimal surface equation in \cite{MR1190161}, while Wang explored estimates for the $p$-Laplace equation in \cite{MR1264526}.

In \cite{MR2334822}, Savin extended these estimates to include operators $F=F(M,p,z,x)$ that only need to be uniformly elliptic in the neighborhood of smooth solutions. Notably, this result enabled the treatment of equations that become degenerate as $(M,p,z)$ grows large, complementing the lack of uniform ellipticity whenever $(M,p,z)$ belongs to an unbounded regions. This approach allowed Savin to establish a long-standing conjecture of De Giorgi about the level sets of semi-linear equations arising from the Ginzburg-Landau energy in \cite{MR2480601}.

\subsubsection{Further Developments}

The groundbreaking work in \cite{MR2334822} has paved the way for numerous significant developments in the field. Although there are too many to comprehensively cover all of them here, some noteworthy results include the following research directions.

\noindent\textit{Free Boundary Problems:} De Silva developed the perturbative approach for the Bernoulli free boundary problem (also known as the One-Phase or Alt-Caffarelli problem), starting in \cite{MR2813524}. In this case, the free boundary condition linearizes to a zero Neumann type problem on the half-space, or equivalently the non-local problem $\D^{1/2}u=0$. A parabolic counterpart was studied in \cite{MR3983138}. A setup leading to the obstacle problem for $\D^{1/2}$ was studied in \cite{MR3916702} and used in \cite{MR4285137} to analyze the branching points in a two-phase free boundary problem.

\noindent\textit{Partial and Boundary Regularity for Uniformly Elliptic Equations:} Armstrong, Silvestre, and Smart applied this method to develop partial regularity results for fully non-linear equations in \cite{MR2928094}. One should keep in mind that, in general, uniformly elliptic problems may develop singularities as the one exhibited by \cite{MR3125267}. However, the interesting result proved in \cite{MR2928094}, shows that those singularities must have positive Hausdorff codimension (and hence zero measure). Silvestre and Sirakov proved in \cite{MR3246039} that the singular set can only be found in the interior of the domain of the equation.

\noindent\textit{Optimal Regularity for Degenerate Equations:} Colombo and Figalli developed regularity estimates for degenerate equations arising from traffic congestion models in \cite{MR3133426}. In collaboration with Pimentel, we demonstrated in \cite{MR4249793} the continuity of $|Du|$, where $u$ solves the gradient-constrained problem $\max\{1-|Du|,\Delta u+1\}=0$.

In \cite{MR3714836}, the paraboloid method is implemented to show a Harnack inequality for singular elliptic equations that are $(D^2u)^{-1}$-like. These arise as linearizations of the Monge-Ampère equation.

\noindent\textit{Non-Local Minimal Surfaces:} Caffarelli, Roquejoffre, and Savin established regularity estimates for non-local minimal surfaces in \cite{MR2675483}.

\noindent\textit{Parabolic and Non-Local Counterparts:} The technique was extended to parabolic problems by Wang in \cite{MR3158522}, and for non-local equation by Yu in \cite{MR3605294}.

\section{Open Problems}

One of the most famous open problems in the field is characterizing the dimension for which the equation $F(D^2u)=0$ has $C^2$ estimates.

Nirenberg, using complex analysis tools in \cite{MR0064986}, showed that $C^{2,\a}$ estimates are available in dimension 2. Recall that in \hyperlink{ex:hess_est_convx}{Problem 19} we saw that $C^{2,\a}$ estimates are available for convex/concave operators in any dimension. Some further generalizations have been studied in \cite{MR1995493,MR1793687}. However, it is known that in general we may find singularities starting in dimension $5$, \cite{MR3125267}.

\textbf{Do solutions of the uniformly elliptic equation $F(D^2u)=0$ have interior $C^2$ estimates or do singularities also develop in dimension 3 or 4?}

Whenever singularities develop for the equation $F(D^2u)=0$, it was shown in \cite{MR2928094} that they have positive Hausdorff codimension. This codimension is at least the exponent $\e$ in the measure estimate for the Hessian (\hyperlink{ex:w2p2}{Problem 14}), which yields the question of how large could $\e$ actually be. In the same article it was observed that in dimension 2, one must have $\e \leq C(\L/\l)^{-1}$ and conjectured that this is actually optimal.

In a recent approach to the Krylov-Safonov theory by Mooney \cite{MR3952774}, it has been shown that one can take $\e \geq C_n(\L/\l)^{1-n}$, settling the previous conjecture in dimension 2; see also \cite{MR4088812}. \textbf{It was conjectured in \cite{MR3952774} that this estimate is actually the optimal one in any dimension.}

One may also wonder about the optimal integrability exponent $p$ allowed as a forcing term in a uniformly elliptic equation in order to get interior Hölder estimates. As we know from the Sobolev and Morrey embeddings, this exponent should be strictly greater than $n/2$. For uniformly elliptic operators, we just saw that the theory holds for $p\geq n$. Due to fundamental estimates arising from harmonic analysis in \cite{MR0771392}, it was shown in \cite{MR1237053} that the optimal $p$ can always be taken slightly smaller in the interval $(n/2,n)$. \textbf{The actual dependence of the optimal integrability exponent $p$ on the dimension and the ellipticity constants remains open.}

A quite challenging question related to the integrability exponent for non-local equations of order $\s \in(0,2)$ is the following one: In contrast to the second-order theory, the non-local ABP estimate proved in \cite{MR2494809} depends on the $L^\8$-norm for the forcing term. \textbf{It is expected that a maximum principle also holds in terms of the $L^p$-norm for the forcing term, for some $p \in (n/(2\s),\8)$.} For the fractional Laplacians this is a consequence of the $L^p$ bounds for the Riesz transform and the Sobolev embedding. In the uniformly elliptic setting, the proof of this conjecture is only known under restrictive assumptions on the kernel \cite{MR2968592,MR4354731}. 

Similar questions can also be posed regarding the optimal Hölder exponent. So far, it seems that new techniques may be required in a dimension greater than or equal to 3. See \cite{MR4186265} for some very interesting conjectures in the case of the $p$-Laplacian.

The article \cite{MR3500837}, which studied Hölder estimates for elliptic equations that hold whenever the gradient is large (Theorem \ref{thm:int_hold_est_deg}), posed two interesting problems. The first was to extend the estimate for equations with drift term. This was solved in \cite{MR3295593}, whose strategy is presented in this note (in the absence of drift).

The second problem remains open: Is it possible to get a Hölder estimate on a solution of the parabolic problem
\[
\begin{cases}
    \p_t u - \mathcal M^-_{\l,\L}(D^2u) \geq -f \text{ in } \{|Du|>\gamma\}\cap B_1\times(-1,0],\\
    \p_t u - \mathcal M^+_{\l,\L}(D^2u) \leq f \text{ in } \{|Du|>\gamma\}\cap B_1\times(-1,0].
\end{cases}
\]
Notice that any $u=u(t)$, independent of $x$, is automatically a solution for any forcing term $f\geq 0$. This shows that we can not expect a regularity estimate in time. However, \textbf{we expect that a Hölder estimate exclusively in space holds true.}

Other reasonable extension in the static case would be to replace the constant $\gamma$ by some non-negative function $\gamma \in L^p$. For $p>n$, we have by Morrey's inequality that $|Du|\leq \gamma$ implies $u \in C^{0,1-n/p}$. \textbf{Would the Theorem \ref{thm:int_hold_est_deg} hold under the assumption $\gamma \in L^p$ for some $p>n$?}

\section{Ideas and References for the Problems}





\noindent\hypertarget{sol:harnack}{\hyperlink{ex:harnack}{\textbf{Harnack's inequality}}}

\noindent\textbf{1.} Use the mean value formula and the fact that for $r\in (0,1/3)$ and $x,y\in B_r$ one has $B_{1-3r}(y) \ss B_{1-r}(x) \ss B_1$.

\noindent\textbf{2.} Follows by using $u(0)$ instead of $\fint_{B_{1/3}}u$ in the proof of Lemma \ref{lem:dim_osc}. We also replace $B_{1/3}$ by $B_{1/4}$.

\noindent\hypertarget{sol:cvx}{\hyperlink{ex:cvx}{\textbf{Concave/Convex functions}}}

\noindent\textbf{1.} Consider a supporting plane for the graph of $u$ at $x_0 \in \argmin_{\overline{B_r}}u \ss \p B_r$.

\noindent\textbf{2.} The set $\{u \geq \sup_{B_r}u\}$ contains a ball of diameter $(1-r)$. We can then conclude by applying Markov's inequality.

\noindent\hypertarget{sol:weyl}{\hyperlink{ex:weyl}{\textbf{Weyl's lemma}}}

\noindent\textbf{1.} Use polar coordinates for the integration defining the convolution.

\noindent\textbf{2.} Using the Taylor expansion around zero for $x = \e\theta$ and $\theta \in B_1$
\[
u(x) - u(0) = \e Du(0)\cdot \theta + \frac{\e^2}{2} \theta\cdot D^2u(0)\theta + o(\e^2).
\]
The first term is odd and its integral over $B_\e$ vanishes. The last term is an error that vanishes faster than $\e^2$ as $\e\to 0$. Hence, we just keep the quadratic terms in the limit. Once again, we notice the odd symmetry of $\theta_i\theta_j$ for $i\neq j$ to keep only the quadratic terms of the form $\theta_i^2$ in the integration.

Using spherical symmetry once again
\[
\int_{B_1}\theta_i^2d\theta = \frac{1}{n} \int_{B_1} |\theta|^2 d\theta = \frac{|\p B_1|}{n}\int_0^1 r^{n+1} dr = \frac{|B_1|}{n+1}.
\]

Details can be consulted in \cite[Theorem 1.8 in Section 1.2]{MR2777537}.

\noindent\hypertarget{sol:max_prin}{\hyperlink{ex:max_prin}{\textbf{Maximum principle}}}

\noindent\textbf{1.} For $\e>0$, consider $\phi_\e \in C^\8(\R)$ with $0 \leq \phi_\e' \leq 1$, $\phi_\e''\geq0$, and
\[
\phi_\e(z) = \begin{cases}
    z-\e \text{ if } z\geq 2\e,\\
    0 \text{ if } z \leq 0.
\end{cases}
\]
Then we apply the mean value property to $v_\e := \phi_\e(u)$ and conclude after sending $\e\to0^+$. Notice that
\begin{align*}
(\D v_\e)_- &= \max\{-\phi_\e'(u) \D u - \phi_\e''(u) |Du|^2,0\} \leq \mathbbm 1_{\{u>0\}} (\D u)_-.
\end{align*}

\noindent\textbf{2.} Apply the previous result to $v(y) = u(x+2Ry)$ for every $x\in \W$.

\noindent\hypertarget{sol:bdry_reg}{\hyperlink{ex:bdry_reg}{\textbf{Boundary regularity}}}

\noindent\textbf{1.} Assume as usual that $\sup_{B_1} u\leq 1$ and $\|\D u\|_{L^p(\W)}\leq \d$ is small enough. From the estimate \eqref{eq:7} in the previous exercise we get that
\[
C\1\inf_{B_{1/3}}(1-u) +\d\2 \geq \int_{B_{1/3}}(1-u) \geq |\{u\leq 0\}\cap B_{1/3}| \geq \eta|B_{1/3}|.
\]
Therefore, we can choose $\theta = \d = \frac{|B_{1/3}|}{2C}\eta$ to get $\inf_{B_{1/3}}(1-u)\geq \theta$.

By iterating this estimate, we get the desired modulus of continuity from above. The one from below can be obtained by applying the same reasoning to $-u$.

\noindent\textbf{2.} Let $\a_0 \in (0,1)$ be the exponent from the interior Hölder estimate Lemma \ref{cor:dim_osc}, $\a_1 \in (0,1)$ the exponent from the previous part, and let $\a = \min\{\a_0,\a_1,2-n/p\}$.

Given two distinct points $x,y\in \W$ such that without loss of generality $\dist(y,\p\W)\leq \dist(x,\W) =: r$, we consider two cases:

If $y\in B_r(x)$, then by the Lemma \ref{cor:dim_osc}
\begin{align*}
\frac{|u(y)-u(x)|}{|y-x|^{\a}} &\leq r^{\a_0-\a}\frac{|u(y)-u(x)|}{|y-x|^{\a_0}}\leq C\1r^{-\a}\osc_{B_r(x)}u + R^{2-n/p-\a}\|\D u\|_{L^p(\W)}\2
\end{align*}
Let $x_0\in\p\W$ such that $|x-x_0|=\dist(x,\p\W)$. By the previous part and the maximum principle, we estimate the first term in the right-hand side for $r\leq 2R$
\[
r^{-\a}\osc_{B_r(x)}u \leq (2R)^{\a_1-\a}r^{-\a_1}\osc_{B_{2r}(x_0)}u \leq C\1\osc_{\W} u + \|\D u\|_{L^p(\W)}\2 \leq C\|\D u\|_{L^p(\W)}.
\]

If instead $|y-x|\geq r$, then we just have to apply the previous boundary estimate and the maximum principle twice
\[
\frac{|u(y)-u(x)|}{|y-x|^\a} \leq \frac{|u(y)|}{\dist(y,\p\W)^\a}+\frac{|u(x)|}{\dist(x,\p\W)^\a} \leq C\1\osc_{\W} u + \|\D u\|_{L^p(\W)}\2 \leq C\|\D u\|_{L^p(\W)}.
\]

\noindent\hypertarget{sol:lmp}{\hyperlink{ex:lmp}{\textbf{Local maximum principle}}}

\noindent\textbf{1.} Let $\theta>0$ to be fixed sufficiently small and assume by contradiction that $u < (1+\theta)u(0)$ in $B_1$. The idea is that $\|u_+\|_{L^\e(B_1)}\leq 1$ and the mean value property for $-u$ compete in opposite directions when we estimate the measure of $\{u\geq u(0)/2\}\cap B_{1/3}$.

Let $M\geq 1$ be sufficiently large such that by Markov's inequality
\begin{align}
    \label{eq:1}
    \frac{|\{u\geq u(0)/2\}\cap B_{1/3}|}{|B_{1/3}|}\leq \frac{(u(0)/2)^{-\e}}{|B_{1/3}|} \|u_+\|_{L^\e(B_1)}^\e \leq CM^{-\e} < \frac{1}{2}.
\end{align}

By a combination of the Markov inequality and the mean value property applied to $v = (1+\theta)u(0)-u > 0$
\begin{align*}
\frac{|\{u < u(0)/2\}\cap B_{1/3}|}{|B_{1/3}|} \leq \frac{|\{v>u(0)/2\}\cap B_{1/3}|}{|B_{1/3}|} \leq 2u(0)^{-1}\fint_{B_{1/3}}v.
\end{align*}
Applying the mean value formula to $v$ we get $\fint_{B_{1/3}}v \leq v(0) + C = \theta u(0) + C$ and obtain the estimate
\[
\frac{|\{u < u(0)/2\}\cap B_{1/3}|}{|B_{1/3}|} \leq 2\theta + CM^{-1}.
\]
Hence, by taking $M$ sufficiently large and $\theta$ sufficiently small, we enforce the right-hand side to be less than $1/2$, contradicting \eqref{eq:1}.

\noindent\textbf{2.} The idea proceeds once again by contradiction. Let $x_1 \in \argmax_{\overline{B_{1/2}}}u$ with $u(x_1)=M$ to be fixed sufficiently large. One can use the previous part to obtain a sequence of points $\{x_k\} \ss \overline{B_{3/4}}$ with where $u(x_k) \geq (1+\theta)u(x_{k-1}) \geq \ldots \geq (1+\theta)^kM \to \8$, contradicting the continuity of $u$. To prevent that the points escape from $\overline{B_{3/4}}$, one can also show that we can make this choice with $|x_{k+1}-x_{k}|$ controlled by a geometric vanishing sequence.

The same argument as in the previous part shows that for some $\s \geq 1$ sufficiently large, it holds that if $B_{3r}(x_0)\ss B_1$ with $r=\s u(x_0)^{-\e/n}$, then
\[
\begin{cases}
    \|u_+\|_{L^\e(B_1)} \leq 1,\\
    \|(\D u)_-\|_{L^p(B_1)} \leq 1,\\
    u(x_0)\geq M,
\end{cases} \qquad\Rightarrow\qquad \sup_{B_r(x_0)}u\geq (1+\theta)u(x_0).
\]
Indeed, we only need to modify the upper estimate on the measure of $\{u\geq u(x_0)/2\}\cap B_r(x_0)$ by
\[
\frac{|\{u\geq u(x_0)/2\}\cap B_r(x_0)|}{|B_r(x_0)|}\leq C\s^{-n} < \frac{1}{2}.
\]
The rest of the argument proceeds in the same way as before.

Back to the sequence, one gets that $u(x_k)\geq (1+\theta)^kM$ and
\[
|x_{k+1}-x_{k}| \leq \s u(x_k)^{-\e/n} \leq \s M^{-\e/n} (1+\theta)^{-k\e/n}.
\]
By choosing $M$ even larger, we can enforce that
\[
\s M^{-\e/n}\sum_{k=0}^\8 (1+\theta)^{-k\e/n} \leq \frac{1}{4}.
\]

Further details can be consulted in \cite[Lemma 4.7 and the proof of Lemma 4.4]{MR1351007}.

\noindent\hypertarget{sol:he}{\hyperlink{ex:he}{\textbf{Heat equation}}}

\noindent\textbf{1.} (Mean Value Property) The main idea is to consider $v=\eta u$, for some non-negative and compactly supported cut-off function $\eta \in C^\8_c(B_1\times(-1,0])$, so that $v$ would now satisfy the heat equation in $\R^n\times (-\8,0]$.

Given that
\[
\p_t v - \D v = \eta f + (\p_t \eta - \D \eta)u -2D\eta\cdot Du.
\]
Then, if $\eta(0,0)=1$,
\begin{align*}
v(0,0) = u(0,0) &= \iint (f\eta +(\p_t \eta - \D \eta)u -2D\eta\cdot Du)H_{-t}dxdt\\
&= \iint (f\eta H_{-t} +[(\p_t \eta + \D \eta)H_{-t}+2D\eta\cdot DH_{-t}]u)dxdt\\
&= \iint (f\eta H_{-t} +[(\p_t \eta + \D \eta) + t^{-1}x\cdot D\eta]H_{-t}u)dxdt.
\end{align*}
Hence we take
\[
\begin{cases}
    K_2 = \eta H_{-t},\\
    K_1 = [(\p_t \eta + \D \eta)+ t^{-1}x\cdot D\eta]H_{-t}.
\end{cases}
\]
We can easily see that $K_2\geq 0$, now we need to choose $\eta$ such that $(\p_t \eta + \D \eta)+ t^{-1}x\cdot D\eta\geq 0$. The function $\eta = (1-2|x|^2+4nt)_+$ works, but it is not smooth and the above computations cannot be applied directly. Instead, we may consider a modification of this candidate by approximating the positive part with a smooth function as was done, for example, in the \hyperlink{sol:max_prin}{solution of Problem 5}.

See \cite[Chapter 2.3]{MR2597943} for an alternative way to construct the mean value property for the heat equation.

\noindent\textbf{2.} (Weak Harnack) In the previous construction we have that for some $\r\in(0,1/2)$ sufficiently small, $\min_{\overline{B_{\r}}\times [-\r,-2\r]}K_1>0$.

Using that $K_2\leq H_{-t}$ and the scaling symmetry of $H_t$, we get that after two iterations of Hölder's inequality 
\[
\iint fK_2 \geq -C \int_{-1}^0 \|f_-(t)\|_{L^p(B_1)}t^{-n/(2p)}dt \geq -C\|f_-\|_{L^q_t((-1,0]\to L^p_x(B_1))}.
\]
Notice that the second iteration requires $2>n/p+2/q$ to integrate $t^{-nq'/(2p)}$ with $q'$ the conjugate exponent of $q$.

Then, putting together this information with the previous part
\[
\iint_{B_{\r}\times (-\r,-2\r]}u \leq C(u(0,0) + \|f_-\|_{L^q_t((-1,0]\to L^p_x(B_1))}).
\]
The conclusion can be finally obtained from this estimate by a geometric argument, similar to that illustrated in Figure \ref{fig:weak_har} and adapted to the parabolic setting.






\noindent\hypertarget{sol:mp}{\hyperlink{ex:mp}{\textbf{ABP maximum principle}}}

\begin{figure}
    \centering
    \includegraphics[width=0.8\textwidth]{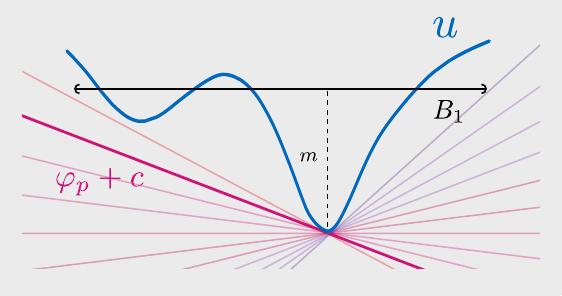}
    \caption{For the ABP maximum principle we consider the hyper-planes that support the graph of $u$ from below. Any plane with slope less than $m/2$ can be slid from below until it touches the graph of $u$.}
    \label{fig:abp5}
\end{figure}

\noindent\textbf{1.} The geometric idea is illustrated in Figure \ref{fig:abp5}. For $m := \max_{\overline{B_1}}u_-$ and $T=Du\colon A\to \R^n$ we get that $B_{m/2}\ss T(A)$. The Jacobian of $T$ for $x_0\in A$ can be estimated by using that if $e \in \operatorname{eig}(D^2u(x_0))$ then $0\leq e\leq \l^{-1}f(x_0)$. Therefore, $0\leq \det(DT)\leq Cf^n$ in $A$.

Details can be consulted in \cite[Chapter 3]{MR1351007}.

\noindent\textbf{2.} Let $d := \dist(x_0,\p\W)$ and $\theta\in\p B_1$ such that $d\theta \in \p B_{d}(x_0)\cap \p\W$. Consider the mapping as above using the set of slopes in the convex envelope of
\[
B_{|u(x_0)|/\diam(\W)} \cup \{\tfrac{|u(x_0)|}{d}\theta\}.
\]

Details can be consulted in \cite[Theorem 2.8]{MR3617963}.

\noindent\hypertarget{sol:envelope}{\hyperlink{ex:envelope}{\textbf{Inf/Sup-convolutions.}}}

\noindent\textbf{1.} Follows from $u-\varphi^{1/\e_2}_{y_0}\leq u-\varphi^{1/\e_1}_{y_0}$ and $\varphi_{y_0}^{1/\e}\leq 0$.

\noindent\textbf{2.} Let $M := \osc_{\overline{B_r(x_0)}}u <\8$, and $\w$ is a modulus of continuity for $u$ over $\overline{B_r(x_0)}$.

For $\r := \sqrt{2M\e} < r$ and $y_0 \in B_{r-\r}(x_0)$, the infimum that defines $u_\e$ is actually a minimum that is achieved over $\overline{B_\r(y_0)}$. Let
\[
x_0 \in \argmin\{(u-\varphi_{y_0}^{1/\e})(x)\ | \ x\in \overline{B_\r(y_0)}\}.
\]
Then
\[
0 \leq u(y_0)-u_\e(y_0)  = u(y_0)-u(x_0) - \frac{1}{2\e}|x_0-y_0|^2 \leq u(y_0) - u(x_0) \leq \w(\r) = \w(\sqrt{2M\e}).
\]

\noindent\textbf{3.} One can notice that $x \mapsto u_\e(x) + |x|^2/(2\e)$ is the infimum of affine functions. Alternatively one can show that every point in the graph of $u_\e$ has a supporting paraboloid of the form $\varphi_{x_0}^{-1/\e}+C$ from above, see Figure \ref{fig:dual_par}.

\begin{figure}
    \centering
    \includegraphics[width=0.8\textwidth]{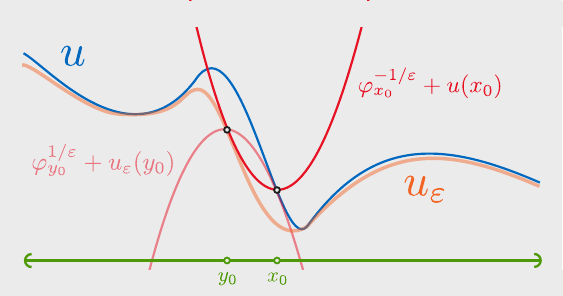}
    \caption{The paraboloid $\varphi_{x_0}^{-1/\e}+u(x_0)$ must be tangent to $u_\e$ from above. Otherwise we would be able to find at least one vertex of the form $(y_1,u_\e(y_1))$ strictly on top of the graph of $\varphi_{x_0}^{-1/\e}+u(x_0)$. However, by symmetry of the paraboloids, we observe that in such case the paraboloid $\varphi_{y_1}^{1/\e}+u_\e(y_1)$ must be larger than $u$ at $x_0$, which contradicts the definition of $u_\e(y_1)$.}
    \label{fig:dual_par}
\end{figure}

\noindent\textbf{4.} Follows from $\varphi_{y_0}^{\a_1}+\varphi_{y_0}^{\a_2} = \varphi_{y_0}^{\a_1+\a_2}$.

\noindent\textbf{5.} The largest $C\in\R$ for which $\varphi^{1/\e}_{y_0}+C\leq u$ is given by $u_\e(y_0)$.

\noindent\textbf{6.} $x\in \argmin(u-\varphi^{1/\e}_{y_0})$ if and only if there exists $C\in \R$ such that $\varphi^{1/\e}_{y_0}(x)+C = u(x)$, hence $\G_\e(x)=u(x)$.

\noindent\hypertarget{sol:abp_drift}{\hyperlink{ex:abp_drift}{\textbf{Measure estimates for operators with a drift term}}}

\noindent\textbf{1. and 2.} (Measure Estimate and Localization with Small Drift) Notice that in the proof of Lemma \ref{lem:abp} and Lemma \ref{lem:loc} we get that the contact set is contained in a set where $|Du|$ is universally bounded. Hence, the drift term can be absorbed once we assume that $\|b\|_{L^n(B_1)}$ is sufficiently small.

\noindent\textbf{3.} (ABP Maximum Principle) Let $m = \max_{\overline{B_1}}u_-$ and consider for $r\in(0,1/2]$, the contact sets
\[
A_r := \bigcup_{p \in B_{rm}\sm B_{rm/2}} \argmin(u-\varphi_p).
\]
Show then that for some $C>1$, independent of $r$,
\[
1 \leq C(\|b\|^n_{L^n(A_r)} + (rm)^{-n}\|f\|_{L^n(A_r)}^n).
\]
Add then the results for $r=1/2,1/4,\ldots,1/2^k$ such that $k = \lfloor2C\L^n\rfloor+1$, and use that $\|b\|_{L^n(B_1)} \leq \L$.

The proof in \cite[Theorem 2.21]{MR2777537} also relies on giving different weights to the contact sets $A_r$

\noindent\textbf{4.} (Localized Measure Estimate) The idea is to show that some $\overline{B_\e(x_0)}\ss B_{\r}$ such that $\|b\|_{L^n(B_\e(x_0))}$ is sufficiently small and one has that
    \[
    \inf_{B_{1/2}}u \leq 1 \qquad\Rightarrow\qquad \inf_{B_\e(x_0)} u \leq M.
    \]
After doing this, the results in Parts 1 and 2 of this problem conclude the argument.

Details can be consulted in \cite[Proposition 6.1]{MR3295593}.

\noindent\hypertarget{sol:par_eq1}{\hyperlink{ex:par_eq1}{\textbf{Measure estimates for parabolic equations}}}

\begin{figure}
    \centering
    \includegraphics[width=0.8\textwidth]{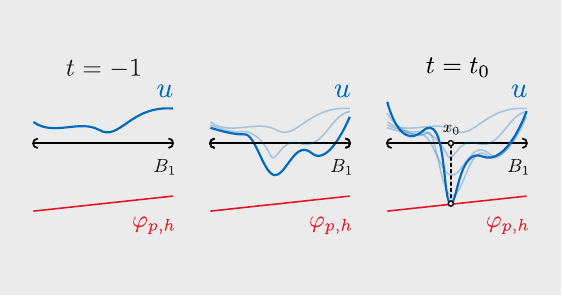}
    \caption{The test functions are planes of the form $x\mapsto p\cdot x-h$ that at time $t=-1$ are below $u$. We bring them form the past until they touch the graph of $u$ at some time $t_0$ and some point $x_0$.}
    \label{fig:tso}
\end{figure}

\noindent\textbf{1.} (ABP Maximum Principle) The test functions and contact set are illustrated in the Figure \ref{fig:tso}. Let $T\colon A\to \R^n\times \R$ such that $(p_0,h_0) = T(x_0,t_0)$ if and only if $u \geq \varphi_{p_0,h_0}$ in $B_1\times(-1,t_0]$, with equality at $(x_0,t_0)$. This means that $T(x,t) = (Du(x,t),Du(x,t)\cdot x-u(x,t))$.

For $m := \max_{\overline{B_1}\times[-1,0]}u_-$ we get that
\[
\{(p,h) \in \R^n\times \R \ | \ p \in B_{h/2}, h \in (0,m/2)\} \ss T(A).
\]
The Jacobian of $T$ is computed by
\[
\det(D_{x,t}T) = \det \begin{pmatrix}
    D^2u & D(\p_tu)\\
    x^TD^2u & D(\p_tu)\cdot x - \p_tu
\end{pmatrix} = -\p_t u\det(D^2u).
\]
For $(x_0,t_0) \in A$ and $e\in \operatorname{eig}(D^2u(x_0,t_0))$
\[
0 \leq \l e \leq \cM_{\l,\L}^-(D^2u(x_0,t_0)) \leq \p_t u(x_0,t_0)+f(x_0,t_0)\leq f(x_0,t_0).
\]
Therefore $0\leq \det(D_{x,t}T)\leq Cf^{n+1}$ over the contact set.

Details can be consulted in \cite{MR0790223}.

\noindent\textbf{2.} (Measure estimate) Let $\varphi_{y_0,s_0}(x,t) := (t-s_0) - 2|x-y_0|^2$ for $(y_0,s_0)\in \R^n\times \R$. Notice that $\varphi_{0,-1}(0,0)=1$ and $\varphi_{0,-1}<0$ in $\p_{\text{par}}(B_1\times(-1,0])$. For some $\r \in(0,1)$ sufficiently small and $(y_0,s_0)\in B:= B_\r\times [-1,-1+\r)$, we still have that $\varphi_{y_0,s_0}<0$ in $\p_{\text{par}}(B_1\times(-1,0])$ and $\theta := \min_{(y_0,s_0)\in B}\varphi_{y_0,s_0}(0,0)>0$.

If $u$ is non-negative and $u(0,0)\leq \theta$, then for every $(y_0,s_0)\in B$, the graph of $\varphi_{y_0,s_0}$ must eventually touch the graph of $u$ from below at some point in $B_1\times(-1,0]$. Consider in this way the contact set
\begin{align*}
A := \{(x_0,t_0) \in B_1\times(-1,0] \ | \ &\text{For some $(y_0,s_0) \in B$,}\\
&\text{$u > \varphi_{y_0,s_0}$ in $B_1\times(-1,t_0)$,}\\
&\text{and $u(x_0,t_0) = \varphi_{y_0,s_0}(x_0,t_0)$}\},
\end{align*}
and the surjective map $T\colon A\to B$ such that $(y_0,s_0) = T(x_0,t_0)$ if and only if $u> \varphi_{y_0,s_0}$ in $B_1\times(-1,t_0)$, with equality at $(x_0,t_0)$. This means that 
\[
T = \1x+\frac{1}{4}Du,t-u-\frac{1}{8}|Du|^2\2,
\]
and its Jacobian is
\[
\det(D_{x,t}T) = \11-\p_t u\2\det\1I+\frac{1}{4}D^2u\2.
\]
Using the uniform ellipticity we get that
\[
0 \leq \det(D_{x,t}T) \leq C(1+f^{n+1}).
\]

\begin{figure}
    \centering
    \includegraphics[width=0.8\textwidth]{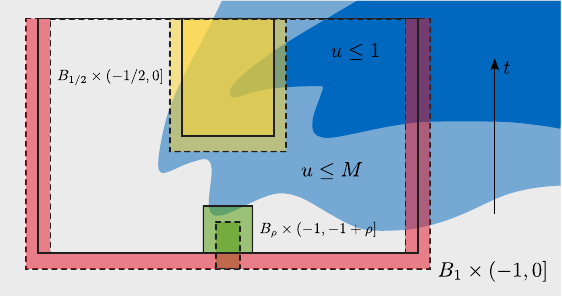}
    \caption{The function $\varphi$ is negative in (a parabolic neighborhood of) $\p_{\text{par}}(B_1\times (-1,0])$ (red region), positive in (a parabolic neighborhood of) $B_{1/2}\times(-1/2,0]$ (yellow region), and satisfies $\p_t \varphi - \mathcal M^-_{\l,\L}(D^2\varphi)< 0$ in (a parabolic neighborhood of) $B_1\times (-1,0] \sm B_{\r}\times(-1,-1+\r]$ (complement of the green region). We leave some room so that for every $(y_0,s_0) \in B_{\r/2}\times[0,\r/2)$, the translations $\varphi_{y_0,s_0}(x,t) := \varphi(x-y_0,t-s_0)$ keep similar properties.}
    \label{fig:par_loc}
\end{figure}

\noindent\textbf{3.} (Localization) The argument is an adaptation of the Lemma \ref{lem:loc}. It now requires to design a test function $\varphi=\varphi(x,t)\colon B_{1+\r/2}\times(-1-\r/2,0]\to \R$ such that
\[
\begin{cases}
    \varphi < 0 \text{ in } B_{1+\r/2}\times(-1-\r/2,0] \sm B_{1-\r/2}\times(-1,0],\\
    \varphi > 0 \text{ in } \overline{B_{1/2+\r/2}}\times[-1/2 - \r/2,0],\\   
    \p_t \varphi - \mathcal M^-_{\l,\L}(D^2\varphi) < 0 \text{ in } B_{1+\r/2}\times(-1-\r/2,0]\sm B_{\r/2}\times(-1-\r/2,-1+\r/2].
\end{cases}
\]
Notice that these characteristics resemble the backwards fundamental solution centered at $(0,-1)$ (Figure \ref{fig:par_loc}). This problem can be reduced to an ODE assuming that $\varphi$ is a truncation of a function of the form $(t+1)^{-\a}\psi(|x|^2/(t+1))$. See \cite[Lemma 4.16]{imbert:hal-00798300} for further details.

\noindent\textbf{4.} (Localized measure estimate) Follows by combining the previous parts.

\noindent\hypertarget{sol:w2p}{\hyperlink{ex:w2p}{\textbf{Measure estimates for the Hessian I}}}

\noindent\textbf{1.} We can prove it in two steps as in Lemma \ref{lem:abp} followed by Lemma \ref{lem:loc}.

Let $x_0 \in A^1_{\{0\}}\cap \overline{B_{1/2}}$, and consider for $y_0 \in B_{1/4}(x_0)$
\[
\varphi_0^1 + (\varphi_{y_0}^{1} + 1/32) = \varphi_{y_0/2}^{2} + c_{y_0}.
\]
For $y_0 \in B_{1/4}(x_0)$ we have that $\varphi_{y_0}^{1}(x_0) + 1/32>0$ and $\varphi_{y_0}^{1} + 1/32<0$ in $\R^n\sm B_{1/2}(x_0)$ (Figure \ref{fig:abp7}). Therefore
\[
\emptyset \neq \argmin(u-\varphi_{y_0/2}^{2}) \ss B_{1/2}(x_0).
\]
From this one can show that $|A^2_{\R^n}\cap B_1|\geq \eta$.

\begin{figure}
    \centering
    \includegraphics[width=0.8\textwidth]{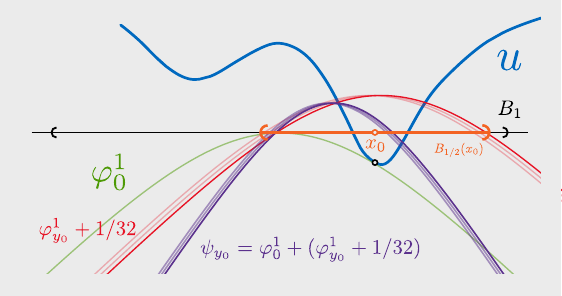}
    \caption{Assume without loss of generality that $\varphi_0^1$ touches $u$ from below at $x_0\in B_{1/2}$. For $y_0 \in B_{1/4}(x_0)$, the paraboloids $\varphi_{y_0}^{1} + 1/32$ are given such that $u$ must cross the sum $\psi_{y_0} := \varphi_0^1 + (\varphi_{y_0}^{1} + 1/32)$. Indeed $\psi_{y_0}(x_0) > \varphi_0^1(x_0) = u(x_0)$ and $\psi_{y_0} < \varphi_0^1 \leq u$ in $\R^n\sm B_{1/2}(x_0)$.}
    \label{fig:abp7}
\end{figure}

The localization lemma can be shown modifying the test function from Lemma \ref{lem:loc}. instead of truncating it with a constant we can cap it with a paraboloid that we fit in a $C^{1,1}$ fashion.  

See \cite[Lemma 7.6]{MR1351007} for an alternative proof using the ABP maximum principle.

\noindent\hypertarget{sol:par_eq2}{\hyperlink{ex:par_eq2}{\textbf{Covering lemmas for parabolic equations}}}

\noindent\textbf{1.} (Staking lemma) We can reduce the problem to the case of zero spatial dimension by using Fubini's Theorem. Consider then $\mathcal I$ to be a finite set of disjoint open intervals in $(-1,0)$. For every $I = (t-r,t)$ and $m\geq 1$, let $I^m := (t,t+mr)$. Then we have to show that
\[
    \left|\bigcup_{I \in \mathcal I} I^m\right| \geq \frac{m}{m+1}\left|\bigcup_{I\in \mathcal I} (I^m\cup I)\right|.
\]

Let $T$ be a mapping between intervals such that for $I=(a,b)$,
\[
T(I):=(a-(b-a)/m,b).
\]
In particular, $T(I^m)=I^m\cup I$. Extend $T$ to unions of disjoint open intervals $E = \bigcup I$ such that $T(E) = \bigcup T(I)$. We still have that the length of $T(E)$ is less than or equal to $(m+1)/m$ times the length of $E$.

On the other hand, any open set of $\R$ can be decomposed in a unique way as a numerable union of disjoint open intervals. If the open set $E = I_1\cup\ldots\cup I_N$ is given as a finite union of non-necessarily disjoint intervals, we get by an inductive argument that $T(E) \supseteq T(I_1)\cup\ldots\cup T(I_N)$.

Recover the desired estimate by applying the previous observations to $E = \bigcup_{I \in \mathcal I} I^m$
\[
\frac{m+1}{m}\left|\bigcup_{I \in \mathcal I} I^m\right| \geq \left|T\1\bigcup_{I \in \mathcal I} I^m\2\right| \geq \left|\bigcup_{I \in \mathcal I} T(I^m)\right| = \left|\bigcup_{I \in \mathcal I} (I^m\cup I)\right|.
\]

\noindent\textbf{2.} (Construction for the parabolic growing ink-spot lemma) The idea is illustrated in the Figure \ref{fig:par_ink}.

\begin{figure}
    \centering
    \includegraphics[width=0.8\textwidth]{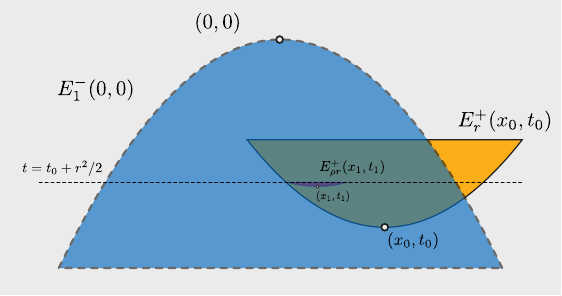}
    \caption{The slice of $E^+_r(x_0,t_0)$ at $t=t_0+r^2/2$ is the ball $B_{r_1}(x_0)$ with $r_1=r/\sqrt{2}$. Meanwhile by convexity, the segment between $(0,0)$ and $(x_0,t_0)$ meets $t=t_0+r^2/2$ in the smaller ball $B_{r_0}(x_0)$ with $r_0=r/2$. Let $\theta := -x_0/|x_0|$ if $x_0\neq 0$, otherwise let $\theta =0$. For $x_1 = \frac{r_0+r_1}{2}\theta$, $\r=\frac{1/\sqrt{2}+1/2}{2}$, and $t_1 = t_0+r^2/2-\r^2r^2$ it can be shown that $E^+_{\r r}(x_1,t_1) \ss E^+_r(x_0,t_0)\cap E^-_1(0,0)$.}
    \label{fig:par_ink}
\end{figure}

\noindent\hypertarget{sol:w2p2}{\hyperlink{ex:w2p2}{\textbf{Measure estimates for the Hessian II}}}

\noindent\textbf{1.} The set $\{m(f^n) > \d^n\}$ is considered because we need to have that the $L^n$-norm of the operator remains small after a rescaling of the form $u_r(x) = r^{-2}u(rx+x_0)$. In fact, if $x_1 \in \{m(f^n)\leq \d\}\cap Q_r(x_0)$, then $f_r(x) := (\cM^-_{\l,\L}(D^2u_r(x)))_+ = f(rx+x_0)$ satisfies
\[
\fint_{B_{3\sqrt n}} |f_r|^n =  \fint_{Q_{3r\sqrt n}(x_0)} |f|^n \leq C\fint_{Q_{7r\sqrt{n}/2}(x_1)} |f|^n \leq C\d^n.
\]

After this observation, the idea is to use the result from \hyperlink{ex:w2p}{Problem 12}.

\noindent\textbf{2.} Assume without loss of generality that $\|u\|_{L^\8(B_{3\sqrt n})}\leq 1$ and $\|f\|_{L^n(B_{3\sqrt n})}\leq \d$. The previous part can be used to show the following diminish of the distribution
\begin{align*}
|Q_1\sm A^{M^{k+1}}_{\R^n}| &\leq (1-\eta)|(Q_1\sm A^{M^{k+1}}_{\R^n})\cup \{m(f^n) > (\d M^k)^n\}|\\
&\leq (1-\eta)^{k+1} + \sum_{i=1}^{k} (1-\eta)^{}|\{m(f^n) > (\d M^{i})^n\}|
\end{align*}
Finally use the Hardy–Littlewood maximal inequality to bound
\[
|\{m(f^n) > (\d M^i)^n\}| \leq \frac{C}{(\d M^i)^n}\d^n = CM^{-in}.
\]

Details can be consulted in \cite[Lemma 7.7 and Lemma 7.8]{MR1351007}.

\noindent\hypertarget{sol:fl}{\hyperlink{ex:fl}{\textbf{Fractional Laplacian}}}

\noindent\textbf{1.} (Comparison principle) Follows from the facts that $u\leq v$ with $u(x_0)=v(x_0)$ implies
\[
u(y+x_0)+u(y-x_0)-2u(x_0) \leq v(y+x_0)+v(y-x_0)-2v(x_0).
\]

\noindent\textbf{2.} (Weak Harnack) Let $\b \in C^\8_c(B_{5/8})$ be a radial bump function taking values between $0$ and $1$, and such that $\b = 1$ in $B_{1/2}$. Assume by contradiction that for some $\theta\in(0,1)$, the graph of $\theta\b$ touches the graph of $u$ from below at $x_0 \in B_{5/8}$. Consider now the function $v = \theta\b + \mathbbm 1_{B_1\sm B_{3/4}} u$, see Figure \ref{fig:frac_lap}.

\begin{figure}
    \centering
    \includegraphics[width=0.8\textwidth]{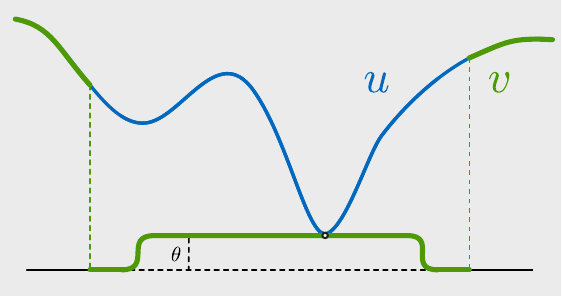}
    \caption{The function $v$ is defined piece-wise as a small bump that touches $u$ from below in $B_{5/8}$, and it is equal to $u$ outside $B_{3/4}$. Then the comparison principle between $u$ and $v$ leads to the desired measure estimate.}
    \label{fig:frac_lap}
\end{figure}

Then from the comparison principle we get
\[
1 \geq \D^{\s/2} u(x_0) \geq \D^{\s/2} v(x_0) \geq \theta\D^{\s/2} \b(x_0) + C\int_{B_1\sm B_{3/4}}u \geq -C + C\int_{B_1\sm B_{3/4}}u.
\]
This provides a contradiction once we assume that $\int_{B_1\sm B_{3/4}}u$ is sufficiently large.

\noindent\textbf{3.} (Hölder estimate) The challenge in the implementation of the diminish of oscillation is the non-local character of the equation. Even if we assume that $u$ takes values between $0$ and $1$ over $B_1$, the weak Harnack estimate we showed requires $u$ to be positive throughout $\R^n$.

First, we can truncate the tail of the solution by considering instead $v = \mathbbm 1_{B_1}u$, this introduces an additional forcing term to the right-hand side, bounded by a multiple of $\pm \int_{\R^n} |u|/(1+|y|^{n+\s})dy$.

We assume in an inductive way and without loss of generality that $u$ satisfies for some small $\a,\d\in (0,1)$
\[
\begin{cases}
    |u|\leq (1+\d)|x|^\a \text{ in } \R^n\sm B_1,\\
    |u| \leq 1 \text{ in } B_1,\\
    \|(-\D)^\s u\|_{L^\8(B_1)} \leq \d.
\end{cases}
\]
Then we consider whether $u$ is positive or negative at least half of the measure in $B_1\sm B_{3/4}$.

If we have $|\{u\geq 0\}\cap(B_1\sm B_{3/4})|\geq |B_1\sm B_{3/4}|/2$, then we translate and truncate the function as $v = (u+1)_+$. For $x\in B_{3/4}$, we have that $\D^{\s/2}v(x)$ is bounded by $\D^{\s/2}u(x)$ and a small term
\[
\D^{\s/2}v(x) \leq \D^{\s/2}u(x) + C\int_{1/4}^\8 \frac{(1+\d)\r^\a-1}{\r^{1+\s}}d\r.
\]
The second term can be made as small as needed if $\a$ and $\d$ are sufficiently small. Then, by the weak Harnack estimate, there is $\theta\in(0,1)$ such that $u+1 = v\geq \theta>0$ in $B_{1/2}$.

We need to focus on the relation between the constants $\a$, $\d$, and $\theta$, so that after the rescaling, the inductive step holds. Assume that we improved the lower bound as in the previous paragraph. Then we consider $v(x) = (u(x/2)-\theta/2)/(1-\theta/2)$ so that $|v|\leq 1$ in $B_1$. To have $|v|\leq (1+\d)|x|^\a$ in $B_1\sm B_2$, we need $(2-\theta-2^{1-\a})2^{\a+1}\geq \theta(2-\theta)$. This can be achieved by choosing $\theta$ sufficiently small in terms of $\a$.

Details can be consulted in \cite{MR2244602}.

\noindent\hypertarget{sol:int_grad_est}{\hyperlink{ex:int_grad_est}{\textbf{Interior gradient estimates}}}

\noindent\textbf{1.} Follows as in the proof of Corollary \ref{cor:high_ord_est}.

\noindent\hypertarget{sol:lmp2}{\hyperlink{ex:lmp2}{\textbf{Local maximum principle - uniformly elliptic}}}

\noindent\textbf{1.} Follows as in the \hyperlink{sol:lmp}{solution} of \hyperlink{ex:lmp}{Problem 6}.

\noindent\hypertarget{sol:harnack2}{\hyperlink{ex:harnack2}{\textbf{Harnack Inequality - uniformly elliptic}}}

\noindent\textbf{1.} Follows by combining the weak Harnack inequality and the local maximum principle. Details can be consulted in \cite[Section 4.2]{MR1351007}

\noindent\hypertarget{sol:hess_est_convx}{\hyperlink{ex:hess_est_convx}{\textbf{Interior Hessian estimates for convex operators}}}

\noindent\textbf{1.} $v := -v_2$ satisfies $\mathcal M^+_{\l,\L}(D^2v) \geq 0$, then by the local maximum principle we just need to provide a bound for $\|v\|_{L^\e(B_1)}$. This follows by the measure estimate for the Hessian (\hyperlink{ex:w2p2}{Problem 15}) with a slightly larger exponent.

\noindent\textbf{2.} Notice that $v_A$ can be computed as the following limit
\[
v_A(x) = \lim_{\e\to0} \frac{C}{\e^2}\fint_{B_\e}(u(x+Ay)-u(x))dy.
\]
Given that $u$ satisfies a convex equation,we can obtain an equation for the average $\fint_{B_\e}u(x+Ay)dy$.

\noindent\textbf{3.} Assume by contradiction that for some projection $P_0$, we have $\inf_{B_{1/2}} v_{P_0}\leq -(1-\theta)$ for some $\theta\in(0,1)$ to be fixed sufficiently small.

Given that $v_{P_0}\geq -1$ satisfies $\mathcal M_{\l,\L}^-(D^2v_{P_0})\leq 0$, we obtain by the weak Harnack inequality that $|\{v_{P_0}>-1+\theta^{1/2}\}\cap B_{1/4}|\leq C\theta^{\e/2}$.

In the complementary set $E:= \{v_{P_0}\leq -1+\theta^{1/2}\}\cap B_{1/4}$ we have
\[
0 \leq \underbrace{\max_{\substack{P \in \R^{n\times n}_{\text{sym}}\\P^2=P}} v_P}_{v_+} - v_{I-P_0} = v_{P_0}-\underbrace{\min_{\substack{P \in \R^{n\times n}_{\text{sym}}\\P^2=P}} v_P}_{-v_-} \leq \theta^{1/2}.
\]
By uniform ellipticity and $F(0)=0$,
\[
0\leq \mathcal M^+_{\l,\L}(D^2u) = \L v_+ - \l v_-.
\]
Then, for $\theta$ sufficiently small, we get that in $E$
\[
v_{I-P_0}\geq v_+ - \theta^{1/2} \geq \frac{\l}{\L}v_- - \theta^{1/2} \geq \frac{\l}{\L}(1-\theta^{1/2}) - \theta^{1/2} \geq \frac{\l}{2\L}.
\]

However, $w := \l/(2\L) - v_{I-P_0}$ satisfies $\mathcal M^+_{\l,\L}(D^2w)\geq0$, and then by the local maximum principle we obtain the following contradiction (for $\theta$ chosen even smaller)
\[
\frac{\l}{2\L} = w(0) \leq C\|w_+\|_{L^1(B_{1/4})} \leq C \theta^{\e/2}.
\]

Details can be consulted in \cite{MR2550191}.

\noindent\hypertarget{sol:imbsil}{\hyperlink{ex:imbsil}{\textbf{Imbert-Silvestre's approach}}}

\noindent\textbf{1.} Details can be consulted in \cite[Proposition 3.3]{MR3500837}.

\bibliographystyle{alpha}
\bibliography{mybibliography}

\end{document}